\documentclass{amsart}

\usepackage{amsfonts,amsmath,hyperref,amsthm,bbold,subcaption}
\usepackage{amssymb}
\usepackage{mathabx}

\usepackage{color}
\usepackage{enumitem}

\usepackage{tikz-cd}

\usepackage{graphicx}

\newtheorem{maintheorem}{Theorem}

\usepackage{overpic}

\theoremstyle{plain}
        \newtheorem{theorem}{Theorem}[section]
        \newtheorem*{theorem*}{Theorem}
        \newtheorem*{maintheorem*}{Main Theorem}
        \newtheorem*{conj*}{Conjecture}
        \newtheorem{lemma}[theorem]{Lemma}
        \newtheorem{corollary}[theorem]{Corollary}
        \newtheorem{proposition}[theorem]{Proposition}

\theoremstyle{definition}
        \newtheorem{definition}[theorem]{Definition}
        \newtheorem*{definition*}{Definition}

\theoremstyle{remark}

        \newtheorem{rem}[theorem]{Remark}

        \newtheorem*{claim}{Claim}

\renewcommand{\AA}{{\mathcal A}}

\newcommand{\blowup}{{\operatorname{blow}}}

\newcommand{\filled}{{\mathcal K}}

\newcommand{\id}{\operatorname{id}}

\newcommand{\TT}{{\mathcal T}}

\newcommand{\N}{{\mathbb N}}
\newcommand{\Z}{{\mathbb Z}}

\newcommand{\C}{{\mathbb C}}

\newcommand{\disk}{{\mathbb D}}

\renewcommand{\phi}{\varphi}

\newcommand{\hide}[1]{}

\newcommand{\comment}[1]{\marginpar{#1}}
\renewcommand{\comment}[1]{}

\renewcommand{\sp}{{\ \ }}

\newcommand{\per}{{\operatorname{per}}}
\newcommand{\Fat}{{\mathcal F}}
\newcommand{\Jul}{{\mathcal J}}
\newcommand{\UU}{{\mathcal U}}
\newcommand{\orb}{{\operatorname{orb}}}
\newcommand{\diam}{{\operatorname{diam}}}

\newcommand{\GC}{\mathcal{G}}

\newcommand{\KC}{{\mathcal K}}
\newcommand{\post}{\operatorname{post}}

\newcommand{\CC}{{\mathcal C}}

\newcommand{\DD}{{\mathcal D}}
\newcommand{\NL}{{\mathfrak N}}

\newcommand{\PC}{{\mathcal P}}

\newcommand{\MC}{{\operatorname{MultiCurve}}}
\newcommand{\Curve}{{\operatorname{Curve}}}
\newcommand{\psMC}{{\operatorname{psMultiCurve}}}
\newcommand{\psCurve}{{\operatorname{psCurve}}}

\newcommand{\bi}{{\operatorname{bi}}}

\newcommand{\wC}{{\widehat \C}}
\newcommand{\intr}{{\operatorname{int}}}

\newcommand\lift[2]{{\uparrow}_{#1}^{#2}}
\newcommand\selfmap{\righttoleftarrow}

 \newcommand{\cro}{{\operatorname{dec}}}
  \newcommand{\Th}{{\operatorname{Th}}}
  \newcommand{\Sie}{{\operatorname{Sie}}}

  \newcommand{\mesh}{{\operatorname{mesh}}}

\title[A canonical decomposition of postcritically finite rational maps]{A canonical decomposition of \\ postcritically finite rational maps\\ and their maximal expanding quotients}

\author{Dzmitry Dudko}
\address {Mathematics Department, Stony Brook University, NY 11794, USA}
\email{dzmitry.dudko@stonybrook.edu}

\author{Mikhail Hlushchanka}
\address{Mathematisch Instituut, Universiteit Utrecht,
 3508 TA Utrecht, The Netherlands}
\email{m.hlushchanka@uu.nl}

\author{Dierk Schleicher}
\address {Aix--Marseille Universit\'{e} and CNRS, UMR 7373, Institut de Math\'{e}atiques de Marseille, 163 Avenue de Luminy Case 901, 13009 Marseille, France}
\email{dierk.schleicher@univ-amu.fr}

\thanks{This project was supported by the Advanced Grant ``HOLOGRAM'' of the European Research Council.}

\begin{document}

\begin{abstract}

We provide a natural canonical decomposition of postcritically finite rational maps with non-empty Fatou sets based on the topological structure of their Julia sets. The building blocks of this decomposition are maps where all Fatou components are Jordan disks with disjoint closures (\emph{Sierpi\'{n}ski maps}), as well as those where any two Fatou components can be connected through a countable chain of Fatou components with common boundary points (\emph{crochet} or \emph{Newton-like maps}). 

We provide several alternative characterizations for our decomposition, as well as an algorithm for its effective computation. We also show that postcritically finite rational maps have dynamically natural quotients in which all crochet maps are collapsed to points, while all Sierpi\'{n}ski maps become small spheres; the quotient is a \emph{maximal expanding cactoid}. The constructions work in the more general setup of B\"{o}ttcher expanding maps, which are metric models of postcritically finite rational maps.

\end{abstract}



\maketitle
\setcounter{tocdepth}{1}

\tableofcontents

\section{Introduction}
The dynamics of a rational map is controlled in a very strong sense by its critical orbits. If all the critical orbits are finite, then the map is called \emph{postcritically finite} (PCF). PCF maps are like rational points of the parameter space and have been in the focus of intense research in holomorphic dynamics.

It is often convenient to abstract from the complex structure and consider rational maps as topological branched coverings.  In the PCF setting we naturally obtain a branched self-covering $f\colon (S^2,A)\selfmap$ on the marked sphere $(S^2, A)$, where $A\subset S^2$ is a finite invariant set containing all the critical values of $f$.  
Topological maps are more amenable to classifications and applicable to many surgeries, such as decompositions and amalgams. 

The \emph{Thurston fundamental theorem of complex dynamics} characterizes those PCF branched coverings of $S^2$ that are ``realized'' by rational maps \cite{DH_Th_char}. Roughly speaking, it says that $f\colon (S^2,A)\selfmap$ is isotopic to a rational map if and only if $f$ does not admit a collection of disjoint annuli violating the Gr\"otzsch inequality.  A \emph{Thurston obstruction} is an invariant multicurve composed of the core curves of such annuli.  The proof of Thurston's theorem is based on a fixed-point argument.  The map $f\colon (S^2,A)\selfmap$ naturally defines a pullback map $\sigma_f\colon \TT_A\selfmap$ on the Teichm\"uller space of $(S^2,A)$ (the space of complex structures). It follows that $f$ is isotopic to a rational map if and only if $\sigma_f$ has a fixed point.  Somewhat similar ideas were used by Thurston in his work on geometry of $3$-manifolds and surface homeomorphisms.

Treating rational maps as topological closely links complex dynamics to the theory of mapping class groups. Just like homeomorphisms,  PCF branched coverings of $S^2$ can be decomposed by cutting the sphere along invariant multicurves. The general decomposition theory was developed by Pilgrim  \cite{Pilgrim_Comb}, who also introduced the first \emph{canonical decomposition} of $f\colon(S^2,A)\selfmap$ along the Thurston obstruction consisting of curves that get shorter  under iteration of the pullback map $\sigma_f$.  These curves cut the sphere into ``small'' maps of three types  \cite{SelingerPullback}: homeomorphisms, double covers of torus endomorphisms, rational maps.  Recently, the \emph{canonical Levy decomposition} was introduced in \cite{BD_Exp} as the smallest Levy multicurve such that all small maps in the decomposition are either homeomorphisms or Levy-free maps with $\deg >1$.  It follows from \cite{SelingerPullback, SY_Decid} that the canonical Levy decomposition is a subdecomposition of the Pilgrim canonical decomposition.


Mapping class groups naturally appear in the conjugacy problem (also known as the Thurston equivalence) between branched coverings of the sphere. As it is shown in the Bartholdi-Nekrashevych solution of the Hubbard twisting rabbit problem \cite{BarNekr_Twist},  the main difficulty is to understand how homeomorphisms interact with branched coverings.   
Computational theory developed in~\cite{BD_Dec} allows to effectively reduce the original conjugacy problem to the conjugacy and centralizer problems between small maps of some canonical decomposition.  In particular, the conjugacy problem for PCF branched coverings of the sphere  is decidable \cite{BD_Algo}.  However, the current methods of finding canonical decompositions are not effective and do not have any complexity estimates.

\subsection{Invariants in complex dynamics} There are many invariants characterizing different classes of PCF rational maps. Perhaps the most well-known is the \emph{Hubbard tree} of a polynomial. It is a finite planar tree in the core of the filled-in Julia set \cite{DH_Orsay}. Two polynomials are conjugate if and only if their Hubbard trees are planar conjugate.  This allows to provide a \emph{combinatorial classification} of all PCF polynomial maps \cite{BFH_Class,Poirier}.  A complex polynomial can also be described using invariant (or periodic in the subhyperbolic case) spiders \cite{HS_Spider}. This approach leads to the Poirier notion of supporting rays \cite{Poirier}. Quadratic polynomials can also be described using kneading sequences and internals addresses. All these invariants can be converted into each other.

Given two polynomials, we can topologically glue them together along the circle at infinity and obtain a branched covering of the sphere, called the \emph{formal mating} of the polynomials \cite{Douady_Mating,Meyer_Mating}.  If the formal mating is equivalent to a rational map, then the resulting rational map is conjugate to the \emph{geometric mating} -- the gluing of the polynomial filled-in Julia sets along their boundaries parameterized by external angles. The mating operation attracts a lot of attention \cite{MatingQuestions} and is well-understood in the quadratic case due to M. Rees, Tan Lei, and M. ~Shishikura \cite{Rees_1,TanLeiMatings,Shishikura_Rees}. In general, the mating does not respect the structures of polynomial Julia sets, and rational maps often can be unmated in many ways \cite{M_unmating}.  However, if one of the polynomials is sufficiently simple (for example,  it is the Basilica or a generalized Rabbit), then its Hubbard tree gives a useful invariant of the mating with strong parameter implications.

Newton maps arise in the root-finding problem and form the biggest  well-understood non-polynomial class of rational maps. A Newton map  has a unique repelling fixed point at infinity; every other fixed point is attracting. The immediate attracting basins of fixed points meet at infinity; the associated internal rays form the \emph{channel diagram} that is the beginning of the puzzle theory for Newton maps \cite{ClassNewtonFixed, NewtonRigidity}.  The preimages of the channel diagram together with embedded Hubbard trees (for the renormalizable parts of Newton dynamics) provide the basis of the combinatorial classification of PCF Newton maps \cite{RussellDierk_Class}.

The Dehn-Nielsen-Baer Theorem states that every homeomorphism $f\colon (S^2,A)\selfmap$ is uniquely characterized up to isotopy by the induced pushforward $f_*\colon\pi_1(S^2,A)\selfmap$ viewed as an outer automorphism of $\pi_1(S^2,A)$. The theorem was extended to non-invertible branched coverings by Kameyama \cite{KameyamaThEq} and independently by Nekrashevych~\cite{Nekra}. The group theoretical data arising from $f\colon (S^2,A)\selfmap$ is conveniently described by the $\pi_1(S^2,A)$-\emph{biset} of $f$. There is a bijection between branched coverings $f\colon (S^2,A)\selfmap$ considered up to isotopy rel $A$ and sphere $\pi_1(S^2,A)$-bisets considered up to biset-isomorphism, see ~\cite[Theorem 2.8]{BD_Dec}. There are algorithms to compute the Hubbard trees and spiders out of bisets of polynomials; their implementations are available in the computer algebra system GAP \cite{GAP_IMG}.

\subsection{Expanding maps and quotients} Expansion is one of the key properties of PCF rational maps. It follows from the Schwarz lemma that $f\colon (\wC,A)\selfmap$ expands the hyperbolic metric on $\wC\setminus A$. (If $|A|=2$, then $f(z)= z^d$ expands the Euclidean metric of the cylinder $\C\setminus \{0\}$.)  If some of the points from $A$ are in the Julia set of $f$, then it is more convenient to consider the minimal hyperbolic (or Euclidean) orbifold $(\wC, \orb_f)$ of $f$. Periodic attracting cycles of $f$ are the only removed points in $(\wC,\orb_f)$; everywhere else $f$ is expanding.

A map $f\colon (S^2,A)\selfmap$ is called \emph{B\"ottcher expanding} if it admits an expanding metric mimicking the above expansion of PCF rational maps: $f$ is expanding everywhere except at removed critical cycles where $f$ is attracting and where $f$ has a \emph{B\"ottcher normalization} (i.e.,  it is locally conjugate to $z\mapsto z^d$). By~\cite{BD_Exp},  a non-Latt\`{e}s map $f\colon (S^2,A)\selfmap$ is isotopic to a B\"ottcher expanding map if and only if $f$ does not possesses a Levy obstruction. Other expanding sphere maps can be obtained by collapsing Fatou attracting basins of $f$, see~\cite[Proposition 1.1]{BD_Exp}.

For an expanding map $f\colon (S^2,A)\selfmap$, there is a natural notion of the Julia $\Jul(f)$ and Fatou $\Fat(f)$ sets. The case $\Jul(f)=S^2$ was studied in relation to the Cannon conjecture and quasi-symmetric geometry \cite{HP_Expanding, HP_Dimension,THEBook}. If $\Jul(f)=S^2$, then we say that $f\colon(S^2,A)\selfmap$ is a \emph{totally} expanding map.  More generally, we may consider \emph{totally topologically expanding} maps $g\colon X\selfmap $  on a compact metrizable space $X$,  see Section \ref{subsec:top_exp} for the definition. 

 Suppose that $f\colon(S^2,A)\selfmap$ is a B\"ottcher expanding map with non-empty Fatou set.  Let us denote by $\sim_{\Fat(f)}$ the smallest closed equivalence relation on $S^2$ generated by identifying all points in every Fatou component of $f$. The quotient space $S^2/\sim_{\Fat(f)}$ is a \emph{cactoid},  that is, a continuum composed of countably many spheres and segments pairwise intersecting in at most one point. The map $f$ naturally descends to a continuous map $\overline{f}\colon S^2/\sim_{\Fat(f)}\selfmap$ on the cactoid.  Then the corresponding (monotone) quotient map $\pi_{\overline f}\colon S^2 \to S^2/\sim_{\Fat(f)} $ provides a semi-conjugacy from $f$ to $\overline f$. It is naturally characterized by the following result. 

\begin{maintheorem}
\label{thm: max exp quotient}
Let $f\colon(S^2,A)\selfmap$ be a B\"ottcher expanding map with $\Fat(f)\not=\emptyset$. Then the induced map $\bar f\colon S^2/\sim_{\Fat(f)}\selfmap$ is the maximal totally expanding quotient. That is, any other semi-conjugacy $\pi_g\colon S^2\to Y$ from $f$ to a totally topologically expanding map $g\colon Y\selfmap$ factorizes through $\pi_{\overline f}\colon S^2\to S^2/{\sim_{\Fat(f)}}$:

\begin{center}
\begin{tikzcd}[row sep=tiny]
                         & S^2\arrow{dl}{\pi_{\overline f}}  \arrow{dd}{\pi_{g}}\arrow[loop right]{l}{f} \\
  S^2/\sim_{\Fat(f)} 
  \arrow[loop left]{l}{\overline{f}} 
  \arrow{dr}&              \\
  &Y \arrow[loop right]{l}{g}
\end{tikzcd}
\end{center}
\end{maintheorem}

\subsection{Crochet decomposition} We say that a B\"ottcher expanding map $f\colon(S^2,A)\selfmap$ is a \emph{crochet map} if there is a connected forward-invariant zero-entropy graph containing $A$. Polynomials, Newton maps, matings where one of the polynomials has a zero-entropy Hubbard tree are examples of crochet maps.  The following result provides various characterizations of crochet maps.

\begin{maintheorem}
\label{thm:crochet_inro}
Let $f\colon(S^2,A)\selfmap$  be a B\"ottcher expanding map with a non-empty Fatou set. Then the following are equivalent:
\begin{enumerate}[label=\text{(\roman*)},font=\normalfont,leftmargin=*]
\item $f$ is a crochet map, that is,  there is a connected forward-invariant zero-entropy graph $G$ containing $A$; 
\item $S^2/\sim_{\Fat(f)}$ is a singleton;
\item every two points in $A$ may be connected by a path $\alpha$ with $\alpha\cap\Jul(f)$ being countable.
\end{enumerate}
\end{maintheorem}

We say that a  B\"ottcher expanding map  $f\colon(S^2,A)\selfmap$ is a \emph{Sierpi\'{n}ski map} if its Julia set is homeomorphic to the standard Sierpi\'{n}ski carpet,  that is,   the Fatou set $\Fat(f)$ is non-empty,  Fatou components have pairwise disjoint closures, and the closure of every Fatou component is a Jordan domain.  The following easily follows from Whyburn's characterization \cite{Why}, Moore's theorem \cite{Moore,WhyBook}, and \cite[Section~4.5]{BD_Exp}.


\begin{proposition}
\label{thm:sierpinski}
Let $f\colon(S^2,A)\selfmap$ be a a B\"ottcher expanding map with $\Fat(f)\not=\emptyset$. Then the following are equivalent:
\begin{enumerate}[label=\text{(\roman*)},font=\normalfont,leftmargin=*]
\item $f$ is a Sierpi\'{n}ski map, i.e., $\Jul(f)$ is homeomorphic to the standard Sierpi\'{n}ski carpet;
\item $S^2/\sim_{\Fat(f)}$ is a sphere and $\sim_{\Fat(f)}$ is trivial on $A$;
\item every connected periodic zero-entropy graph is homotopically trivial rel.  $A$. 
\end{enumerate}
\end{proposition}

  Let us write $f\colon(S^2,A,\CC)\selfmap$ for a branched covering $f\colon(S^2,A)\selfmap$ with an invariant multicurve $\CC=f^{-1}(\CC)$. Then $\CC$ splits $(S^2,A)$ into finitely many spheres marked by $A$ and $\CC$. Every small sphere of $(S^2,A,\CC)$ is either periodic or preperiodic under $f$; the first return map along a periodic cycle determines the type of maps in the cycle. The inverse operation is an \emph{amalgam}  producing a global map out of small maps and the gluing data \cite{Pilgrim_Comb}. We remark that decompositions and  amalgams are topological operations and the resulting objects are unique up to isotopy. If $f\colon(S^2,A,\CC)\selfmap$ is expanding, then every small sphere of $f\colon(S^2,A,\CC)\selfmap$ has the associated \emph{small Julia set} in $\Jul(f)$. Small Julia sets may intersect and may even coincide with the global Julia set $\Jul(f)$ (for example for matings).

The next theorem is the main result of this paper providing the \emph{crochet canonical decomposition} of expanding maps into crochet and Sierpi\'{n}ski maps.

\begin{maintheorem}
\label{thm:cro decomp}
Let $f\colon(S^2,A)\selfmap$ be a B\"ottcher expanding map with $\Fat(f)\not=\emptyset$.  There is a unique canonical invariant multicurve $\CC_\cro$ whose small maps are Sierpi\'{n}ski and crochet maps such that for the quotient map $\pi_{\overline{f}}\colon S^2\to S^2/{\sim_{\Fat(f)}}$ the following are true:
\begin{enumerate}[label=\text{(\roman*)},font=\normalfont,leftmargin=*]
\item  small Julia sets of Sierpi\'{n}ski maps project onto spheres;
\item small Julia sets of crochet maps project to points;
\item  different crochet Julia sets project to different points in $S^2/{\sim_{\Fat(f)}}$.
\end{enumerate}
\end{maintheorem}  
  
  \subsection{Bicycles and Sierpi\'{n}ski maps} Let $\CC$ be a multicurve consisting of periodic (up to homotopy) curves such that $\CC$ is maximally strongly connected: for all $\gamma,\delta\in\CC$ an iterated preimage of $\gamma$ is homotopic to $\delta$ and $\CC$ can not be enlarge while keeping this property.  We call $\CC$ a \emph{bicycle} if its curves replicate: there is an $n\ge 1$ such that every $\gamma\in \CC$ is homotopic to at least two components of $f^{-n}(\gamma)$.


The crochet decomposition can now be characterized as follows:
  
\begin{maintheorem}
\label{thm:Sierp bicycle}
Let $f\colon(S^2,A)\selfmap$ be a B\"ottcher expanding map with $\Fat(f)\not=\emptyset$. Then maximal Sierpi\'{n}ski small maps are well-defined. If two bicycles have a positive geometric intersection number, then these bicycles are within a small Sierpi\'{n}ski map.   

The crochet multicurve $\CC_\cro$ is generated by the boundaries of maximal Sierpi\'{n}ski small maps and the remaining bicycles. 
\end{maintheorem}

Let us say that a B\"ottcher expanding map $f\colon (S^2,A)\selfmap$ with $\Fat(f)\not=\emptyset$ is \emph{Sierpi\'{n}ski-free} if it does not contain any small Sierpi\'{n}ski maps with respect to any invariant multicurve. In this case $\CC_\cro$ is generated by all bicycles of $f$.

A topological space $X$ is a \emph{dendrite} if it is a locally connected continuum that contains no simple closed curves.

\begin{maintheorem}\label{thm: sierp-free-intro}
Let $f$ be a B\"ottcher expanding map with $\Fat(f)\not=\emptyset$. 
Then the following are equivalent:
\begin{enumerate}[label=\text{(\roman*)},font=\normalfont,leftmargin=*]
\item none of the small maps in the decomposition of $f$ along the crochet multicurve $\CC_\cro$ is a Sierpi\'{n}ski map;
\item $S^2/{\sim_{\Fat(f)}}$ is a dendrite;
\item the decomposition of $f$ with respect to every invariant multicurve $\CC$ does not produce a Sierpi\'{n}ski small map.
\end{enumerate}
\end{maintheorem}

\begin{figure}[t]
  \centering
  \begin{subfigure}[b]{0.30\textwidth}
    \begin{overpic}
      [height=4cm, tics=10,
      ]{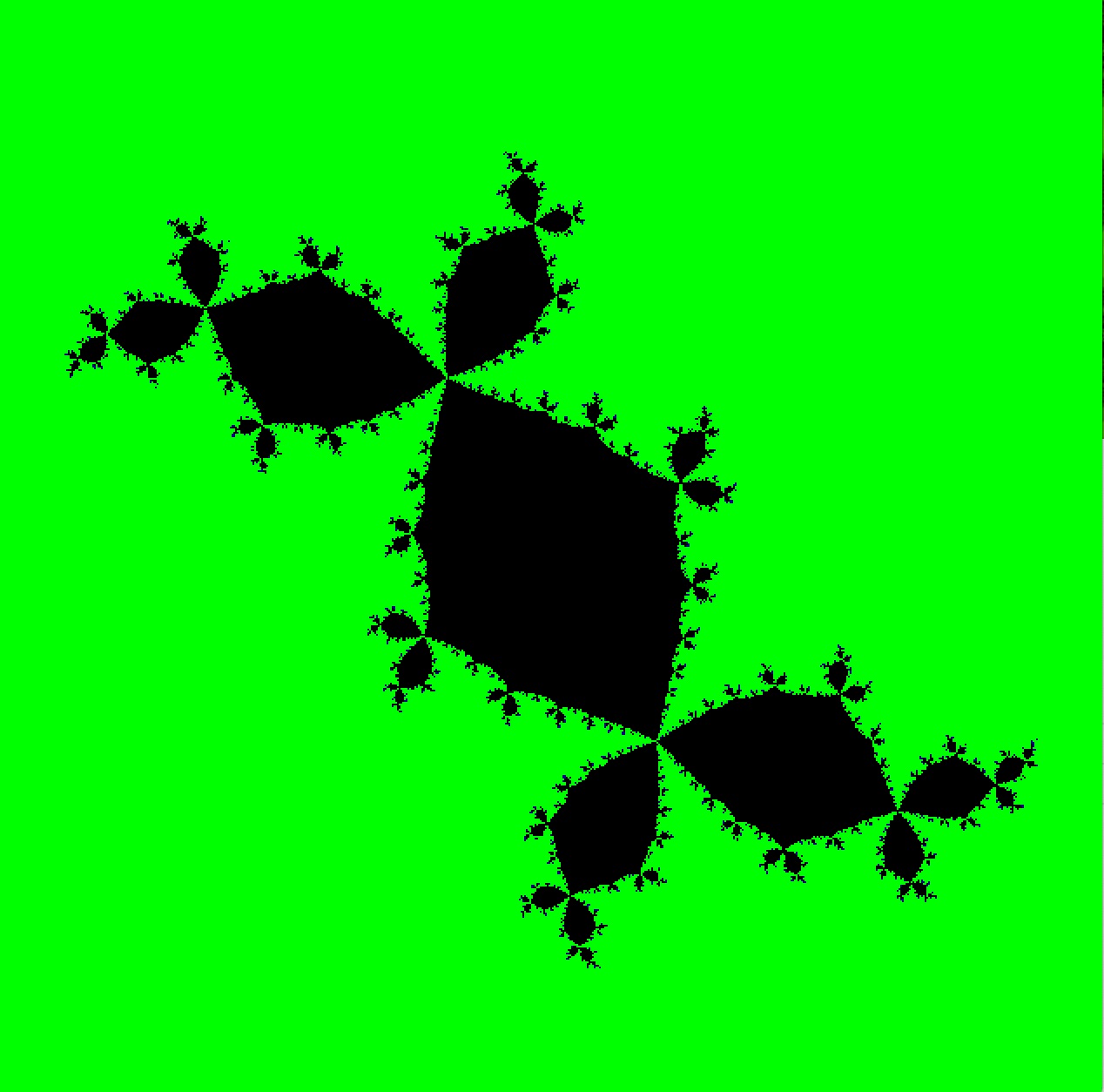}
    \put(45,-10){\footnotesize $(a)$}
    \end{overpic}
  \end{subfigure}
  \hfill
  \begin{subfigure}[b]{0.33\textwidth}
    \begin{overpic}
      [height=4cm, tics=10,
      ]{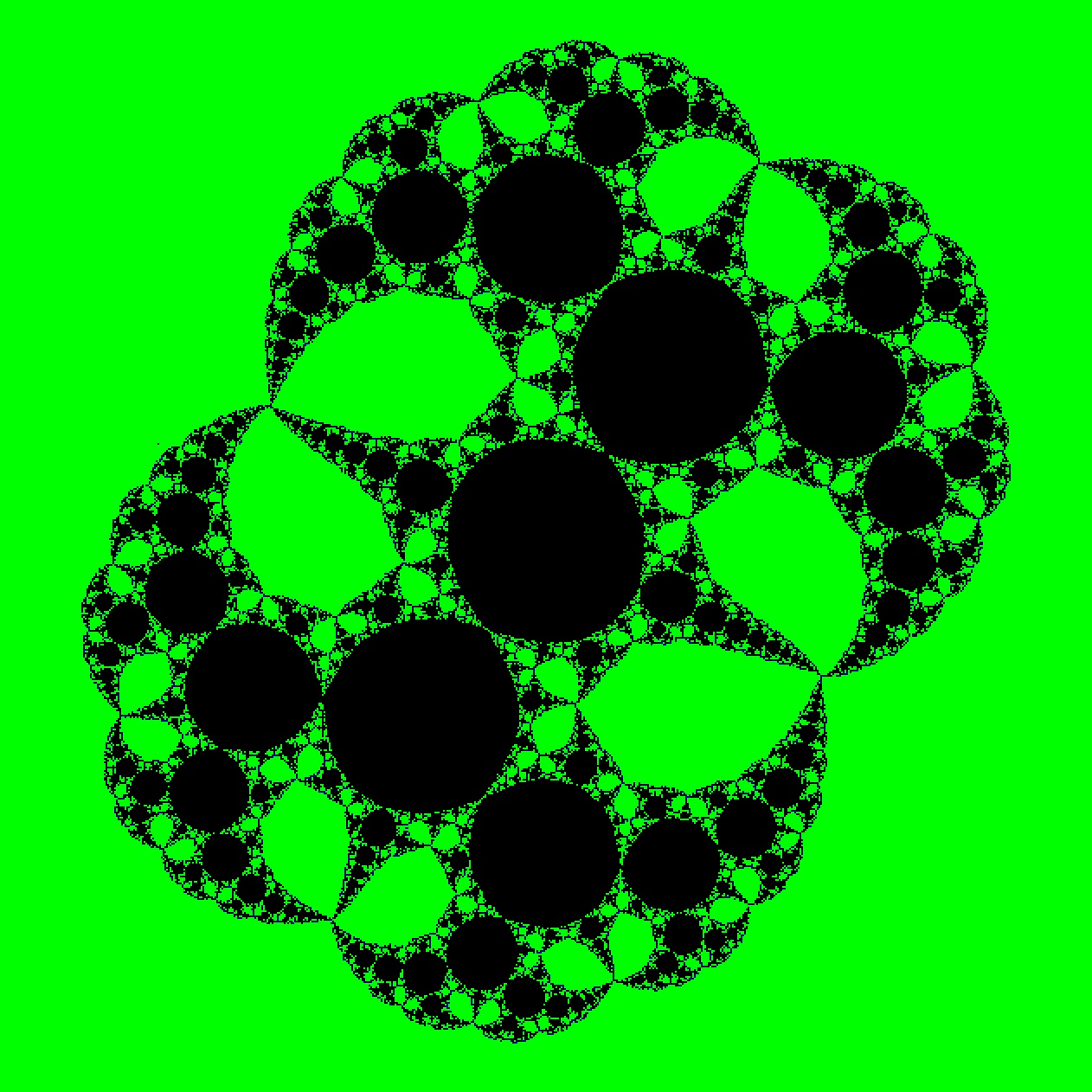}
          \put(45,-10){\footnotesize $(b)$}
        \end{overpic}
  \end{subfigure}
    \begin{subfigure}[b]{0.33\textwidth}
    \begin{overpic}
      [height=4cm, tics=10,
      ]{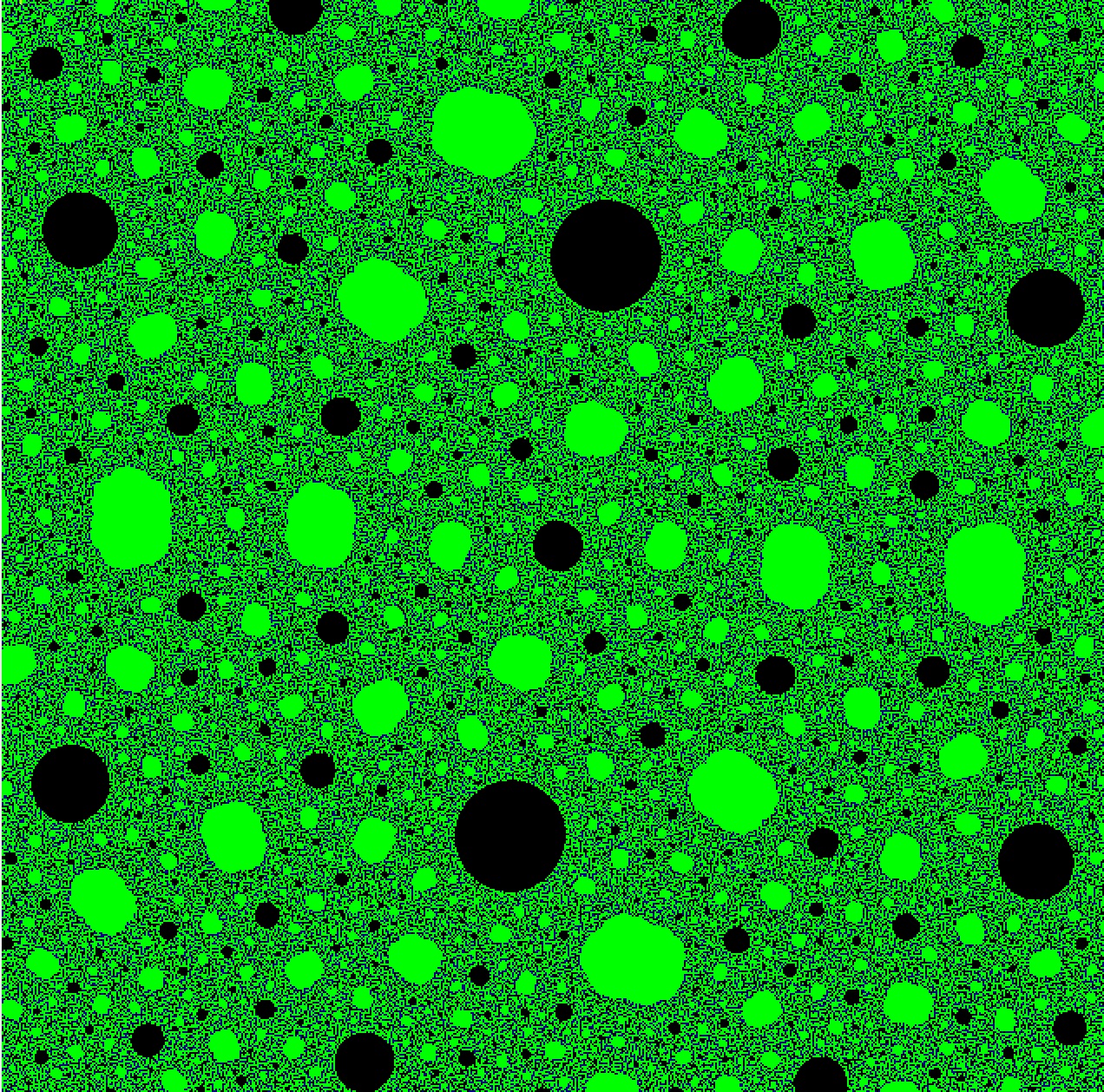}
          \put(45,-10){\footnotesize $(c)$}
        \end{overpic}
  \end{subfigure}\\
  \vspace{0.5cm}\mbox{}\\
   \hspace{0.6cm}
   \begin{subfigure}[b]{0.30\textwidth}
    \begin{overpic}
      [height=4cm, tics=10,
      ]{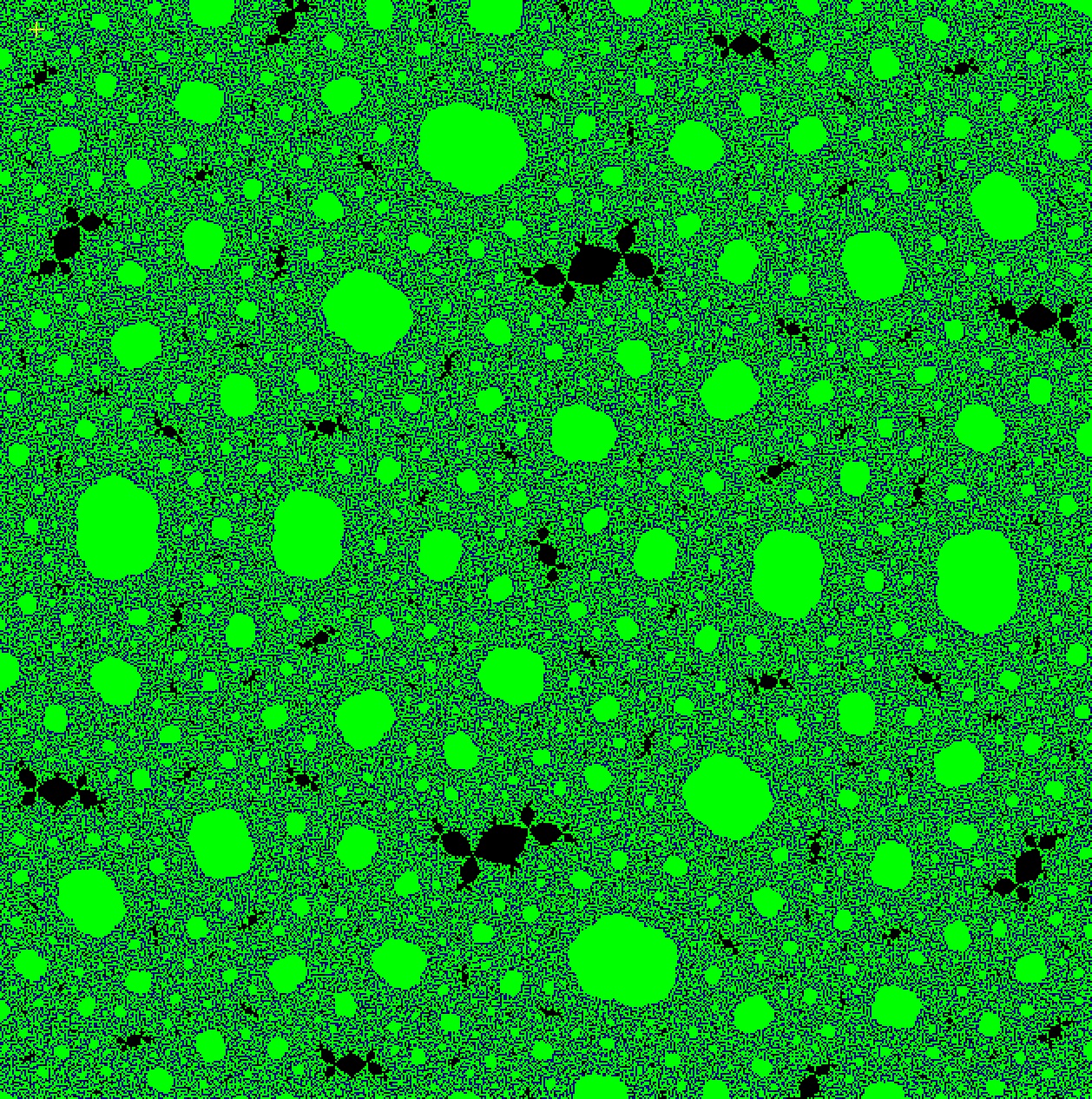}
    \put(45,-10){\footnotesize $(d)$}
    \end{overpic}
  \end{subfigure}
   \hspace{1cm}
 \begin{subfigure}[b]{0.33\textwidth}
    \begin{overpic}
      [height=4cm, tics=10,
      ]{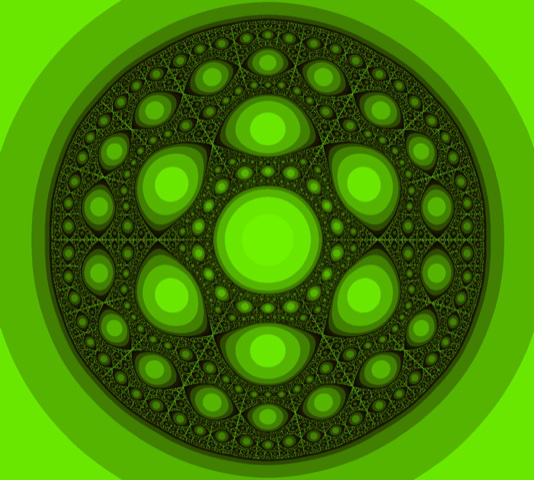}
          \put(45,-10){\footnotesize $(e)$}
        \end{overpic}
  \end{subfigure}

          \mbox{}\\
  \caption{Julia sets of PCF rational maps.}
\label{fig:Julia-sets}
\end{figure}

For a B\"ottcher expanding map $f\colon (S^2,A)\selfmap$ with $\Fat(f)\not=\emptyset$, let $\CC_\Sie$ denote the multicurve generated by the boundaries of maximal Sierpi\'{n}ski small maps. Then $\CC_\Sie\subset \CC_\cro$ and both of these multicurves encode topological features of $\Jul(f)$:
\begin{itemize}
\item maximal Sierpi\'{n}ski small maps correspond to small spheres of $S^2/\sim_{\Fat(f)}$;
\item small Sierpi\'{n}ski-free maps of $f\colon (S^2,A,\CC_\Sie)\selfmap$ correspond to dendrites of $S^2/\sim_{\Fat(f)}$ that share at most one point with any small sphere; 
\item bicycles in $\CC_\cro\setminus \CC_\Sie$ correspond to arcs in $S^2/\sim_{\Fat(f)}$ that share at most one point with any small sphere;
\item small crochet maps correspond to points in $S^2/\sim_{\Fat(f)}$.
\end{itemize}
Let us remark that the above properties are already visible in $\Jul(f)$, see Figure~\ref{fig:Julia-sets}. Here, (a),(b) are Julia sets of crochet maps: any Fatou component may be connected to another one by a (countable) chain of touching Fatou components, that is, $S^2/\sim_{\Fat(f)}$ is a singleton.  The Julia set in (c) is a Sierpi\'{n}ski carpet: Fatou components are Jordan domains with disjoint closures.  The  Julia set in (d) corresponds to a tuning, and its canonical decomposition returns (a) and (c). Note that the quotient $S^2/\sim_{\Fat(f)}$ is a sphere for (c) and (d), but $\sim_{\Fat(f)}$ is not trivial on the postcritical set for (d).  For the Julia set in (e), the quotient $S^2/\sim_{\Fat(f)}$  is a segment with the quotient dynamics $\overline f\colon S^2/\sim_{\Fat(f)}\selfmap$ of a Chebychev polynomial.  The Fatou component of infinity corresponds to a small crochet map.  We also see a Cantor set of Jordan curves in $\mathcal{J}(f)$ separating infinity and the Fatou component in the center.

\subsection{Crochet Algorithm}\label{ss: crochet_algo_intro}

The proofs of Theorem~\ref{thm:cro decomp} and~\ref{thm:Sierp bicycle} give an algorithm to effectively compute the crochet decomposition:
\begin{enumerate}
\item Compute maximal clusters of touching Fatou components and their boundary multicurve.\label{cr alg:st 1}
\item Decompose the map with respect to the boundary multicurve of the clusters.\label{cr alg:st 2}
\item Iterate Steps~\ref{cr alg:st 1} and~\ref{cr alg:st 2} for each small map until all small maps are crochet and Sierpi\'{n}ski.
\item Glue small crochet maps that correspond to the same point in $S^2/\sim_{\Fat(f)}$.
\end{enumerate} 

All steps in the Crochet Algorithm can be performed symbolically with the input being a sphere biset of $f\colon(S^2,A,\CC)\selfmap$. 

\subsection{Applications of the crochet decomposition} The crochet decomposition appears to be a useful tool in complex dynamics. Below we briefly discuss several natural applications in quite different contexts.  

Let $f$ be a B\"ottcher expanding map with non-empty Fatou set,  and suppose that $\CC_{\Th}$ is the Pilgrim canonical obstruction and $\CC_\cro$ is the canonical crochet multicurve of $f$.  Using \cite[Theorem 3.2]{PT}, see also \cite[Theorem 7.6]{ParkLevy}, it follows that up to isotopy each curve $\gamma\in \CC_{\Th}$ either belongs to $\CC_{\cro}$ or to a Sierpi\'{n}ski small sphere (wrt.\ the decomposition along $\CC_\cro)$.  Since  $\CC_\cro$ can be efficiently computed,  \cite{BD_Exp} implies that the problem of detecting obstructions for Thurston maps is reduced to efficient localization of Levy multicurves for arbitrary Thurston maps and Thurston obstructions for B\"ottcher expanding maps $f$ with $\Jul(f)$ being the whole sphere $S^2$ or a Sierpi\'{n}ski carpet. 

In \cite{ParkThesis}, I.~Park studies the \emph{Ahlfors regular conformal dimension} (for short, ARConfDim) of the Julia sets of crochet maps. In particular, he proves that $\operatorname{ARConfDim}(\Jul(f))=1$ for a PCF hyperbolic rational map $f$ if and only if $f$ is a crochet map. Our work (in particular, Theorems \ref{thm:cro decomp} and \ref{thm:Sierp bicycle}) provides some ingredients for the proof of one of the directions. I.~ Park also conjectures that for a B\"ottcher expanding map $f$ with non-empty Fatou set  the crochet decomposition provides a lower bound on  $\operatorname{ARConfDim}(\Jul(f))$ as the maximum of the ARConfDim's of the small Julia sets with respect to $\CC_\cro$ and an extremal quantity $Q(\CC_\cro)$ associated with the invariant multicurve $\CC_\cro$ (see \cite[Section~7.5]{Dylan_Positive} for the definition). We refer the reader to \cite{CanaryMinsky,CarrascoConfDim} for similar results in the context of geometric group theory.


It is conjectured that the \emph{iterated monodromy groups} (for short, IMGs) of PCF rational maps are \emph{amenable}. Theorem \ref{thm:crochet_inro} and \cite[Theorem 7.1]{NPT_Amenability}
provide the necessary ingredients for application of the amenability criterion from \cite{Jusch_Nekr_Salle} resulting in the following partial result: the IMGs of PCF crochet rational maps are amenable; c.f. \cite[Corollary 7.2]{NPT_Amenability}. This generalizes the previous amenability results from \cite{AmenabilityInBoundedGroups, H_Thesis}.

\subsection{Organization of this paper} The paper is organized as follows.  In the next section we review background.  In particular, we discuss decomposition theory for Thurston maps in Section \ref{ss: dec and amal}.  In Section \ref{sec: geom amal}, we introduce the geometric amalgam operation for  B\"ottcher expanding maps, which generalizes the notion of geometric mating.  In Section \ref{sec:clusters} we introduce clusters induced by touching Fatou components and use them to build $0$-entropy invariant graphs.  In Section \ref{sec: existence} we introduce the notion of a crochet multicurve, which gives the crochet decomposition,  and provide the Crochet Algorithm that constructs this curve (with a proof).  In Section \ref{sec: cactoid maps}, we describe how a B\"ottcher expanding maps with an invariant multicurve (satisfying certain natural conditions) generates a cactoid with an expanding dynamics on it. In Section \ref{sec: cannonicity}, we prove canonicity of crochet decomposition and provide alternative characterizations.

\subsection{Acknowledgments}
We are grateful to a number of colleagues for helpful and inspiring discussions during the time when we worked on the 
project, in particular Laurent Bartholdi, Mario Bonk, Andr\'{e} de Carvalho, Kostya Drach, Misha Lyubich, Curtis McMullen, 
Daniel Meyer, Volodia Nekrashevych, Insung Park, Kevin Pilgrim, Bernhard Reinke, and Dylan Thurston.
We also thank the Mathematical Sciences Research Institute in Berkeley, California, where the authors were in residence 
during the Spring semester of 2022, for its hospitality, support, and stimulating research environment.

\section{Background}
\subsection{Notations and conventions}
\label{ss:Not Conv}
 The sets of positive integers, integers, and complex numbers are denoted by $\N$, $\Z$, and $\C$, respectively. We use the notation $\disk \coloneq\{z \in \C\colon |z|<1\}$ for the open unit disk in $\C$, $\widehat{\C}\coloneq\C\cup\{\infty\}$ for the Riemann sphere, and $S^2$ for a topological $2$-sphere.

Let $X$ be a topological space. A curve (or a path) in $X$ is the image of a continuous function from a closed interval of the real numbers into $X$. By default, every curve $\alpha$ is parametrized by $[0,1]$. As common, we will use the same notation for the curve and its parametrizing function, so that $\alpha(0)$ and $\alpha(1)$ denote the starting and ending point of $\alpha$, respectively.  If $\alpha$ and $\beta$ are two paths in $X$ such that $\alpha(1)=\beta(0)$, then $\alpha\# \beta$ denotes their concatenation; i.e., the path $\alpha\# \beta$ first runs through $\alpha$ and then through $\beta$.





Let $A\subset S^2$ be a finite set.  We refer to the pair $(S^2, A)$ as a \emph{marked sphere}, and to the points in $A$ as \emph{marked points}. 

A \emph{sphere map} $f\colon (S^2,C) \to (S^2,A)$ between marked spheres is an orientation-preserving branched covering map between the underlying spheres such that $A$ contains $f(C)$ as well as all critical values. If $\alpha$ is a path in $S^2\setminus A$ starting at $x$, then $\alpha\lift{f}{y}$ denotes the unique lift of $\alpha$ starting at $y\in f^{-1}(x)$.

A \emph{Thurston map} is a self-map $f\colon (S^2,A)\selfmap$ with topological degree $\deg(f)\geq 2$. Note that since $A$ contains all the critical values of $f$ and $f(A)\subset A$,  it follows that each critical point of $f$ has finite orbit, that is, $f$ is \emph{postcritically finite}.

A \emph{forgetful map} $i\colon (S^2,C)\to (S^2,A)$ is the identity map that forgets points in $C\setminus A$; in particular $i(C)\supset A$. Note that a non-trivial forgetful map is \emph{not} a sphere map. More generally, a \emph{forgetful monotone map} $i\colon (S^2,C)\to (S^2,A)$ with $C\supset A$ is a continuous monotone map that is homotopic to $\id$ rel. $A$. 

Curves in $(S^2,A)$ will be frequently considered up to homotopy rel.\ $A$ and the endpoints: for two curves $\alpha_0$ and $\alpha_1$ in $S^2$ we write $\alpha_0\simeq_A \alpha_1$ if there is a path of curves $\alpha_t\colon [0,1]\to S^2$ such that $\alpha^{-1}_t(A)$, $\alpha_t(0)$, and $\alpha_t(1)$ are constant (i.e., do not depend on $t$).

By a \emph{nice curve} in $(S^2,A)$ we mean a curve 
$\beta\colon [0,1]\to S^2$ with $\beta(t)\not\in A$ for each $t\in (0,1)$.  
If a nice curve $\beta$ starts at $x\in S^2\setminus A$ and $f\colon (S^2,A)\selfmap$ is a Thurston map, then the lift $\beta\lift{f^n}{y}$ is well-defined for each $n$ and $y\in f^{-n}(x)$.

\subsection{B\"ottcher expanding maps}
\label{ss:B exp maps}

Let $f\colon (S^2,A)\selfmap$ be a Thurston map. We
denote by $A^\infty\subset A$ the forward orbit of periodic critical points of $f$, and by $\per(a)$ the period of each $a\in A^\infty$. 

\hide{ A
\emph{non-torus} map is a map that is not doubly covered by a torus
endomorphism.}

The map $f$ is \emph{metrically expanding} if there exists a forward-invariant subset $A'\subset A^\infty$ and  a length metric on
  $S^2\setminus A'$ such that
\begin{enumerate}[label=\text{(\roman*)},font=\normalfont,leftmargin=*]
 \item for every non-trivial rectifiable curve
   $\gamma\colon[0,1]\to S^2\setminus A'$ the length of every lift of
   $\gamma$ under $f$ is strictly less than the length of $\gamma$;
   and 
 \item\label{defn:MetrExp:ExpCond} $f$ admits \emph{B\"{o}ttcher normalization} at all $a\in A'$: the first return map of $f$ at $a$ is locally
   conjugate to $z\mapsto z^{\deg_a(f^{\per(a)})}$.\label{defn:MetrExp:NormCond}
 \end{enumerate}
In this case, we say that $f$ is metrically expanding rel.\ $A'$. We note that points in $A'$ are cusps or, equivalently, at infinite distance in the metric.  If $A'=A^\infty$, then $f\colon (S^2,A)\selfmap$ is called a
 \emph{B\"ottcher expanding map}.

Metrically expanding maps $f$ admit natural definitions of Fatou and Julia sets. We summarize the relevant notions and properties below; the reader is referred to \cite{BD_Exp} for a more detailed discussion. 

In the following, let $f$ be a metrically expanding map rel.\ to a forward-invariant subset $A'\subset A^\infty$. 
The \emph{Julia set} $\Jul(f)$ of $f$ is the closure of the set of repelling periodic points of $f$; 
the complement $\Fat(f)\coloneq S^2\setminus \Jul(f)$ is the \emph{Fatou set} of $f$. Equivalently, $\Fat(f)$ is the set of points attracted by $A'$.  A \emph{Fatou component} of $f$ is a connected component of $\Fat(f)$. Every Fatou component $F$ is an open topological disk. Moreover, $F$ is eventually periodic, i.e., $f^n(F)=f^m(F)$ for some $n>m\geq 0$.

Condition~\eqref{defn:MetrExp:NormCond} implies that periodic (and hence preperiodic) Fatou components enjoy B\"ottcher coordinates: for each Fatou component $F$ there exists a homeomorphism $\psi_F\colon \disk \to F$ such that $(\psi^{-1}_{f(F)}\circ f|_F \circ \psi_F)(z)= z^{\deg(f|_F)}$ for every $z\in \disk$. Every map $\psi_F$ extends to a continuous map $\psi_F\colon \overline \disk \to \overline{F}$. The image $\{\psi_F(re^{2\pi i\theta})\colon r\in [0,1]\}$ of a radius in $\overline{\disk}$ is called the \emph{internal ray} of angle $\theta\in[0,1)$ in $F$. The image $\{\psi_F(re^{2\pi i\theta})\colon \theta\in [0,1)\}$ of a circle in $\disk$ concentric about $0$ is called the \emph{equipotential} of height $r\in (0,1)$ in $F$. Finally, the image $\psi_F(0)$ is called the \emph{center} of $F$.


Fix a small $r\in (0,1)$. For every $a\in A'$, let $U_a$ to be the open disk around $a$ bounded by the equipotential of height $r$ in the Fatou component centered at $a$. Set $W\coloneq \bigcup _{a\in A'} U_a$. Note that $f(W)\subset W$.




Consider a nice curve $\alpha$ in $(S^2,A)$ parameterized by $[0,1]$. Assume that $\alpha$ is not in $W$, and set  $t_0^\alpha\coloneq\inf\{t\in[0,1]\colon \alpha(t)\notin W\}$ and $t_1^\alpha\coloneq\sup\{t\in[0,1]\colon \alpha(t)\notin W\}$. We define the  \emph{truncated length} of $\alpha$ as \[| \alpha |_{\simeq} \coloneq \inf\{\text{length of}\sp \alpha'| [t^{\alpha'}_0,t^{\alpha'}_1] \colon  \alpha'\simeq_A \alpha \}.\]
Note that if $\alpha$ starts and ends outside $W$ then $|\alpha|_\simeq$ coincides with the minimal length within the homotopy class of $\alpha$.

The following lemmas are immediate.


\begin{lemma}
\label{lmm:TrLength:shrinks}
There are constants $\lambda>1$ and $C>0$ such that the following holds. Let $\alpha$ be a nice curve connecting a point in $S^2\setminus W$ to a point in $A'\cup S^2\setminus W$. Then the truncated length of every lift of $\alpha$ under $f^n$ is at most $\frac 1{\lambda^n} |\alpha|_{\simeq}+C$.\qed
\end{lemma}

\begin{lemma}
\label{lmm:TrLength:fin many}
For every $M>0$ and every $x,y\in A'\cup S^2\setminus W$, there are at most finitely many homotopy classes of curves connecting $x$ and $y$ with truncated length at most $M$. \qed
\end{lemma}
\begin{proof}
For every $a\in A'$ and $x\in \partial U_a$, there is a unique up to homotopy path in $\overline U_a$ connecting $a$ and $x$. This reduces the statement to the case $x,y\in S^2\setminus W$.
\end{proof}

\newcommand{\MM}{{\mathcal M}}
\subsection{Topologically expanding maps} \label{subsec:top_exp}
Let $\MM',\MM$ be compact metrizable topological spaces with $\MM'\subset\MM$, and $f\colon \MM'\to \MM$ be a continuous map (e.g., a branched covering of finite degree). 
Suppose $\UU$ is a collection of open sets in $\MM$. We denote by $\mesh(\UU)$ the supremum of all diameters of connected components of sets in $\UU$ (with respect to some fixed metric on $\MM$ generating the topology of $\MM$).  The \emph{pull-back} of $\UU$ by $f^n$ is defined as 
\[f^{*n}(\UU)\coloneq \{\text{$V$: $V$ is a component of $f^{-n}(U)$,  for some $U\in \UU$}\}.\]
The partial self-cover  $f\colon \MM'\to \MM$  is called \emph{topologically expanding} if there exists a finite open cover $\UU$ of $\MM$ such that $\mesh(f^{*n}(\UU))\to 0$ as $n\to \infty$. The \emph{Julia set} of $f$ is the set of points in $\MM'$ that do not escape $\MM'$ under iteration of $f$. 
If $\MM'=\MM$, then we say that $f\colon \MM\selfmap$ is \emph{totally topologically expanding}.


Consider now a Thurston map $f\colon (S^2,A)\selfmap$ and let $A'$ be a forward-invariant subset of $A^\infty$. We call $f$ \emph{topologically expanding} (\emph{rel.\ $A'$}) if there exist compact $\MM'\subset \MM\subset S^2$  with a topologically expanding branched covering map $f \colon \MM'\to \MM$, such that every connected component $U$ of $S^2\setminus \MM$ is a disk containing a unique point $a\in A'$, all points in $U$ are attracted to the orbit of $a$, and the first return map of $f$ at $a$ is locally conjugate to $z\mapsto z^{\deg_a(f^{\per(a)})}$. Every topologically expanding Thurston map $f\colon (S^2,A)\selfmap$ is obtained from a B\"ottcher expanding map $\widetilde{f}\colon (S^2,A)\selfmap$ that is isotopic to $f$ by collapsing grand orbits of those Fatou components that are attracted towards $A\cap \big(\Fat(\widetilde{f}) \setminus \Fat(f)\big)$, see \cite[Theorem A and Proposition 1.1]{BD_Exp}. 

\subsection{Monotone maps between spheres} 
Recall that a map $\tau \colon X\to Y$ between topological spaces is called \emph{monotone} if $\tau^{-1}(y)$ is connected for every $y\in Y$. Every monotone self-map $\iota\colon S^2\selfmap$ arises as a uniform limit of homeomorphisms \cite{MonotoneMaps}.  In particular, the Dehn-Nielsen-Baer Theorem is applicable for monotone maps: two monotone maps $i_1,i_2\colon (S^2,A)\selfmap$ are homotopic if and only if their pushforwards $i_{1,*},i_{2,*}\colon \pi_1(S^2,A)\selfmap$ induce the same elements of the outer automorphism group of $\pi_1(S^2,A)$. Below we will review the pullback argument in the setting of topologically expanding maps.

\subsubsection{Uniform convergence} Suppose $\UU_n = \{U_{n,k}\}_k$ are finite covers of $S^2$ by open connected sets such that $\lim_{n\to\infty}\mesh(\UU_n)=0$, that is,  the maximal diameter of $U_{n,k}$ tends to $0$ as $n\to \infty$. 
Then a sequence of homeomorphisms $i_m\colon S^2\selfmap$ \emph{converges uniformly} if
\[ \forall n'>1\sp \exists N, M \ge 1 \sp \forall n\ge N\sp  \forall U_{n,k}\in \UU_n\sp \forall m\ge M\]
the image $i_m(U_{n,k})$ is within a component of $\UU_{n'}$. In this case, the limiting self-map $\lim _m i_m\colon S^2\selfmap$ exists and is a monotone map.

\subsubsection{Pullback argument}  
\label{sss:PullbackArg:S^2}
Let $f\colon(S^2,A)\selfmap$ be a Thurston map with B\"ottcher normalization at a forward-invariant set $A'\subset A^\infty$, i.e.,  the first return map at every $a\in A'$ is locally conjugate to $z\mapsto z^{\deg_a(f^{\per(a)})}$. If $g \colon(S^2,A)\selfmap$ is a topologically expanding map rel.\ $A'$ such that $f$ and $g$ are isotopic rel.\ $A$, then there is a continuous monotone map $\iota\colon (S^2,A)\to (S^2,A)$ semi-conjugating $f$ to $g$.  Moreover, $\iota$ is unique if $|A|\ge 3$ (as follows from \cite[Corollary 4.30]{BD_III}).  The map $\iota$ is constructed as follows 
(c.f. \cite{IshiiSmillie},\cite[Section 11.1]{THEBook}, and \cite[Section 4.5]{BD_Exp}).

Let $i_0, i_1\colon (S^2,A)\to (S^2,A)$ be homeomorphisms isotopic to $\id$ rel.\ $A$ that witness an equivalence of $f$ and $g$; i.e., $i_0\circ f= g\circ i_1$ and there is an isotopy \[h_{0}\colon (S^2,A)\times [0,1] \to (S^2,A), 
\sp\sp \text{with} \sp h_{0,0}=i_0 \sp \text{and} \sp h_{0,1}=i_1.\]
(Here and in the following, we adopt the convention that if $h_n$ is a homotopy then $h_{n,t}$ denotes the time-$t$ map $h_n(\cdot, t)$.) Up to modification of $i_0$, we may assume that $i_0,i_1,h_0$ respect B\"{o}ttcher coordinates (of $f$ and $g$) at $A'$, so that $i_0=i_1=h_{0,t}$ for each $t\in[0,1]$ in a small neighborhood of each $a\in A'$.

By lifting, we may inductively define homeomorphisms $i_{n+1}\colon (S^2,A)\to (S^2,A)$ and isotopies $h_{n}$ between $i_{n}=h_{n,0}$ and $i_{n+1}=h_{n,1}$. That is, the isotopy $h_{n+1}$ is the lift of $h_{n}$ (so that $h_{n+1,0}=h_{n,1}=i_{n+1}$), and we have the following infinite commutative diagram:
\[
{\begin{tikzpicture}
  \matrix (m) [matrix of math nodes,row sep=4em,column sep=7em,minimum width=2em]
  {
    \cdots \rightarrow (S^2,  A) &  (S^2,  A) & (S^2,  A) \\
    \cdots \rightarrow  (S^2,  A) &(S^2,  A) & (S^2,A) \\};
  \path[-stealth]
    (m-1-1) edge node [right] {$h_{2,t}$} (m-2-1)
    (m-1-2) edge node [right] {$h_{1,t}$} (m-2-2)
    (m-1-3) edge node [right] {$h_{0,t}$} (m-2-3)

      (m-1-1)  edge     node [above] {$f $} (m-1-2)
             (m-1-2)    edge  node [above] {$f $} (m-1-3)
 
    (m-2-1.east|-m-2-2) edge node [above] {$g$}
          (m-2-2)
        (m-2-2.east|-m-2-3) edge node [above] {$g$}
          (m-2-3);
\end{tikzpicture}}.
\]

A \emph{track} of $h_{n}$ is the curve $\gamma$ defined by \[\gamma(t) \colon t\mapsto h_{n,t}(x),\sp t\in [0,1],\]
for some fixed $x\in S^2$. Clearly, tracks of $h_{n+1}$ are lifts of tracks of $h_{n}$ under $g$. Tracks for concatenations $h_{n}\#h_{n+1}\#\dots \#h_{n+k}$ of isotopies are defined in a similar way.


We assume that $g$ is topologically expanding with respect to a finite open cover $\UU_0$ of $X_0$, where $X_0$ is the sphere $S^2$ without a small forward-invariant neighborhood of $A'$ (where $i_0,i_1,h_0$ respect B\"{o}ttcher coordinates). That is, we have $\lim_{n\to\infty}\mesh(\UU_n)=0$, where $\UU_n\coloneq g^{*n}(\UU_0)$
is the pull-back of $\UU_0$ by $g^n$. Note that $\UU_n$ is an open cover of $X_n\coloneq g^{-n}(X_0)$.

For a curve $\gamma\colon [0,1]\to X_n$, we denote by $|\gamma|_n$ the \emph{$n$-combinatorial length} of $\gamma$ with respect to $\UU_n=\{U_{n,k}\}_k$: 
\[|\gamma|_n\coloneq \min \{L\colon  \gamma=\gamma_1\#\gamma_2\#\dots \#\gamma_L \text{ so that }\gamma_j\subset U_{n,k(j)}\}.\] Since $\lim_{n\to\infty}\mesh(\UU_n)=0$, there are $\lambda>1, m\ge 1, C>0$ such that $|\gamma|_0\le \frac 1 \lambda|\gamma|_m+C$ for all $\gamma\subset X_m$. Lifting under $g^n$ we obtain 
\begin{equation}
\label{eq:topol expans}
    |\gamma|_{n}\le \frac 1 \lambda |\gamma|_{n+m} +C \sp\sp\sp\text{for all }n\ge 1 \text{ and }\gamma\subset X_{n+m}.
\end{equation} 
By uniform continuity, there is a constant $T'>0$ that bounds from above all the tracks of $h_0, \dots, h_{m-1}$ with respect to $|\ |_0$. 
By lifting and applying a geometric series argument to~\eqref{eq:topol expans}, we obtain a uniform $T\ge1 $ such that the $0$-combinatorial length of all the tracks of $h_0\#h_{1}\#\dots \#h_{k}$ is bounded by $T$ for all $k\geq0$. Lifting under $g^s$, we obtain that $T$ also bounds the $s$-combinatorial length of all the tracks of $h_s\#h_{s+1}\#\dots \#h_{s+k}$ for all $s,k$. Since $\lim_{n\to\infty}\mesh(\UU_n)=0$, the homotopy $h_s\#h_{s+1}\#\dots \#h_{s+k}$ converges uniformly to the identity and $i_n$ converges uniformly to a monotone map $\iota\colon S^2\to S^2.$

\subsection{Multicurves on marked spheres}\label{ss:MC on marked sph}

Let $(S^2,A)$ be a marked sphere. We consider simple closed curves $\gamma$ on $(S^2,A)$, that is, $\gamma\subset S^2\setminus A$.  Such a curve $\gamma$ is called \emph{essential} if both components of $S^2\setminus \gamma$ contain at least two marked points.  Otherwise,  the curve $\gamma$ is called \emph{peripheral}.


A \emph{multicurve} on $(S^2,A)$ is a collection $\CC$ of essential simple closed curves on $(S^2,A)$ that are pairwise disjoint and non-isotopic. Here and elsewhere, in the context of simple closed curves on $(S^2,A)$ isotopies are always considered rel.\ $A$ (unless specified otherwise). It would be also convenient to consider (multi)curves defined up to isotopy.

If $\CC'$ is a finite set of pairwise disjoint simple closed curves on $(S^2,A)$, then $\MC(\CC')$ is the multicurve obtained from $\CC'$ by identifying isotopic curves and removing non-essential curves.

Let $K$ be a compact connected subset of $S^2$. Then every connected component $U$ of $S^2\setminus K$ is an open topological disk, and so we may choose a homeomorphism $\rho\colon U\to \disk$. For $r\in (0,1)$, let $E^r=\{z\colon |z|=r\}\subset \disk$ be the circle of radius $r$ centered at the origin.  Note that $\rho^{-1}(E^r)$ is a simple closed curve in $S^2$ whose isotopy class rel.\ $A$ is constant for $r$ close to $1$. Moreover, this isotopy class does not depend on the choice of the homeomorphism $\rho$. 
We define $\Curve(K\mid U)$ to be such a unique up to isotopy curve $\rho^{-1}(E^r)$. 
Furthermore, we set 
\[ \MC(K)\coloneq \MC \left( \bigcup_{U\subset S^2\setminus K} \Curve(K\mid U) \right),\]
where the union is taken over all connected components $U$ of $S^2\setminus K$. 

For a compact, but not necessarily connected, set $K\subset S^2$, we set 
\[\MC(K) \coloneq \MC\left(\bigcup_{K'\subset K}  \MC(K')\right),\]
where the union is taken over all connected components of $K$. The multicurve $\MC(K)$ is also defined for a countable union of pairwise disjoint compact sets. 


We will write $\Curve_A(\cdot)$ and $\MC_A(\cdot)$ if we wish to underline that the (multi)curves are considered on $(S^2,A)$. 

\begin{lemma}
\label{lem:MC:lifting}
Suppose $f\colon(S^2,C)\to (S^2,A)$ is a sphere map and $K\subset S^2$ is compact. Then
\[\MC_C( f^{-1}(K)) = \MC_C\big(f^{-1}\big(\MC_A(K) \big)\big).\] 
\end{lemma}
\begin{proof}
It is sufficient to prove the identity for connected $K$. Let $U$ be a connected component of $S^2\setminus K$.  Then up to isotopy rel.\ $C$ we have the following equality
\[f^{-1}(\Curve_A(K\mid U)) =\bigcup_{K',U'} \Curve_C(K'\mid U'),\]
where the union is taken over all components $K'$ of $f^{-1}(K)$ and all the components $U'$ of $S^2\setminus K'$ such that 
the $f$-image of $U'$ intersected with a sufficiently small neighborhood of $K'$ is within $U$. 
The assertion of the lemma now follows from the definitions. 
\end{proof}

\subsubsection{Pseudo-multicurves}\label{sss:psMC} By a \emph{pseudo-multicurve} $\CC$ on $(S^2,A)$ we mean a collection of pairwise disjoint and pairwise non-isotopic non-trivial simple closed curves on $S^2\setminus A$; i.e.~peripheral curves are allowed. Naturally, pseudo-multicurves are considered up  to isotopy. Pseudo-multicurves will only appear in the proof of Theorem~\ref{thm: max_clusters_exist}.

 Allowing peripheral curves, we define $\psCurve(K\mid U)$ and $\psMC(K)$ in the same way as $\Curve(K\mid U)$ and $\MC(K)$. Similar to Lemma~\ref{lem:MC:lifting}, we have 
\begin{equation}
\label{eq:sss:psMC}
\psMC_C( f^{-1}(K)) = \psMC_C\big(f^{-1}\big(\psMC_A(K) \big)\big),
\end{equation}
where peripheral curves are allowed.

\subsection{Invariant multicurves}
\label{ss: inv_multicurves}
In the following, let $f\colon (S^2,A)\selfmap$ be a Thurston map.  We say that a multicurve $\CC$ on $(S^2,A)$ is \emph{invariant} (under $f$) if the following holds:
\begin{enumerate}[label=\text{(\roman*)},font=\normalfont,leftmargin=*]
\item $f^{-1}(\CC)\subset \CC$,  which means that each essential component of $f^{-1}(\CC)$ is isotopic to a curve in $\CC$. 
\item $\CC\subset f^{-1}(\CC)$,  which means that each curve in $\CC$ is isotopic to a component of $f^{-1}(\CC)$.
\end{enumerate}
In other words, the multicurve $\CC$ is invariant if $\MC_A(f^{-1}(\CC))=\CC$.
We will use the notation $f\colon (S^2,A, \CC)\selfmap$ for a Thurston map  $f\colon (S^2,A)\selfmap$ with an invariant multicurve $\CC$. 

Suppose $\CC'$ is any multicurve on $(S^2,A)$.  If $\CC'\subset f^{-1}(\CC')$, then there is a unique invariant multicurve $\CC$ \emph{generated} by $\CC'$, which is given by the intersection of all invariant multicurves containing $\CC'$.

Let $\CC$ be an invariant multicurve. We consider the following directed graph $\Gamma=\Gamma_\CC$: its vertex set is $\CC$,  and for every curve $\gamma\in \CC$ and every essential component $\widetilde \gamma$ of $f^{-1}(\gamma)$ we add a directed edge in $\Gamma$ from $\gamma$ to the curve $\delta \in \CC$ that is isotopic to $\widetilde\gamma$.

Two vertices in $\Gamma$ are said to be \emph{strongly connected to each other} if there is a walk in $\Gamma$ between them in each direction. Clearly, this defines an equivalence relation on $\CC$.  An equivalence class $\CC'\subset \CC$ is then called a \emph{strongly connected component} of $\CC$, and the induced subgraph $\Gamma[\CC']$ is called  a \emph{strongly connected component} in $\Gamma$. Note that a singleton without a self-loop is never a strongly connected component in $\Gamma$.

Strongly connected components are partially ordered: if $\CC'$ and $\CC''$ are two strongly connected components of $\CC$, we write $\CC'\prec\CC''$ if there is a (directed) walk in $\Gamma$ from a curve in $\CC'$ to a curve in $\CC''$.  A strongly connected component $\PC\subset \CC$ is called a \emph{primitive component} of $\CC$ if it is minimal for the partial order $\prec$.  It immediately follows from the definitions that the multicurve $\CC$ is \emph{generated} by its primitive components in the following sense: for each curve $\gamma\in \CC$ there is a primitive component $\PC\subset \CC$ and an iterate $n$ such that $\gamma$ is isotopic to a component of $f^{-n}(\PC)$.

A strongly connected component $\CC'\subset \CC$ is called a \emph{bicycle} if for every $\gamma,\delta\in \CC'$ there exists an $n\in \N$ such that at least two (directed) walks of length $n$ join $\gamma$ and $\delta$ in in $\Gamma$.  Otherwise, we call $\CC'$ a \emph{unicycle}. 


\subsection{Decompositions and amalgams of Thurston maps} 
\label{ss: dec and amal}

We briefly review the decomposition theory of Thurston maps developed by K.~Pilgrim~\cite{Pilgrim_Comb}. The algebraic version of this theory was introduced in~\cite{BD_Dec}.


Consider a Thurston map $f\colon(S^2,A,\CC)\selfmap$. Up to homotopy, we may write $f$ as a correspondence
\begin{equation}
\label{eq:corr:f_i}(S^2,A,\CC)\overset i\leftarrow(S^2,f^{-1}(A),f^{-1}(\CC))\overset
f\to(S^2,A,\CC).
\end{equation}
Here, the map $f$ is the same as the original map $f$, but it is considered now as a covering (i.e., we remove the marking sets from the domain and target). 
The map~$i$, identifying the domain and target marked spheres, is specified as follows. It first forgets all points in $f^{-1}(A)\setminus A$ and all curves in $f^{-1}(\CC)$ that
are not isotopic rel.\ $A$ to curves in $\CC$. Then it squeezes all annuli between
the remaining curves in $f^{-1}(\CC)$ that are isotopic, and maps them to the corresponding curve in
$\CC$. This uniquely defines $i$ as a monotone 
map on $S^2$ up to homotopy rel.\ $A$.



A \emph{small sphere} $S_z$ of $(S^2,A,\CC)$ is a connected component of $S^2\setminus \CC$. Viewing holes in $S_z$ as punctures, we obtain a sphere $\widehat S_z$ marked by the respective subset $A_z$ of $A\cup \CC$. With a slight abuse of terminology, we will refer to this marked sphere as a small sphere of $(S^2,A,\CC)$ as well. Similarly, we introduce small spheres of $(S^2,f^{-1}(A),f^{-1}(\CC))$.

A small sphere $S_z$ of $(S^2,f^{-1}(A),f^{-1}(\CC))$ is called
\begin{enumerate}[label=\text{(\roman*)},font=\normalfont,leftmargin=*]
\item \emph{trivial}, if $S_z$ is homotopic rel.\ $A$ to 
a point or peripheral curve in $(S^2,A)$;
\item \emph{annular}, if $S_z$ is homotopic rel.\ $A$ to a curve in $\CC$;
\item \emph{essential}, otherwise.
\end{enumerate}

For every small sphere $S_z$ of $(S^2,f^{-1}(A),f^{-1}(\CC))$ there is a unique small sphere $S_{f(z)}$ of $(S^2,A,\CC)$ such that $f\colon S_z\setminus f^{-1}(A)\to S_{f(z)}\setminus A$ is a covering. Filling-in holes, we view the latter as a sphere map $f\colon (\widehat  S_z,A_z)\to  (\widehat S_{f(z)},A_{f(z)})$.

Essential small spheres of $(S^2,f^{-1}(A),f^{-1}(\CC))$ may be canonically identified with small spheres of $(S^2,A,\CC)$. Namely, for every small sphere $S_z$ of $(S^2,A,\CC)$ there is a unique small sphere $S_{i^*(z)}$ of $(S^2,f^{-1}(A),f^{-1}(\CC))$ such that $S_{i^*(z)}$ is homotopic to $S_z$ rel.\ $A$ (the corresponding homotopy fills in the holes in $S_{i^*(z)}$ associated with $f^{-1}(\CC)\setminus \CC$). This induces a forgetful map $(\widehat S_{i^*(z)}, A_{i^*(z)})\to  (\widehat S_z,A_z)$; its inverse $(\widehat S_z,A_z) \to  (\widehat S_{i^*(z)}, A_{i^*(z)})$ is a sphere map. 

The composition 
\begin{equation}
\label{eq:small sphere map}
\widehat S_z\to \widehat S_{i^*(z)} \to \widehat S_{f(z)} \sp\sp\text{ where } f(z)\coloneq f (i^* (z)) 
\end{equation}
is well-defined (see \cite[Lemma 4.9]{BD_Dec}) and is called a \emph{small (non-dynamical) sphere map} of $f\colon(S^2,A,\CC)\selfmap$. The small sphere map $(\widehat S_z, A_z)\to  (\widehat S_{f(z)}, A_{f(z)})$ is unique up to homotopy rel.\ the marked points.

Note that~\eqref{eq:small sphere map} naturally induces a map on the small spheres of $(S^2,A,\CC)$, which we still denote by $f$ for simplicity. Clearly, every small sphere of $(S^2,A,\CC)$ is either \emph{periodic} or \emph{strictly preperiodic}. Moreover, there are only finitely many periodic cycles of small spheres. A \emph{small (self-)map} of $(S^2,A,\CC)$ is the first return map $\widehat f\coloneq f^k \colon (\widehat S_z, A_z) \selfmap$ along such a periodic cycle (with some choice of a base small sphere $\widehat S_z$). Each such small map is either a homeomorphism or a Thurston map. 


By the \emph{decomposition} of $f\colon(S^2,A,\CC)\selfmap$ (along the invariant multicurve $\CC$) we mean 
either
\begin{itemize}
\item  the collection of small sphere maps (a non-dynamical decomposition); or
\item  
the collection of small self-maps, one per every periodic cycle of small spheres.
(a dynamical decomposition).
\end{itemize}

The converse procedure is called \emph{amalgam}. It takes as input a collection of small sphere maps $(f\colon \widehat S_z\to \widehat S_{f(z)})_z$, as well as an appropriate ``gluing data'',  and outputs a global map $f\colon(S^2,A,\CC)\selfmap$; see \cite[\S3]{Pilgrim_Comb}.



\section{Formal amalgam}
\label{sec: geom amal}
%
%
%
%
%
%
%
%
%
%
%

 \subsection{Expanding quotients}
Let $f\colon(S^2,A)\selfmap$ be a Thurston map with B\"ottcher normalization at each point in $A^\infty$. We denote by $\Jul_f$ the set of points in $S^2$ that are not attracted by $A^\infty$.  We also fix a base metric on $S^2$ that induces the given topology on $S^2$.

Two points $x,y \in \Jul_f$ are called \emph{homotopy equivalent} if there is an $M>0$ such that for every $n\ge 0$ the points $f^n(x)$ and $f^n(y)$ can be connected by a nice curve~$\ell_n$ (see Section \ref{ss:Not Conv}) with $|\ell_n|\le M$ such that $\ell_n\lift{f^n}{x}$ ends at $y$. Moreover, we say that $x,y$ are \emph{strongly homotopy equivalent} if the curves $\ell_n$ additionally satisfy  $\ell_n\simeq_{A} \ell_{n+1}\lift{f}{f^n(x)}$ for each $n
\geq 0$.

\begin{proposition}
\label{prop:hom equiv}
Let $f\colon(S^2,A)\selfmap$ be a B\"ottcher normalized map and $\bar f\colon(S^2,A)\selfmap$ be a B\"ottcher expanding map isotopic to $f$.  Suppose that $h\colon (S^2,A)\to (S^2,A)$ is a continuous monotone map with $h\simeq_A \id$ that provides a semi-conjugacy between $f$ and $\bar f$ (see Section \ref{sss:PullbackArg:S^2}). Then $h^{-1}(\Jul_{\bar f})=\Jul_f$ and the following are equivalent for $x,y\in \Jul_f$:
\begin{itemize}
\item $h(x)=h(y)$;
\item $x,y$ are homotopy equivalent;
\item $x,y$ are strongly homotopy equivalent.
\end{itemize}
Moreover, the constant $M$ in the definition of the strong homotopy equivalence can be taken to be uniform over all equivalence classes. 
\end{proposition}
\begin{proof}
Since $\Jul_f$ and $\Jul_{\bar f}$ are non-escaping sets, we have $h^{-1}(\Jul_{\bar f})=\Jul_f$.

Clearly, strong homotopy equivalence implies homotopy equivalence.

Suppose that $x,y$ are homotopy equivalent rel.\ $f$. Then $h(x),h(y)$ are homotopy equivalent rel.\ $\bar f$ and we obtain  $h(x)=h(y)$.

Before proving converse, let us introduce some additional terminology. We will assume that $f$ has hyperbolic orbifold. (The case when $f$ has parabolic orbifold will follow in a similar way with some natural modifications.) Fix a universal orbifold covering map $\rho\colon \disk\to (S^2,A,\orb_f)$. For $X\subset S^2$, we define $\diam_\rho (X)\in [0,\infty]$ to be the diameter (with respect to the hyperbolic metric on $\disk$) of a connected component of $\rho^{-1}(X)$. 

\begin{lemma}
\label{lem:const:M}
There is an $M>0$ such that $\diam_\rho\big(h^{-1}(z)\big) \le M$ for every $z\in \Jul_{\bar f}$.  
\end{lemma}
\begin{proof}
Let us choose closed topological disks $V_1,\dots,V_s$ such that \[\Jul_{\bar f}\subset \bigcup _{i=1}^s V_i\sp\sp \text{ and} \sp\sp |V_i\cap A|\le 1\sp \text{ for all }i.\] 

Choose a monotone $\tau_i\colon (S^2,A)\to (S^2,A)$ such that $\tau_i(V_i)=\{v_i\}$ and $\tau_i\mid S^2\setminus V_i$ is injective.

Since $\tau_i\circ h\colon(S^2,A)\to (S^2,A)$ is monotone, \[M_i\coloneq \diam_\rho \big(h^{-1}\circ \tau^{-1}_i(v_i)\big)<\infty .\] Then $M\coloneq\max_{1\le i\le s} M_i$ provides the desired bound.
\end{proof}

Suppose $x,y\in Z\coloneq h^{-1}(z)$ for some $z\in \Jul_{\bar f}$. We will prove that $x,y$ are strongly homotopy equivalent with the constant $M$ from Lemma~\ref{lem:const:M}. 

Consider \[f^n(x),f^n(y)\in  Z_n\coloneq f^n(Z)=h^{-1}\circ \bar f^n(z)\sp\sp\sp \text{ for  }n\ge 0.\]
Fix a connected component $\widetilde Z_n$ of $\rho^{-1}(Z_n)$. Choose a very small open neighborhood $U_n$ of $Z_n$, and denote by $U$ the component of $f^{-n}(U_n)$ containing $Z$. Connect $x,y$ by a curve $\gamma$ in $U$. Since $U_n$ is a very small neighborhood of $Z_n$, the curve $\gamma_n\coloneq f^n(\gamma)$ has a lift $\widetilde \gamma_n$ connecting two lifts of $f^n(x),f^n(y)$ in $\widetilde Z_n$. Let us homotope $\tilde \gamma_n$ into a geodesic $\tilde \ell_n$. Then $\ell_n\coloneq \rho(\tilde \ell_n)$ connects $f^n(x)$, $f^{n}(y)$ and has length at most $M$. By construction, $\ell_n$ has a lift $\ell$ homotopic to $\gamma$ rel $A,\orb_f$. 
\end{proof}

\subsection{Formal amalgams}\label{ss:FA} In this subsection, we extend the notion of a formal mating to amalgams; compare with the notion of ``trees of correspondences'' from \cite[\S6.2]{BD_Exp}. This will allow us to relate the Julia set of an amalgam with the Julia sets of its small maps.

Let $\widetilde S$ be a finite (disjoint) union of topological spheres marked by a finite set~$A$. We assume that $f\colon (\widetilde S,A)\selfmap$ is B\"ottcher expanding: it expands a length metric on $\widetilde S\setminus A^\infty$, where $A^\infty$ is the forward orbit of periodic critical points, and the first return map at all $a\in A^\infty$ is conjugate to $z\mapsto z^{\deg_a(f^{\per(a)})}$.

Given a forward-invariant set $A_\blowup\subset A^\Fat\coloneq A \cap \Fat(f)$, let 
\begin{equation}
\label{eq:sph map}
 f_\blowup\colon (\widetilde S_\blowup,A\setminus A_\blowup )\dashrightarrow (\widetilde S_\blowup,A\setminus A_\blowup )
\end{equation} 
be the partial branched covering obtained by blowing up every point $a\in A_\blowup$ into a closed circle $\delta_a$. More precisely:
\newcommand{\indet}{{\operatorname{indet}}}
\begin{itemize}
\item 
$\widetilde S_\blowup$ is a (disjoint) union of spheres with boundary components $\Delta_\blowup =(\delta_a)_{a\in A_\blowup}$ together with a monotone map 
\begin{equation}
\label{eq:rho_blowup}
    \rho_\blowup\colon (\widetilde S_\blowup,A\setminus A_\blowup )\to (\widetilde S, A)
\end{equation}    
such that $\rho_\blowup| \widetilde S_\blowup \setminus \Delta_\blowup$ is injective and  $\rho_{\blowup}(\delta_a)=a$  for every $a\in A_\blowup$;
\item $\rho_\blowup$ semi-conjugates $f_\blowup$ to $f$ on $\widetilde S_\blowup \setminus  A_\indet$, where
\[ A_\indet\coloneq  f^{-1}(A_\blowup)\setminus A_\blowup. \]
\end{itemize} 
Note that the map $f_\blowup$ is not defined on $A_\indet$. At the same time, the map $f_\blowup$ is uniquely defined on each boundary circle $\delta_a, a\in A_\blowup $, by continuity.

An \emph{annular map} is a partial covering map of the form
\begin{equation}
\label{eq:ann map}
f\colon \AA'\to \AA\sp\sp \text{ with }\sp \AA'\subset \AA \sp\sp\text{ and }\sp\partial \AA'\supset \partial \AA
\end{equation}
where $\AA $ and $\AA'$ are finite unions of closed annuli and circles (i.e., degenerate annuli). The map $f$ is \emph{expanding} if it expands a length metric on $\AA$.

Consider a Thurston map $f\colon (S^2,A,\CC)\selfmap$, where $A$ contains all the critical points. We say that $f$ is a \emph{formal amalgam of expanding maps} if it is obtained by
\begin{itemize}
 \item gluing a blown up B\"ottcher expanding  map~\eqref{eq:sph map} with an expanding annular map~\eqref{eq:ann map}, 

 \item fixing identification of the marked sets:
 \begin{equation}
 \label{eq:FA:marked sets}
   \nu\colon \sp [A\setminus  A_\blowup]_{\text{ of $f_\blowup$ \eqref{eq:sph map} }}  \xrightarrow{\text{gluing identification}} A _{\text{ of $f$}} \ ,
 \end{equation}
 \item by redefining the resulting map in a small neighborhood of  $A_\indet$ so that $f$ is a branched covering respecting~\eqref{eq:FA:marked sets} and so that $f$ has local B\"ottcher normalization around every $A^\infty_\text{ of $f$ }\cap \nu(\AA_\indet\cap A)$.
\end{itemize}
By construction, $\nu$ respects the dynamics on $A\setminus (A_\blowup\cup A_\indet)$ but $\nu$ may change the dynamics on $A\cap A_\indet$. The adjustment of $f$ in a small neighborhood of $A_\indet$ is unique up to isotopy  (such a neighborhood contains at most one critical point, which is necessary in $A\setminus A_\blowup$); see \cite[\S4]{Pilgrim_Comb}. By construction, a formal amalgam has B\"ottcher normalization. 
 
 Naturally, we view each component of $\widetilde S_\blowup$ and $\AA$ as a subset of $S^2$. Then $\Delta_\blowup = \partial \AA$. Our convention is that 
\begin{itemize}
    \item $\CC=\MC(\AA)$ is the multicurve induced by $\AA$.
\end{itemize}
We also note that primitive unicycles of $\CC$ give rise to circle components (i.e., degenerate annuli) of $\AA$. The relation between marked sets of $f_\blowup$ and the resulting blowup is stated in Lemma~\ref{lem:A:relat f f_blowup}.

\begin{lemma}
\label{lem:FormAmalg}
Assume  $f\colon (S^2,A,\CC)\selfmap$ is a Levy-free map not doubly covered by a torus endomorphism, where $A$ contains all the critical points of $f$. Then $f$ is isotopic to a formal amalgam $f_{\text{FA}}\colon (S^2,A,\CC)\selfmap$ of \eqref{eq:sph map} and \eqref{eq:ann map} where every (possibly degenerate) annulus of $\AA$ is homotopic to a unique component of $\CC$ and vice versa.
\end{lemma}
\begin{proof}
   We can thicken $\CC$ into a finite union $\AA$ of closed annuli and circles and we can isotope $f$ so that 
   \[\AA'\coloneq f^{-1}(\AA)\setminus \{\text{peripheral components of $f^{-1}(\AA)$}\}\]
   satisfies $\AA'\subset \AA$ and $\partial \AA'\supset \partial \AA$. Define $\widetilde S'\coloneq S^2\setminus \AA$ and compactify $\widetilde S'$ into $\widetilde S_\blowup$ by adding boundary circles. Since $f$ is Levy-free and not doubly covered by a torus endomorphism, so are the first return maps on periodic small spheres of the induced map $f_\blowup$ on $(\widetilde S_\blowup,A)$. Therefore, we can isotope $f_\blowup$ into an expanding map on $(\widetilde S_\blowup,A)$ satisfying~\eqref{eq:sph map}. After that, we isotope $f$ on $\intr(\AA)$ so that the induced map \eqref{eq:ann map} is expanding. 
\end{proof}

\subsubsection{Non-escaping sets}
\label{sss:NonEscSets}

Let us  denote by $\filled_{\widetilde S}$ the set of points in $\widetilde S\setminus \Delta_\blowup$ that do not escape into the Fatou components around $\Delta_\blowup$ under the iteration of \eqref{eq:sph map}. Also, let $\Jul_\AA$ be the non-escaping set for the map \eqref{eq:ann map}; it is a collection of simple closed curves isotopic to $\CC$. 

\begin{lemma}
\label{lem:A:relat f f_blowup} In a formal amalgam $f\colon (S^2,A,\CC)\selfmap$, every periodic cycle of $f\mid A$ is either within $\filled_{\widetilde S}$ or contains a point in $\nu \big( [A\cap A_\indet ]_{\text{ of $f_\blowup$ \eqref{eq:sph map} }}\big)$, see~\eqref{eq:FA:marked sets}. These possibilities are mutually excluded.     
\end{lemma}
\begin{proof} By construction, periodic points of $f\mid (A\cap \filled_{\widetilde S})$ are identified with periodic points of $f_\blowup \mid A\setminus A_\blowup$ under~\eqref{eq:FA:marked sets}. The new periodic point can be created only through redefining the dynamics on $A_\indet$. By construction, the new periodic points are not in $\filled_{\widetilde S}$. 
\end{proof}

Consider a point $a\in A_\blowup$ and an internal ray $I$ inside the Fatou component of $f$ centered at $a$. An \emph{internal ray between $\filled_{\widetilde S}$ and $\AA$} is the closure of $\rho_\blowup^{-1}(\intr(I))$. It is a closed arc connecting the boundary of the Fatou component around $\delta_a$ for $f_\blowup$ and $\delta_a$.



\subsubsection{Iterating formal amalgams} \label{sss:FA:n}Consider a formal amalgam $f\colon (S^2,A,\CC)\selfmap$ of~\eqref{eq:sph map} and~\eqref{eq:ann map}. Define 
\begin{equation} \widetilde S^n_\blowup \coloneq f^{-n}(\widetilde S_\blowup)\subset S^2 \sp\sp\text{ and }\sp\sp \AA^n \coloneq f^{-n}(\AA)\subset S^2.
\end{equation}
By construction, every component of $\widetilde S^n_\blowup $ and $\AA^n$ is within a component of either $\widetilde S_\blowup$ or of $ \AA$. 
Set $\AA'^n\coloneq f^{-1}(\AA^n)\cap \AA^n$. We view
\begin{equation}
\label{eq:FA:iteration} f|\widetilde S^n_\blowup\colon  \widetilde S^n_\blowup\dashrightarrow  \widetilde S^n_\blowup\sp\sp\sp\text{ and }\sp\sp\sp f| \AA'^n \colon \AA'^n\to \AA^n
\end{equation}
as a blown-up B\"ottcher expanding map and an expanding annular map. This allows to represent $f\colon (S^2,f^{-n} (A),f^{-n}(\CC))\selfmap$ as a formal amalgam of maps in~\eqref{eq:FA:iteration}. Clearly, components of $\widetilde S^n_\blowup$ are parametrized by small spheres of $(S^2,f^{-n} (A),f^{-n}(\CC))$ and annuli of $\AA^n$ are parametrized by $\CC^n\coloneq f^{-n}(\CC)$. We denote by $\filled_{\widetilde S^n}, \Jul_{\AA^n}$ the non-escaping sets of maps in~\eqref{eq:FA:iteration}; they are iterated preimages of $\filled_{\widetilde S}, \Jul_{\AA}$. Clearly,
\[ \filled_{\widetilde S^n}\subset \filled_{\widetilde S^m}\sp\sp\text{ and }\sp\sp\Jul_{\AA^n}\subset \Jul_{\AA^m} \sp\sp\text{ for }n\le m.\]
Note also that for $n\leq m$, $\Jul_{\AA^m} \setminus \Jul_{\AA^n}$ contains only peripheral curves rel. $f^{-n}(A)$.

\subsubsection{Gluing $\filled_{\widetilde S}$ and $\Jul_{\AA}$}
Following \cite[Definition 6.4]{BD_Exp}, a \emph{pinching cycle} connecting $x,y\subset \filled_{\widetilde S}\cup \Jul_\AA$ is a simple arc  $I_1\# I_2\#\dots  \#I_k$ formed by a concatenation of internals rays $I_s$, where each $I_s$ connects a point in $\filled_{\widetilde S^n}$ to a point in $\partial \AA^n\subset \Jul_{\AA^n}$ (see Section~\ref{sss:NonEscSets}) for some $n\ge 0$.

\begin{theorem}
\label{thm:Form amalg}
let $f\colon(S^2,A)\selfmap$ be a formal amalgam of a blown-up B\"{o}ttcher expanding map \eqref{eq:sph map} 
and an expanding annular map \eqref{eq:ann map}. 
Suppose that $\tau$ is a semi-conjugacy from  $f\colon(S^2,A)\selfmap$ to a B\"ottcher expanding map $\bar f\colon(S^2,A)\selfmap$ as in Section \ref{sss:PullbackArg:S^2}.  Then \[\tau(x)=\tau(y)\sp\sp\sp\text{ for }\sp x,y\in \filled_{\widetilde S}\cup \Jul_\AA \] if and only if there is a pinching cycle of internal rays connecting $x$ and $y$.

Moreover, there is an $N\in \N$ such that $\tau| \filled_{\widetilde S}\cup \Jul_\AA$ identifies at most $N$ points.
\end{theorem}


\begin{proof}
If there is a pinching cycle $I_1\# I_2\# \dots \# I_k$ of internal rays between $x, y\in \Jul_{\widetilde S} \cup \Jul_\AA$, then $\ell_n\coloneq  f^n(I_1)\#f^n(I_2)\#\dots \# f^n(I_k)$ is a pinching cycle between $f^n(x)$ and $f^n(y)$. The curves $\ell_n$ define a homotopy equivalence between $x$ and $y$. By Proposition~\ref{prop:hom equiv}, $\tau(x)=\tau(y)$.

Suppose that $\tau(x)=\tau(y)$.  By Proposition~\ref{prop:hom equiv}, the points $x,y$ are strongly homotopically equivalent. Consider a system of curves $\ell_n$ realizing the homotopy equivalence between $x,y$. By putting $\ell_n$ into the minimal position with $\partial \AA$, we can decompose every $\ell_n$ as \[\ell_n = T^{(n)}_1\#I^{(n)}_1\#T^{(n)}_2\#I^{(n)}_2 \dots T^{(n)}_k\#I^{(n)}_k \]
such that
\begin{itemize}
\item $T^{(n)}_j$ is a curve within a small sphere or a small annulus;
\item $I^{(n)}_1$ is a curve connecting a point in $\filled_{\widetilde S}$ and a point in $\Jul_\AA$; 
\item $T^{(n)}_j$ and $I^{(n)}_j$ are lifts of $T^{(n+1)}_j$ and $I^{(n+1)}_j$ respectively. 
\end{itemize}
By the expansion, we may assume that $I^{(n)}_j$ are internal rays between small spheres and annuli while $T^{(n)}_j$ are trivial. This proves the first part of the theorem.

Since the length of internal rays between small spheres and annuli are bounded below, the constant $M$ from Proposition~\ref{prop:hom equiv} determines the maximal size of pinching cycle of internal rays. It remains to show that for every expanding map $g\colon(S^2,C)\selfmap$ there is $N_g$ that bounds the number of internal rays landing at any point of $x\in\Jul_g$. By replacing $f$ with its iterate and replacing $x$ with its iterated image, we may obtain that $x$ is on the boundaries of only fixed Fatou components. For every such fixed Fatou component $F$, we can choose a basis $X_F$ for the biset of $g$ so that the associated symbolic presentation of $\Jul_g$ includes the standard parametrization of $\partial F$ by internal rays. Then the number of internal rays of $F$ landing at $x$ is bounded by the nucleus in the basis $X_F$. This implies the required existence of $N_g$. 
\end{proof}

\subsubsection{Gluing $\cup \filled_{\widetilde S^n}$ and $\cup \Jul_{\AA^n}$}\label{sss:glunig of S^n A^n}
Following the setup from Section  \ref{sss:FA:n}, for $\gamma\in \CC^n$, we denote by $\AA^n_\gamma$ the annulus in $\AA^n$ homotopic to $\gamma$ rel.\ $f^{-n}(A)$. For a strongly connected component $\Sigma$ of $\CC$ (see Section \ref{ss: inv_multicurves}), we set $\AA_\Sigma\coloneq \displaystyle \bigcup_{\gamma\in \Sigma} \AA_\gamma$ and denote by $\Jul_\Sigma$ the non-escaping set of $f\colon \AA_\Sigma\cap f^{-1}(\AA_\Sigma)\to \AA_\Sigma$. If $\Sigma$ is a unicycle, then $\Jul_\Sigma$ is a finite periodic cycle of simple closed curves. If $\Sigma$ is a bicycle, then $\Jul_\Sigma$ is a \emph{Cantor bouquet of simple closed curves}: topologically, $\Jul_\Sigma$ is a direct product between a Cantor set and $\mathbb{S}^1$. In all cases, every curve in $\Jul_\Sigma$ is isotopic to a unique curve in $\Sigma$. For $\gamma\in \Sigma$, we also write $\Jul_{\Sigma, \gamma}\coloneq \Jul_\Sigma\cap \AA_\gamma$.


It is easy to see that for $\delta$ in $\CC^n$, the set $\Jul_{\AA^n}\cap \AA_\delta$ is the union of iterated preimages of the $\Jul_{\Sigma, \gamma}$ over all trajectories realizing the condition $\delta \subset f^{-n}(\gamma)$ (up to homotopy), where $\gamma$ is a periodic curve in a strongly connected component $\Sigma$.

We say that 
\begin{itemize}
 \item a simple closed curve $\beta$ in $\Jul_{\AA^n}$ is a \emph{neighbor} to a component $\filled_{\widetilde S^n, i}$ of $\filled_{\widetilde S^n}$ if $\beta $ is one the boundary of the component $\widetilde S_{\blowup, i}^n$ containing $\filled_{\widetilde S^n, i}$;
\item two curves $\beta_1,\beta_2$ in $\Jul_{\AA^n}$ are \emph{neighbors} if they have a common neighboring component $\filled_{\widetilde S^m, i}$ for some $m\ge n$.
\item two components $\filled_{\widetilde S^n, i}$, $\filled_{\widetilde S^n, j}$ are \emph{neighbors} if they are neighbors to a common curve in $\Jul_{\AA^n}$;
\item a curve $\beta$ in $\Jul_{\AA^n}$ is a \emph{neighbor of a neighbor} to a component $\filled_{\widetilde S^n, i}$ if they have a common neighboring component $\filled_{\widetilde S^m, j}$. 
\end{itemize}
We remark that the notion of neighbors is independent of the embedding $\Jul_{\AA^n}\subset \Jul_{\AA^m}$ and of viewing small spheres of $f\colon (S^2,f^{-n}(A),f^{-n}(\CC))\selfmap$ as small spheres of $f\colon (S^2,f^{-m}(A),f^{-m}(\CC))\selfmap$ for $m\ge n$; i.e., neighbors remain neighbors if $n$ is increased.

We say that a curve $\beta$ in $\Jul_{\AA^n}$ is \emph{buried} if it does not have any neighboring component $\filled_{\widetilde S^m ,i}$ for all $m\ge n$.

Let $\tau$ be the semi-conjugacy from $f$ to $\bar f$ from Theorem~\ref{thm:Form amalg}. This theorem implies that for every buried curve $\beta$ in $\Jul_{\AA^n}$, the map $\tau|\beta$ is injective and $\tau^{-1}(\tau(\beta)) = \beta$. In particular, the image $\tau(\beta)$ is disjoint from $\tau(\filled_{\widetilde S^m}\cup (\Jul_{\AA^m}\setminus \beta))$ for all $m\ge n$. 
The gluing of neighbors is described by the pinching cycle condition in Theorem~\ref{thm:Form amalg}. 

\begin{lemma}
\label{lem:TotDisc} Every connected component of $Y\coloneq \displaystyle S^2\setminus \bigcup_{n\ge 0}  \tau\big(\filled_{\widetilde S^n} \cup \Jul_{\AA^n}\big)$ is either a singleton or the closure of a Fatou component; this closure is a closed Jordan disk. This Fatou component is in the attracting basin of cycles intersecting $\nu \big( [A\cap A_\indet ]_{\text{ of $f_\blowup$ \eqref{eq:sph map} }}\big)$, see Lemma~\ref{lem:A:relat f f_blowup}.

The set $Y$ consists of points whose $\bar f$-orbits do not enter $\tau\big(\filled_{\widetilde S} \cup \Jul_{\AA}\big)$.
\end{lemma}


\begin{proof}
By construction, the set $Y$ is the $\tau$-image of the set $X$ of points in $S^2$ whose $f$-orbits never enter $\filled_{\widetilde S}\cup \Jul_{\AA}\cup \{\text{Fatou components around $\Delta_\blowup$}\}$. Equivalently, $x\in X$ if and only if the orbit of $x$ passes infinitely many times through Fatou components of $S^2\setminus \filled_{\widetilde S}$ associated with $A_\indet$. We obtain that every component $V$ of $X$ is a nested intersection of compactly contained disks. Therefore, $f^n(V)$ is a continuum containing at most one point in $A$ for all $n\ge 0$. If the orbit of $V$ is disjoint from $A^\infty$, then all points in $V$ are homotopically equivalent and $\tau(V)$ is a singleton by Proposition~\ref{prop:hom equiv}. If $f^n(V)$ intersects $A^\infty$, then $\tau(V)$ contains the closure of a Fatou component $F'$ and, since $f^n\colon V\to f^n(V)$ is a cyclic branched covering (i.e., it is topologically $z\mapsto z^D$), every point in $V$ is homotopically equivalent to a point in $\partial \tau^{-1}(F')$. By Proposition~\ref{prop:hom equiv}, $\tau(V)=\overline F'$. Clearly, there are no non-trivial Levy arcs starting at $A^\infty \cap \nu \big( [A\cap A_\indet ]_{\text{ of $f_\blowup$ \eqref{eq:sph map} }}\big)$, see~\eqref{eq:FA:marked sets}. Therefore, $\tau(V)=\overline F'$ is a Jordan disk.
\end{proof}

\section{Clusters of Fatou components}
\label{sec:clusters}

\subsection{Zero-entropy invariant graphs and clusters}

Consider a B\"ottcher expanding map $f\colon (S^2,A)\selfmap$ with the Julia set $\Jul\coloneq \Jul(f)$, where $A$ is the full preimage of the postcritical set.

A \emph{$0$-entropy graph} is a finite forward-invariant graph $G$ (embedded in $S^2$) such that $f|G$ has entropy $0$. We will always assume that $G$ intersects Fatou components along internal rays.

Consider a $0$-entropy graph $G$. Let $\sim_{G}\subset A\times A$ be the equivalence relation such that $a\sim_G b$ if and only if $a,b$ are in the same connected component of $G$. Let $\{G_i\}_{i\in I}$ be the set of connected components of $G$. We obtain the induced dynamics $f\colon I\selfmap$ specified by $f(G_i)\subset G_{f(i)}$.

For every $n\ge0$ and $i\in I$ we denote by $G^{(n)}_i$ the connected component of $ f^{-n}(G)$ containing $G_i$. It may happen that $ G^{(n)}_i= G^{(n)}_j$ for $i\not=j$. Set \[K_i^{(n)} \coloneq \overline{\bigcup _{F\cap G_i^{(n)}\not=\emptyset} F},\]
where the union is taken over all Fatou components that have non-empty intersection with $G_i^{(n)}$; these Fatou components are centered at $G_i^{(n)}$. Finally, we set
\[K_i\coloneq  \overline{\bigcup_{n}K_i^{(n)}}\]
and we denote by
$\Jul_i\coloneq \Jul\cap K_i$ the \emph{Julia part} of $K_i$.

We call $K_i$ a (\emph{crochet}) \emph{cluster}. Clearly,  each $K_i$ is connected and $f(K_i)\subset K_{f(i)}$.

\begin{lemma}
\label{lem:ext 0 entr graph}
Let $S\subset K_i$ be a finite set of (pre)periodic points. Then there is a $0$-entropy graph $\widetilde G \supset G$ such that
\begin{itemize}
\item if $\widetilde G_j$ is a connected component of $\widetilde G$ containing $G_j$, then $\bigcup_j \widetilde G_j=\widetilde G$;
\item $\widetilde G_i \supset S$. 
\end{itemize} 
\end{lemma}
\begin{proof}
Observe first that 
\[G^{(n)}\coloneq \bigcup_{i\in I}G_i^{(n)}\]
 is again a $0$-entropy graph. By replacing $G$ with $G^{(m)}$ for a sufficiently big $m$ and by identifying indices in $I$, we can assume that
\begin{itemize} 
\item  $G_i^{(n)}\not= G_j^{(n)}$ if $i\not= j$;
\item $G^{(m)} \cap f^{-1}(G_i)= G_i^{(1)}$ for each $m\geq 1$. 
\end{itemize}  
 
We assume that $S=\{x\}$ is a singleton; the general case easily follows by induction. The case $x\in G_i^{(n)}$ follows by setting $\widetilde G\coloneq G^{(n)}$. Therefore, we assume that  $x\not\in G_i^{(n)}$ for all $n\ge0$. (Note that $x\in G_j^{(n)}$ is allowed for $j\not =i$.)



Suppose that $x$ is on the boundary of a (pre)periodic Fatou component $F\subset K_i$.  Then the center of $F$ belongs to $G^{(n)}_i$ for some $n$. We connect $x$ to the center of $F$ via a (pre)periodic internal ray $\ell_x$ of $F$, and set $\widetilde{G} \coloneq G^{(n)} \cup \bigcup_{j\geq 0} f^j(\ell_x)$.  Note that the graph $\widetilde{G}$ is finite,  $\ell_x$ is (pre)periodic.

From now on, we assume that $x$ is not on the boundary of any Fatou component of $K_i$. We will assume, in addition, that $x$ is periodic (the preperiodic case will follow via lifting). 

\begin{lemma}
\label{lem:bubble ray land}
The periodic point  $x$ is the landing point of a \emph{periodic bubble ray}: there is a sequence of simple arcs $\ell_n\subset \overline {G_{i}^{(n)}\setminus G_{i}^{(n-1)}}$, where $\ell_n$ is a concatenation of edges,  such that
\begin{itemize}
\item $f^n(\ell_n)\subset G$ is a periodic sequence; and
\item $\widetilde \ell\coloneq \ell_1\#\ell_2\#\ell_3\#\dots $ forms a path connecting a point in $G_i$ to $x$.
 \end{itemize} 
 \end{lemma}
 \noindent Note that $\widetilde \ell$ may have self-intersections.
\begin{proof}

By passing to an iterate of $f$, it is sufficient to prove the lemma in case $f(G_i)\subset G_i$ and $x$ is a fixed point.  Let $w$ be a point in $G_i=G_i^{(0)}$.  

Fix constants $\lambda>1$ and $C>0$ as in Lemma \ref{lmm:TrLength:shrinks} and a positive $\epsilon\ll1$. 

For every vertex $v$ of $G_i^{(1)}$ choose a simple arc $\ell_v$ in $\overline {G_{i}^{(1)}\setminus G_{i}}$ connecting $v$ and a vertex of $G_i$. (If $v\in G_i$, then $\ell_v$ is a trivial arc.) We also fix a nice curve $\beta_v$ that is pseudo-isotopic to $\ell_v$ rel  $A \cup \partial  \ell_v$ via $H_v\colon S^2\times [0,1]\to S^2$.  Let $K>0$ be the maximal (truncated) length of the chosen $\beta_v$.  

For each $v'\in f^{-1}(v)\cap G_i^{(2)}$ we fix a lift of $\beta_v$ starting at $v'$ that necessary ends in a vertex of $G_i^{(1)}$.  The pseudo-isotopy $H_v$ determines a unique ``lift'' $\ell_{v'}=\ell_v \lift{f}{v'}$ of $\ell_v$ starting at each $v'$. 

Choose a nice curve $\alpha_0$ in $(S^2,A)$ connecting $x$ and a vertex $v_0\in G^{(1)}_i$ that is pseudo-isotopic to a curve in $K_i$ (rel $A \cup \partial \alpha_0$).  Fix a sufficiently large constant $M$ that satisfy \[M>\max\left(|\alpha_0|_\approx,  \frac{K + (2C+\epsilon)\lambda}{\lambda-1}\right).\]

We now inductively define nice curves $\alpha_n$ in $(S^2,A)$ ending in a vertex $v_n\in G_i^{(1)}$ with $|\alpha_n|_\approx < M$ in the following way.  There is a unique lift of $\alpha_{n-1}$ starting at $x$ (the uniqueness of the lift follows from $\deg_f(x)=1$).  This lift $\alpha_{n-1} \lift{f}{x}$ necessary ends at a vertex $v'_{n-1}$ in $G_i^{(2)}$ with $f(v'_{n-1})=v_{n-1}$.  Consider the path $\widetilde \alpha_n=\alpha_{n-1} \lift{f}{x} \# \beta_{v_{n-1}} \lift{f}{v_{n-1}'}$. Then its endpoint $v_n$ necessarily belongs to $G_i^{(1)}$. Choose a nice curve $\alpha_{n}$ that is homotopic to $\widetilde \alpha_n$.
We may assume that the truncated lengths of $ \alpha_{n}$ satisfies 
\[
| \alpha_{n}|_\approx \leq |\alpha_{n-1} \lift{f}{x}|_\approx + |\beta_{v_{n-1}} \lift{f}{v_{n-1}'}|_\approx+\epsilon.
\] 
Now Lemma \ref{lmm:TrLength:shrinks} implies:

\begin{align*}
 | \alpha_{n}|_{\approx} &\leq\frac{1}{\lambda} | \alpha_{n-1}|_{\approx} + \frac{1}{\lambda} | \beta_{v_{n-1}}|_{\approx} + 2C+\epsilon \leq \frac{1}{\lambda} (| \alpha_{n-1}|_{\approx} + K) + 2C+\epsilon \\& < \frac{M+K+ (2C+\epsilon)\lambda}{\lambda} < M.
\end{align*}
 
%


By Lemma \ref{lmm:TrLength:fin many},  there are only finitely many homotopy classes $[\alpha_n]$.  And since   $\alpha_n$ depends only on the homotopy type of $\alpha_{n-1}$, the sequence $\alpha_n$ is eventually periodic.  By shifting, we assume that $\alpha_n$, as well as $v_n$,  is periodic, say with a period $p$. For $n\le 0$, we define $\alpha_{n}$ to be $\alpha_{n+kp}$, where $p\gg 0$ is sufficiently big.

For $n\ge 1$, define $\ell_n$ to be the unique lift of $\ell_{\alpha_{1-n}(1)}$ under $f^{n-1}$ starting where $\alpha_{1-n} \lift{f^{n-1}}{x}$ ends. (such lift is constructed using the corresponding lift of the pseudoisotopy $H_{\alpha_{1-n} (1)}$).  Then $\ell_1\#\ell_2\#\dots$ is a required periodic bubble ray.
\end{proof}

Let $p$ be an eventual period of the sequence $f^n(\ell_n)$. This means that $f^p(\ell_1\#\ell_2\#\dots )$ coincide with $\ell_1\#\ell_2\#\dots$ in a small neighborhood of $x$ because $x$ is disjoint from $\widetilde G^{(n)}_i$.
  
 Let $q$ be the period of $x$; note that $q\mid p$. Let us consider the fundamental torus $T$ of $f^q$ at $x$; i.e. if $U$ is a closed small topological disk around $x$ such that $f(U)\Supset U$, then $T$ is obtained from $f(U)\setminus U$ by gluing along $f\mid \partial U$. Consider a simple arc \[\widetilde b\subset \overline {f(U)\Supset U}\] connecting a point $y\in \partial U$ to its image $f(y)\in \partial f(U)$. Then $\partial U$ and $\widetilde b$ project to simple closed curves $a,b\subset T$ generating the fundamental group $\pi_1(T)\simeq \Z^2$.

 Let us next project \[\bigcup _{k\in \{0,1,\dots, p/q-1\}} f^{kq}(\ell_n\#\ell_{n+1}\#\dots)\] to $T$; we obtain a finite graph $H\subset T$. There is a simple closed curve $\gamma\subset H$ such that $\gamma$ is not homotopic to $a$. Write $\gamma = na+mb$ in $\pi_1(T)$ with $b\not=0$. Lifting $\widetilde \gamma$ back to $S^2$ we obtain simple pairwise disjoint arcs \[\widetilde \gamma_0,\widetilde\gamma_1,\dots, \widetilde\gamma_{|m|-1}\] emerging from $x$ such that $f^q$ cyclically permutes $\widetilde\gamma_j$ and
   \[f^q\left( \bigcup_j\widetilde\gamma_j\right)\subset  \bigcup_j\widetilde\gamma_j \cup G_i^{(n)}\] for a sufficiently big $n$. Observe also that every $\widetilde \gamma_j$ is disjoint from every $ G_j^{(n)}\not= G_i^{(n)}$.  Adding the orbit of $\widetilde \gamma$ to $G^{(n)}$ as well as the orbit of $x$, we obtain a new $0$-entropy graph $\widetilde G$ such that $\widetilde G_i$ contains $x$.

\end{proof}

\subsection{Intersections and combinations of clusters}
Let $X$ be a compact subset of $S^2$. We say that a curve $\alpha$ is \emph{essentially} in $X$ if for every neighborhood $U$ of $X$, the curve $\alpha$ is homotopic (rel endpoints) to a curve in $U\setminus A$.   

Suppose $K_i, K_j$ are periodic clusters; we allow $i=j$. Consider two arcs $\ell_1,\ell_2$ of $S^2\setminus A$ connecting $G_i$ and $G_j$. We say that $\ell_1$ and $\ell_2$ are homotopic rel $G$ if there are arcs $\alpha$ and $\beta$ such that
\begin{itemize}
\item $\alpha$ is essentially in $G_i$;
\item $\beta$ is essentially in $G_j$; and
\item $\alpha\#\ell_1\#\beta$ is homotopic to $\ell_2$
\end{itemize}
 up to changing the orientations of $\ell_1,\ell_2$.

A \emph{Levy arc between $G_i$ and $G_j$} is an arc $\ell\subset S^2\setminus A$ connecting $ G_i$ and $ G_j$ such that $\ell$ is periodic up to homotopy rel $G$: a certain lift $\tilde \ell$ of $\ell$ is homotopic to $\ell$ relative $G^n$.

\begin{lemma}\label{lem: levy arcs between clusters}
Up to enlarging the initial graph $G$ we may assume that the following holds. 

Suppose $K_i\not=K_j$ are two periodic clusters generated by $G$.  If $K_i\cap K_j\not=\emptyset,$ then there is a Levy arc between $G_i$ and $G_j$. There are only finitely many Levy arcs between $G_i$ and $G_j$ (up to homotopy).  If there is a Levy arc between $G_i$ and $G_j$, then there is a periodic point in $K_i\cap K_j$.

Every Levy arc between $G_i$ and $G_i$ is essentially in $G_i$. 
\end{lemma}
\begin{proof}  We assume that every component of $G$ is non-trivial. By replacing $G$ with $G^{(m)}$ for a sufficiently big $m$ and by identifying indices in $I$, we can assume that $G_i^{(n)}\not= G_j^{(n)}$ if $i\not= j$.

There is a constant $M$ such that for every point in $x\in \Jul_i$ there is an arc $\alpha_i(x)$ that is essentially in $K_i$ such that the length of $\alpha_i(x)$ is less than $M$ (see Section~\ref{ss:B exp maps}). 
 
By expansion, every lift of such $\alpha_i(x)$ has length less than $M$. 

For every $y\in K_i\cap K_j$ define $\gamma_y$ to be  $\alpha_i(y)\# \alpha^{-1}_i(y)$. Set $N\coloneq \{\gamma_y\mid y\in K_i\cap K_j\}$; this is a finite set up to homotopy rel $G$. We denote by $t$ the cardinality of $N$. It follows that $\gamma_{f^{m(t+1)}(y)}$ is a Levy arc between $G_i$ and $G_j$,  for the iterate $f^m$ that returns both $K_i$ and $K_j$ to themselves. This proves the first claim. 

By expansion, there are at most finitely many Levy arcs between $G_i$ and $G_j$.

Let $\ell$ be a Levy arc between $G_i$ and $G_j$. By definition, $\ell$ is homotopic to $\alpha_1\# \tilde \ell_1\#\beta_1$, where $\tilde \ell_1$ is a lift of $\ell$ under $f^n$, and $\alpha_1, \beta_1$ are curves that essentially in $G_i,G_j$ respectively. Lifting this homotopy under $f^{nk}$, we obtain that $\ell$ is homotopic to $\alpha_1\#\dots\#\alpha_k \tilde \ell_k \#\beta_k \#\dots  \#  \beta_1$, where $\alpha_i, \beta_i,\tilde \ell _i$ are lifts of   $\alpha_{i-1}, \beta_{i-1},\tilde \ell _{i-1}$ under $f^n$. By expansion, $\alpha_1\#\dots\#\alpha_k\#\dots$ lands at a periodic point in the intersections $K_i\cap K_j$.

There are at most finitely many Levy arcs between $G_i$ and $G_i$. Every such arc is realized as a concatenation of periodic bubble rays; such concatenation can be added to $G_i$, using the argument of Lemma~\ref{lem:ext 0 entr graph}. This finishes the proof of the lemma.
 \end{proof}

\begin{rem}\label{rem:CombinGraphs}
Note that Lemmas \ref{lem:bubble ray land} and  \ref{lem: levy arcs between clusters} allows us to ``combine'' intersecting clusters $K_i$ and $K_j$, that is, construct a new $0$-entropy graph $\widetilde G$ such that the union $K_i\cup K_j$ would be inside a cluster generated by $\widetilde G$.
\end{rem}

\begin{definition}\label{def: weak-span}
Let $f\colon (S^2,A)\selfmap$ be a B\"ottcher expanding map.  Suppose that $G$ is a planar embedded $f$-invariant graph. We say that
\begin{enumerate}
\item $G$ is \emph{$A$-spanning} if $G$ is connected and $A\subset G$;
\item $G$ is \emph{weakly $A$-spanning} if $G$ is connected and each face of $G$ contains at most one point in $A$. 
\end{enumerate}
\end{definition}

\begin{lemma}[Spanning vs weakly spanning graphs]\label{lem: weakly span ext}
Let $f\colon (S^2,A)\selfmap$ be a B\"ottcher expanding map and $G$ be a planar embedded $0$-entropy $f$-invariant graph. If $G$ is weakly $A$-spanning then there is a planar embedded $0$-entropy $f$-invariant graph $\widetilde G\supset G$ that is $A$-spanning. 
\end{lemma}
\begin{proof} 

Since $G$ is weakly $A$-spanning,  $f^{-1}(G)$ is connected. Indeed, each face $U$ of $G$ is a topological disk containing at most one marked point, thus each component of $f^{-1}(U)$ is a topological disk as well. Hence, $f^{-1}(G)$ is connected. 
  
Let $K$ be the cluster of $G$.  By above,  $f^{-1}(K) = K$.  Let $a\in A \setminus G$. If $a\in K$, we are done using Lemma \ref{lem:ext 0 entr graph}.  Otherwise,  by expansion, $a$ must be in the Fatou set of $f$.  Again, by expansion, $\partial F_a\cap K \neq 0$. Lemma \ref{lem: levy arcs between clusters} implies that there is a (pre)periodic point $p$ in $\KC\cap \partial F_a$.  By Lemma \ref{lem:ext 0 entr graph} we can extend $G$ to a $0$-entropy graph that would contain $p$ and the internal ray in $F_a$  landing in $p$. 

Note that we are essentially in the setup of Theorem \ref{thm:Form amalg}: $f$ is a formal amalgam of $f$ with the power maps for the Fatou components $F_a$ of each $a\in (A\setminus G) \cap \Fat(f)$. 
\end{proof}

\subsection{Maximal clusters}\label{ss: clusters}

In the following, we fix a B\"ottcher expanding map $f\colon (S^2,A)\selfmap$,  where $A$ is the full preimage of the postcritical set.  The goal of this subsection is to define \emph{maximal clusters} of touching Fatou components of $f$.  The definition is based on an iterative construction (with finitely many steps). 

Step $0$:  We denote by $\mathbf{F}^{(0)}$ the set of all Fatou components of $f$, which we will call the \emph{pre-clusters of level $0$}. The set $\mathbf{K}^{(0)}= \{\overline{F^0}\colon F^0 \in \mathbf{F}^{(0)}\}$ of the closures of the Fatou components is called the set of \emph{clusters of level $0$}.

Step $n$: Suppose we defined the set $\mathbf{K}^{(n-1)}$ of clusters of level $n-1$,  or simply \emph{$(n-1)$-clusters}, where $n\geq 1$.  Consider an equivalence relation $\sim_{\mathbf{K}, n-1}$ on $\mathbf{K}^{(n-1)}$ defined as follows:
two $(n-1)$-clusters $X^{n-1},Y^{n-1} \in \mathbf{K}^{(n-1)}$ are said to be equivalent if there exists a sequence $K_0^{n-1} = X^{n-1}, K_1^{n-1}, \dots, K_m^{n-1}=Y^{n-1}$ of $(n-1)$-clusters such that $K_{j-1}^{n-1}\cap K_j^{n-1}\neq \emptyset$ for each $j=1,\dots, m$. We will call such a sequence of clusters an \emph{$(n-1)$-chain} of length $m$. 
The set $$\bigcup_{K^{n-1} \in [X^{n-1}]_{\sim_{\mathbf{K}, n-1}}} K^{n-1}$$ 
is called a \emph{pre-cluster of level $n$}.   We denote by $\mathbf{F}^{(n)}$ the set of all pre-clusters of level $n$.  Then the set $\mathbf{K}^{(n)}= \{\overline{F^{n}}\colon F^n \in \mathbf{F}^{(n)}\}$ is called the set of \emph{clusters of level $n$}. 

The following statement follows from Lemmas \ref{lem:bubble ray land} and  \ref{lem: levy arcs between clusters} (see Remark~\ref{rem:CombinGraphs}).

\begin{lemma}
\label{lem:graphs in clusters}
For every cluster $K^n\in \mathbf{K}^{(n)}$ and every finite set of eventually periodic points $S \subset K^n$, there is a forward-invariant graph $G$ such that the points of $S$ are within the same connected component of $G$.
\end{lemma}

Recall that a B\"{o}ttcher expanding map $f\colon (S^2,A) \selfmap$ is called a \emph{crochet map} if there is a connected forward-invariant zero-entropy graph $G$ containing $A$ (i.e., $G$ is an $A$-spanning graph, see Definition \ref{def: weak-span}. 

The following easily follows definitions and Lemma \ref{lem:graphs in clusters}.

\begin{lemma}
\label{lem:crochet-iteration}
Let $f\colon (S^2,A) \selfmap$ be a B\"{o}ttcher expanding map. Then the following are equivalent
 \begin{itemize}
   \item $f$ is crochet;
   \item an iterate $f^m$ is crochet for some $m\ge1$;
   \item $K^n_i=S^2$ for some $n$-cluster $K^n_i \in \mathbf {K^{(n)}}$ with $n\ge 1$.
 \end{itemize}
\end{lemma}

The pre-clusters and clusters satisfy the following invariance properties.

\begin{lemma}\label{lem: images of clusters} Let $F^n\in \mathbf{F}^{(n)}$ be a pre-cluster of level $n\geq 0$ and let $K^n=\overline{F^n}$ be the corresponding $n$-cluster.  Then the following statements are true.
\begin{enumerate}[label=\text{(\roman*)},font=\normalfont,leftmargin=*]
\item\label{lem: images of clusters:1} $f(F^n) \in  \mathbf{F}^{(n)}$ and $f(\overline{F^n})=f(K^n) \in  \mathbf{K}^{(n)}$.  In particular, the image of an $n$-cluster under $f$ is an $n$-cluster.
\item\label{lem: images of clusters:2} Each connected component of $f^{-1}(F^n)$ is a pre-cluster of level $n$.
\item\label{lem: images of clusters:3} Each connected component of $f^{-1}(K^n)$ is a finite union of $n$-clusters. If $F^n=K^n$, then each connected component of $f^{-1}(K^n)$ is an $n$-cluster.
\end{enumerate}
\end{lemma}

\begin{proof}
We proceed by induction on $n$. The base case $n=0$ is easily to be seen.

We start with \ref{lem: images of clusters:2}. Let $F'$ be a connected component of $f^{-1}(F^n)$. Then $f\colon F'\to F^n$ is a branched covering, and we can lift $n$-chains from $F^n$ to $F'$. More precisely, for every pair $x,y\in F'\setminus A$, we can select a curve $\gamma\subset S^2\setminus A$ connecting $x,y$ such that $\gamma$ can be homotoped into any small neighborhood of $F'$ (i.e., $\gamma$ is essentially in $\overline{F'}$). Then $f(\gamma)$ can also be homotoped into any small neighborhood of $F^n$. We can select a finite chain of $(n-1)$ pre-clusters $F^{n-1}_1,\dots ,F^{n-1}_t$ (possibly with repetitions) such that for every $\varepsilon>0$, we can homotope $f(\gamma)$ into a concatenation $\beta_1\#\beta_2\#\dots \#\beta_t$ such that every $\beta_i$ is in within the $\varepsilon$ neighborhood of $F^{n-1}_i$ for every small $\varepsilon>0$. (Roughly, the $(n-1)$-clusters $\overline{F^{n-1}_i}$ have an intersection pattern prescribed by the $\beta_i$.) By induction, every connected component of $f^{-1}(F^{n-1}_i)$ is a pre-cluster. Lifting the $F^n_i$ according to $\beta_1\#\beta_2\#\dots \#\beta_t\lift{f}{x}$, and taking the closure, we obtain a chain of $(n-1)$-clusters connecting $x,y$. So, $F'$ is a subset of a pre-cluster of level $n$. Since, by induction, $(n-1)$-chains are sent over to $(n-1)$-chains by $f$, it follows that $F'$ is actually a pre-cluster of level $n$, i.e., \ref{lem: images of clusters:2} follows.


Again, by the induction assumption, $f(F^n)$ is a subset of a pre-cluster $F^n_0$ of level $n$. Combined with~\ref{lem: images of clusters:2}, we obtain $f(F^n)= F^n_0$. And taking the closure, we obtain $f(K^n)=\overline{F^n_0}$ -- this is~\ref{lem: images of clusters:1}.

Finally, property~\ref{lem: images of clusters:3} follows from \ref{lem: images of clusters:2} because $f$ has a finite degree. 
\end{proof}

We claim that the set of $n$-clusters eventually stabilizes, that is, $\mathbf{K}^{(n)}= \mathbf{K}^{(n+1)}$ for some sufficiently large $n$.  Equivalently, this means that the $n$-clusters are pairwise disjoint for some sufficiently large $n$. For $a\in A$, let $K^n_a\in \mathbf {K^{(n)}}$ be the cluster containing $a$; if $a$ is not in any level-$n$ cluster, then $K^n_a\coloneq \emptyset$. 

\begin{theorem}\label{thm: max_clusters_exist}
There exists $N$ such that $\mathbf{K}^{(n)}= \mathbf{K}^{(n+1)}$ for each $n\geq N$. 

Moreover if $K^n_a = K^{n+1}_a$ for all $a\in A$, then $\mathbf{K}^{(n+1)}= \mathbf{K}^{(n+2)}$.
\end{theorem}

The $N$-clusters are then called the  \emph{maximal clusters of touching Fatou components} of $f$.

\begin{proof}
To prove the theorem, we consider $|A|+1$ sequences of relations on the set $A=f^{-1}(\post(f))$: $\{\sim_{n,in}\}_{n\geq 0}$ and $\{\sim_{n,sep,a}\}_{n\geq 0}$, where $a\in A$. The relation $\sim_{n,in}$ is the \emph{inclusion relation} induced by $\mathbf{K}^{(n)}$: two points $a,a'\in A$ are related under $\sim_{n,in}$ if $a$ and $a'$ belong to the same $n$-cluster.  The relation $\sim_{n,sep,a}$ is the \emph{separation relation} induced by $K^n_a$: two points $b,b'\in A$ are related under $\sim_{n,sep,a}$ if $b$ and $b'$ belong to the same connected component of $S^2\setminus K_a^n$; i.e., $\sim_{n,sep,a}$ is a set of pairwise disjoint subsets of $A$, where $b$ is absent in $\sim_{n,sep,a}$ if $b\in K^n_a$. Since (pre-)clusters are nondecreasing it immediately follows that $\sim_{n,in}\subseteq \sim_{n+1,in}$ and $\sim_{n,sep,a} \supseteq \sim_{n+1,sep,a}$. Our key claim:

\begin{lemma}
\label{lem: eq_relation_strict}
Suppose that $\sim_{n-1,in} = \sim_{n+1,in}$ and $\sim_{n-1,sep,a} = \sim_{n+1,sep,a}$ for all $a\in A$. Then $\mathbf {K^{(n)}}=\mathbf {K^{(n+1)}}$.
\end{lemma}
\begin{proof}
Denote by $\widetilde K^n_a$ the connected component of $f^{-1}\circ f (K^n_a)$ containing $K^n_a$. Consider the following set and the associated pseudo-multicurve, see Section~\ref{sss:psMC}: 
\[ K^n_A \coloneq \bigcup_{a\in A} K^n_a,\sp \sp\sp \sp \CC^n\coloneq \psMC_A\left( K^n_A\right)\]

\begin{claim}
The pseudo-multicurves $\CC^{n-1}$ and $\CC^n$ are homotopic rel $A$. 
\end{claim}
\begin{proof}
 Since $\sim_{n-1,in} = \sim_{n,in}$ we have 
 \[ K^{n-1}_a= K^{n-1}_b\sp\sp\text{ if and only if }\sp\sp  K^n_a= K^n_b.\]
 (If $K^{n-1}_a\not= K^{n-1}_b$ but $K^n_a= K^n_b$, then $\sim_{n-1,in} \subsetneq \sim_{n,in}$.) Since $\sim_{n-1,sep,a} = \sim_{n,sep,a}$, every component $U$ of $S^2\setminus K^{n-1}_a$ with $U\cap A\not=\emptyset$ contains a unique component $U'$ of $S^2\setminus K^n_a$ such that $U'\cap A=U\cap A$. Since $U'\subset U$ are open topological disks, $\psCurve( K^{n-1}_a\mid U)$ is homotopic to $\psCurve( K^n_a\mid U')$. We obtain $\CC^{n-1}=\CC^n$ rel $A$.
\end{proof}

It follows from the claim that $f^{-1}(\CC^{n-1})=f^{-1}(\CC^n)$ rel $f^{-1}(A)$. Writing $ \widetilde K^n_{f^{-1}(A)} \coloneq f^{-1}\big(K^n_A\big)$, we have by~\eqref{eq:sss:psMC}: 
\begin{equation}
\label{eq:prf:main claim}
\psMC_{f^{-1}(A)}\left(\widetilde K^{n-1}_{f^{-1}(A)}\right)=\psMC_{f^{-1}(A)}\left(\widetilde K^n_{f^{-1}(A)}\right).
\end{equation}

Let us now assume that $\widetilde K^{n}_a\not=K^n_a$. Then $K^{n}_a$ intersects another cluster in $\mathbf {K^{(n)}}$; we write this cluster as $K^n_x$, where $x\in f^{-1}(A)\setminus A$ is a point in $K^n_x$.

Since $K^n_a\not=K^n_x$, the sets $\widetilde K^{n-1}_a$ and $\widetilde K^{n-1}_x$ are disjoint because they are in $F^{n}_a$ and $F^{n}_x$ respectively. Let $U$ be the connected component of $S^2\setminus \widetilde K^{n-1}_a$ containing $\widetilde K^{n-1}_x$. Then $\ell\coloneq \psCurve_{f^{-1}(A)}(\widetilde K^{n-1}_a\mid U)$ is a non-trivial curve because it separates $a$ and $x$. By construction, $\ell$ is in the left pseudo-multicurve of~\eqref{eq:prf:main claim} but not in the right. This is a contradiction. Therefore, $K^n_a=F^{n+1}_a=K^{n+1}_a=F^{n+2}_a=K^{n+2}_a$ for all $a\in A$ and Lemma~\ref{lem: images of clusters}\ref{lem: images of clusters:3} finishes the proof.
\end{proof}

The first part of Theorem \ref{thm: max_clusters_exist} now immediately follows from Lemma \ref{lem: eq_relation_strict} due to the finiteness of~$A$. The second part follows from Lemma~\ref{lem: images of clusters}\ref{lem: images of clusters:3}. 
\end{proof}

\section{Crochet algorithm}
\label{sec: existence}



We start by introducing the following technical definition, having its origins in the Crochet Algorithm (see Section \ref{ss: crochet_algo_intro}).

\begin{definition}[Pre-crochet multicurves]
\label{def:pre croch mult}
Let $f\colon (S^2, A, \CC)\selfmap$ be a B\"{o}ttcher expanding map such that each small map in the decomposition wrt.\ $\CC$ is either crochet or Sierpi\'{n}ski. Suppose that the set $I$ parametrizes the small spheres of $(S^2,A,\CC)$ and $f\colon I \selfmap$ provides the dynamics on the small spheres (see Section \ref{ss: dec and amal}). 

Let us denote by $I_\bullet \subset I$ the subset parametrizing all small spheres of $(S^2, A, \CC)$ induced by the small crochet maps. Then $I_\circ \coloneq I\setminus I_\bullet$ parametrizes the small spheres arising from Sierpi\'{n}ski maps. The small spheres in $I_\bullet$ and $I_\circ$ will be called (\emph{small}) \emph{crochet spheres} and \emph{Sierpi\'{n}ski spheres}, respectively.

The invariant multicurve $\CC$ is called \emph{pre-crochet} if there exists a partition 
\[I_\bullet = I^1_\bullet \sqcup \dots \sqcup I^n_\bullet\] into forward-invariant sets such that the following two conditions are satisfied:
\begin{enumerate}[label=\text{(\roman*)},font=\normalfont,leftmargin=*]
\item The invariant multicurve $\CC$ is generated by the multicurve $\{\partial S_z$, $z\in I_\bullet\}$ consisting of the boundary curves of all crochet spheres; see Section \ref{ss: inv_multicurves}. 
\item For every $k\in \{1,\dots, n\}$ and every periodic sphere $\widehat S$ in $I^k_\bullet$, the first return map $\widehat f\colon \widehat S \selfmap$ is \emph{$\CC^k$-vacant}, that is, it admits a weakly spanning $0$-entropy connected invariant graph that does not pass through the Fatou components induced by $\CC\,\setminus\, \CC^k$ in $\partial \widehat S$, where $\CC^k$ 
is the invariant multicurve generated by the boundaries of small crochet spheres in $I^1_\bullet\sqcup \dots \sqcup I^{k-1}_\bullet$. 
\end{enumerate}
\end{definition}


The goal of this section is to show that each B\"ottcher expanding map $f\colon (S^2,A)\selfmap$ with non-empty Fatou set posses a pre-crochet multicurve.  In fact, we will show that the Crochet Algorithm (see Section~\ref{ss: crochet_algo_intro}) constructs such a curve (after running the first three steps).

\subsection{Gluing of crochet maps}\label{ss: gluing_crochet}
We start by recording the following fact that follows from Theorem \ref{thm:Form amalg}.

\begin{proposition}\label{prop: gluing crochet}
Let $f\colon (S^2,A,\CC)\selfmap$ be a B\"ottcher expanding map, and $\CC$ be a pre-crochet multicurve. If all small maps in the decomposition of $f$ rel.\ $\CC$ are crochet and $\CC$ does not contain any bicycle, then $f$ is crochet. 

Furthermore, if each small map in the decomposition rel.\ $\CC$ is vacant with respect to a subset $V\subset A^\infty$, then $f$ is vacant to $V$ as well. 
\end{proposition}
\begin{proof} 
By Lemma~\ref{lem:crochet-iteration}, it is sufficient to prove that an iterate of $f$ is crochet. By passing to an iteration, let us assume that all periodic spheres of $f\colon (S^2,A,\CC)\selfmap$ have period one and all periodic curves in $\CC$ have period one (the latter can be achieved due to the no-bicycles assumption). Suppose that $I_\bullet=I_\bullet^1\sqcup\dots \sqcup I_\bullet^n$ is the partition of the set $I_\bullet$ parametrizing the small spheres of $(S^2,A,\CC)$ that satisfies the conditions in the definition of the pre-crochet multicurve. Below we will apply an induction on the number of small fixed spheres. The induction step is in described in the following lemma. 

\begin{lemma}
\label{lem:glunig crochet:induction}
Under the assumption of Proposition~\ref{prop: gluing crochet}, let us assume that 
\begin{itemize}
\item $\CC$ contains a unique periodic curve $\gamma$ which has period $1$ (i.e., all other curves are strictly preperiodic);
\item $f\colon (S^2,A,\CC)\selfmap$ has exactly two periodic spheres $S_1,S_2$; the period of each $S_i$ is $1$ (i.e., $\gamma$ is the common boundary of $S_1,S_2$);
\item $f_1$ is crochet and is vacant with respect to the Fatou components induced by $V_1\cup \{\gamma\}$, where $V_1$ is a subset of $A^\infty_{1}\coloneq A^\infty \cap \filled_{\widetilde S,1}$, see Section~\ref{sss:FA:n};
\item $f_2$ is crochet and is vacant with respect to the Fatou components induced by $V_2$, where $V_2$ is a subset of $A^\infty_{2}\coloneq A^\infty \cap \filled_{\widetilde S,2}$.
\end{itemize}

Then $f$ is crochet and is vacant with respect to the Fatou components generated by \[V\coloneq V_1\cup V_2\cup \big(A^\infty\setminus (A^\infty_{1}\cup A^\infty_{2})\big).\]
\end{lemma}

See Lemma~\ref{lem:A:relat f f_blowup} for how $A^\infty_1,A^\infty_2$ are related to $A$.
\begin{proof}
Let $f_\text{FA}\colon (S^2,A,\CC)\selfmap$ be a formal amalgam associated with $f\colon (S^2,A,\CC)\selfmap$, see Lemma~\ref{lem:FormAmalg}. We denote by $\tau$ the induced semiconjugacy from $f_\text{FA}$ to $f$. We will show that appropriate clusters of $f_1,f_2$ get glued and produce a required cluster of $f$.

Let $\filled^\text{cr}_1 \subset \filled_{\widetilde S,1}$ be the set of points that do not escape into the Fatou components induced by $V_1\cup \{\gamma\}$. By assumption, $\filled^\text{cr}_1$ is a (crochet) cluster: it supports a weakly spanning graph for $f_1$.

Let $\filled_2 \subset \filled_{\widetilde S,2}$ be the set of points that do not escape into the Fatou components induced by $V_2\cup \{\gamma\}$. Note that $\filled_2$ needs not to be a (crochet) cluster for $f_2$. Since $\gamma$ is a primitive unicycle, $\tau(\filled^\text{cr}_1)$ contains the $\tau$-image $\Theta$ of the boundary of the Fatou component induced by $\gamma$ in $\filled_{2}$. Therefore, every lift of $\Theta$ under $f\colon \tau(\filled_2)\selfmap$ is glued with an appropriate lift of $\tau(\filled^\text{cr}_1)$ -- this is the ``tuning'' of $\filled^\text{cr}_1$ and $\filled_2$. By the vacancy assumption on $f_2$, we obtain a (crochet) cluster for $f$ containing $\tau(\filled^\text{cr}_1)\cup \tau(\filled_2)$ and separating the remaining points in $A$; the remaining set contains $V$. Therefore, $f$ is vacant rel $V$. The lemma is proven.
\end{proof}

Let us finish the proof of Proposition~\ref{prop: gluing crochet}. Select an ``anti-primitive'' unicycle $\{\gamma\}\subset \CC$; i.e.~the multicurve $\CC_\gamma$ generated by $\gamma$ does not contain any other unicycle. Then the decomposition of $f\colon (S^2,A,\CC_\gamma)\selfmap$ has exactly two periodic spheres $S_1$, $S_2$; these spheres have period one. Let $f_i\colon \widehat S_i\selfmap$ be the associated small maps, and let $\CC_i\subset \widehat S_i$ be the multicurve induced by $\CC$. Then $\CC_i$ is a pre-crochet multicurve for $f_i$ (indeed, the disjoint union $I_\bullet=I_\bullet^1\sqcup\dots I^{n}_\bullet$ would provide the respective partition for the set of small spheres). Moreover, all small maps in $f_i\colon (\widehat S_i, \CC_i)\selfmap$ are $V_i$ vacant, where $V_i$ is induced by $V$. Finally, small sphere in either $f_1\colon (\widehat S_1, \CC_1)\selfmap$ or in $f_2\colon (\widehat S_2, \CC_2)\selfmap$ are vacant rel $\gamma$ because such a property holds for small spheres of $f\colon (S^2,A,\CC)\selfmap$ bordering $\gamma$. Assume that $\gamma$ is vacant for small maps of $f_1$. 

By induction assumption, $f_1$ is crochet and is vacant rel $V_1\cup \{\gamma\}$ and $f_2$ is crochet and is vacant rel $V_2$. Lemma~\ref{lem:glunig crochet:induction} finishes the proof.
\end{proof}
 
\hide{\begin{lemma}\label{lem: gluing crochet}
Let $f\colon (S^2,A,\CC)\selfmap$ be a B\"ottcher expanding map and $\PC$ be a primitive unicycle of $\CC$.  Also, let $\CC'\subset \CC$ be the invariant multicurve\footnote{ Note that $\CC'$ is empty, if $\CC$ has only one strongly connected component. }  generated by all the strongly connected components of $\CC$ except $\PC$. Suppose $\widehat S_1$ and $\widehat S_2$ are the two small spheres (in the decomposition wrt.\ $\CC$) with $\gamma\in \PC$ on the boundary.  Denote by $\widehat S$ the unique small sphere wrt.\ $\CC'$ that contains both $\widehat S_1$ and $\widehat S_2$. 
\begin{enumerate}[label=\text{(\roman*)},font=\normalfont,leftmargin=*]
\item If both spheres $\widehat S_1$ and $\widehat S_2$ are Sierpi\'{n}ski, then $\widehat S$ is Sierpi\'{n}ski sphere. 
\item Suppose both spheres $\widehat S_1$ and $\widehat S_2$ are crochet. If one of them is $\CC[\gamma]$-vacant, where $\CC[\gamma]$ is the invariant multicurve induced by $\{\gamma\}$, then $\widehat S$ is a crochet sphere.

\end{enumerate}
\end{lemma}}

Let $f\colon (S^2, A)\selfmap$ be a B\"{o}ttcher expanding map and $\CC$ be a pre-crochet multicurve. Suppose $\{\PC_k\}$ is the collection of all primitive unicycles in $\CC$ separating small crochet spheres of $(S^2, A, \CC)$. Consider the invariant multicurve $\CC'\subset \CC$ generated by all strongly connected components of $\CC$ except the primitive unicycles $\{\PC_k\}$, that is, glue together the clusters of small crochet spheres in $(S^2, A,\CC)$ separated by the unicycles $\PC_k$. We iterate this process until there are no more primitive crochet unicycles. We will call this procedure \emph{iterative elimination of all primitive crochet unicycles} from $\CC$.

\begin{definition}[Crochet multicurve]
\label{def:croch mult}
Let $f\colon (S^2, A)\selfmap$ be a B\"{o}ttcher expanding map and $\CC$ be a pre-crochet multicurve. Consider the invariant multicurve $\DD\subset \CC$ that is generated by the boundaries of Sierpi\'{n}ski small spheres of $(S^2,A,\CC)$ together with bicycles in $\CC$. 
Then the multicurve $\DD$ is obtained by the iterative elimination of all primitive crochet unicycles from $\CC$ and we call it a \emph{crochet multicurve}.

\end{definition}

By definition, a crochet multicurve does not have any primitive crochet unicycles. It would follow that each B\"{o}ttcher expanding map $f\colon (S^2, A, \CC)\selfmap$ with non-empty Fatou set has a unique crochet multicurve (up to isotopy), see Remark~\ref{rem: crochet_multicurve_unique}.

The next lemma follows from the definition and Proposition \ref{prop: gluing crochet}.

\begin{lemma}\label{lem:cro-mult-decomp}
Let $f\colon (S^2,A)\selfmap$ be a B\"{o}ttcher expanding map and $\DD$ be a crochet multicurve. Then each small map in the decomposition wrt.\ $\DD$ is either crochet or Sierpi\'{n}ski. Moreover, if $\DD$ is non-empty then either there is at least one Sierpi\'{n}ski small map in the decomposition or $\CC'$ has a primitive bicycle. 
\end{lemma}


\subsection{Iterative step} 

In the following,  let $f\colon (S^2,A)\selfmap$ be a B\"ottcher expanding map with a non-empty Fatou set, where $A$ contains the full preimage of the postcritical set. We refer to Section \ref{ss: clusters} for the ``cluster terminology''.

\begin{lemma}
\label{lem:iter step}
Let $\mathbf{K}^N=\{K_i\}$ be the collection of maximal clusters of $f$ (given by Theorem \ref{thm: max_clusters_exist}). 
Then $\CC\coloneq \MC(K)$ is an invariant multicurve, where $K=\bigcup_{i} K_i$ is the union of all the clusters.  

Assume the multicurve $\CC\not= \emptyset$. Consider a non-trivial periodic $K_i$, which means that $\MC(K_i)$ is non-empty. Then $K_i$ is within a periodic small sphere $S_i$ of $(S^2,A, \CC)$. The first return map on $\widehat{S_i}$ is crochet.

\end{lemma}
\begin{proof} Since $K\supset A\supset P_f$, $f\colon S^2\setminus f^{-1}(K)\to S^2\setminus K$ is a covering map.  Let $\CC=\MC(K)$.  Lemma \ref{lem: images of clusters} immediately implies that $\CC\subset f^{-1}(\CC)$, as usual, up to homotopy.  Conversely, $f^{-1}(\CC)\subset\CC$ after removal of all duplicate and peripheral curves.  Thus,  $\CC=\MC(K)$ is an invariant multicurve.

Clearly, each non-trivial periodic cluster $K_i$ is within a periodic small sphere $S_i$ of $(S^2,A, \CC)$. Since $\widehat{S_i}$ is marked by $A\cup \CC$, each complementary component of $\widehat{S_i}\setminus K_i$ contains exactly one marked point. Theorem \ref{thm:Form amalg} and Lemma \ref{lem:crochet-iteration} now imply that the first return map on $\widehat S_i$ is crochet.
\end{proof}

\begin{lemma}[Stopping criterion]\label{lem: trivial curve}
Let $\mathbf{K}^N=\{K_i\}$ be the collection of maximal clusters of $f$ and $\CC\coloneq \MC(\bigcup_i K_i)$ be the respective invariant multicurve, as in Lemma \ref{lem:iter step}. Then $\CC$ is empty if and only if $f$ is a crochet or a Sierpi\'{n}ski map.
\end{lemma}
\begin{proof}

Let us suppose that the multicurve $\CC$ is empty. If $K$ consists of a unique cluster, then $f$ is crochet by Lemma \ref{lem:crochet-iteration}. Otherwise, let $K=\bigcup_{i\in I} K_i$. Suppose that some cluster $K_i$ contains at least two marked points. Since $\CC$ is empty, each connected component $U$ of $S^2\setminus K_i$ contains at most one point form $A$. Furthermore, since $\mathcal{F}_f\cap U \neq \emptyset$, $a$ must be in $\mathcal{F}_f$. Thus, there is a an internal ray joining $a$ and $K_i$ which would contradicts maximality of $K_i$.

Consequently, each cluster $K_i$ must have at most one marked point. Again,  there is no Levy arc between any two Fatou points in $A$.  Also, there is no periodic self-arc $\alpha$ at any Fatou point $a$ in $A$. Indeed, otherwise each component of $S^2\setminus \alpha$ contains a marked point $p$ and either $\CC$ is non-empty, or there is a Levy arc connecting $p$ and $a$ (providing a contradiction).  Thus, by the characterization from \cite{BD_Exp}, $f$ must be a Sierpi\'{n}ski map.





\end{proof}

\subsection{Crochet Algorithm}
\label{ss:iter_construction}
Consider a recursive procedure with the following step.

 \begin{itemize}
    \item Given a B\"ottcher expanding map $f\colon (S^2,A)\selfmap$ with $\mathcal{F}_f\neq\emptyset$, we extract the maximal clusters $K_i$ and the respective invariant multicurve $\CC\coloneq \MC(K)$, where $K=\bigcup_i K_i$, as in Lemma \ref{lem:iter step}. 
   \item If $\CC$ is non-empty, then we run the above step for the first return map of each periodic small sphere of $(S^2,A, \CC)$ that does not contain a cluster $K_i$.
   \item If $\CC$ is empty, we stop the recursive step. 
\end{itemize}


    

Since $A$ is finite, the above recursive procedure eventually stops (completing the first three steps of the Crochet Algorithm). Let $\widetilde{\CC}$ be the union of all the invariant curves constructed during the process. We define $\CC_\NL$ to be the set of representatives of the isotopy classes of curves in $\cup_{n\geq 0} f^{-n}(\widetilde{\CC})$, that is, $\CC_\NL$ is the invariant multicurve generated by $\widetilde{C}$.  The following lemma follows from the previous discussion.



\begin{proposition}\label{prop: pre-crochet-algo}
The multicurve $\CC_\NL$ constructed after running the first three steps of the Crochet Algorithm is a pre-crochet multicurve for $f\colon (S^2,A)\selfmap$.
\end{proposition}

\begin{proof}
By construction, $\CC_\NL$ is an invariant multicurve. Moreover, the first return map to each periodic small sphere of $(S^2,A, \CC_\NL)$ is either crochet or Sierpi\'{n}ski (Lemmas \ref{lem:iter step} and \ref{lem: trivial curve}). Suppose that $I_\bullet$ parametrizes all small crochet spheres of $(S^2, A, \CC_\NL)$ induced by the small crochet maps. Partition $I_\bullet$ according to the depth of the recursive step, that is, we write 
\[I_\bullet = I^1_\bullet \sqcup \dots I^n_\bullet,\]
where $I^k_\bullet$, $k\in\{1,\dots,n\}$, corresponds to the small crochet spheres induced by the small crochet maps obtained during the $k$th stage of the recursive procedure. It is now straightforward to check that the above partition satisfies the conditions of Definition \ref{def:pre croch mult}. The statement follows. 
\end{proof}

After the recursive construction of $\CC_\NL$ we do a reduction step that eliminates all primitive crochet unicycles, i.e., the primitive unicycles of $\CC_\NL$ with crochet maps on both sides. In fact, these may only appear at different stages of the recursive procedure (due to the maximality of extracted clusters at each stage). We iterate this reduction step until no more primitive crochet unicycles are left -- this is Step~4 of the Crochet Algorithm (see Section \ref{ss: crochet_algo_intro}).  The resulting multicurve $\CC_\cro$ would be a crochet multicurve by definition; it is generated by the boundaries of Sierpi\'{n}ski small spheres of $(S^2,A,\CC_\NL)$ together with the bicycles of $\CC_\NL$.  That is, we proved the following result. 

\begin{corollary}\label{cor: cro multicurve exists}
Let $f\colon (S^2,A)\selfmap$ be a B\"ottcher expanding map with non-empty Fatou set. Then the Crochet Algorithm produces a crochet multicurve $\CC_\cro$. 
\end{corollary}

The decomposition of $f\colon (S^2,A,\CC_\cro)\selfmap$ along the invariant multicurve $\CC_\cro$ is called the \emph{crochet decomposition} of $f$.  

\section{Cactoid maps} 
\label{sec: cactoid maps}

Let us recall that a \emph{cactoid} $X$ is a continuous monotone image of $S^2$. In general,  $X$ is a locally connected continuum composed of countably many spheres and segments pairwise intersecting in at most one point. We call a cactoid $X$ \emph{finite} if it is a finite CW-complex, i.e., if $X$ is composed of finitely many spheres and segments. The finite cactoids $X$ we work with will usually have a natural finite marking $A\subset X$ so that each component $S$ of $X\setminus A$ is either an open arc or a punctured sphere. We refer to the closure $\overline S$ as a \emph{small segment} of $(X,A)$ (or simply a small segment of $X$ if $A$ is understood) in the former case, and as a \emph{small sphere} in the latter case.

Consider a Thurston map $f\colon (S^2,A,\CC)\selfmap$.  
In this section, we discuss how to collapse $f\colon (S^2,A,\CC)\selfmap$ into a totally topologically expanding map $\bar f\colon X_\infty\selfmap$ on a cactoid $X_\infty$. More precisely, writing $\CC=\CC_\bullet \sqcup \CC_-$ and $I=I_\bullet \sqcup I_\circ$, where $I$ is an index set parametrizing small spheres of $(S^2,A,\CC)$, we will collapse curves in $\CC_\bullet$ into points,  annular neighborhoods of curves in $\CC_-$ into segments, small spheres of $I_\bullet$ into points, and small spheres of $I_\circ$ will project into small spheres of $X_\infty$. To obtain the expansion in the quotient, we will assume that small maps associated with $I_\circ$ are Sierpi\'{n}ski maps and $\CC_-$ is generated by bicycles. In the proof, we will construct a finite expanding model  $\bar f, \bar \iota \colon X_2 \rightrightarrows X_1$, and the desired map $\bar f\colon X_\infty\selfmap$ will be the inverse limit of $\bar f, \bar \iota \colon X_{n+1} \rightrightarrows X_n$.

\subsection{The quotient map $ (S^2,A,\CC)\to (X,\bar A)$} 
Consider a marked sphere $(S^2,A)$ with a multicurve $\CC$, and let $I$ be an index set parametrizing the small spheres of $(S^2,A,\CC)$. Fix decompositions  
\begin{equation}
\label{eq:coll data}
\CC=\CC_\bullet\sqcup \CC_-\sp\sp\text{ and }\sp\sp I=I_\bullet\sqcup I_\circ,
\end{equation}
which we call a \emph{collapsing data} for $(S^2,A,\CC)$. We will now describe a procedure collapsing $(S^2, A, \CC)$ into a finite marked cactoid induced by these decompositions.

For every $\gamma\in \CC$, let $\Gamma_\gamma$ be either $\gamma$ (i.e., a degenerate annulus) or a thickened closed annulus $\Gamma_\gamma$ around $\gamma$. In the second case, we foliate the annulus $\Gamma_\gamma$ by curves $\gamma_{t}, t\in [-1,1]$, isotopic to $\gamma$. If $\gamma=\Gamma_\gamma$, then we assume that $\gamma_t=\gamma$ for all $t$. We assume that
\begin{equation}
   \label{eq:gamma in Gamma}  \gamma\subsetneq \Gamma_\gamma\sp\sp\sp\text{ if }\sp \sp \gamma\in \CC_-. 
\end{equation}
Starting with Section~\ref{sss:MonotMapsCact}, we will require a stronger dynamically invariant condition~\eqref{eq:gamma in Gamma:2} instead of~\eqref{eq:gamma in Gamma}. 

Now we perform the following collapsing procedure on $S^2$:
\begin{enumerate}[label=\text{(\Alph*)},font=\normalfont,leftmargin=*]
\item\label{cond:collapse_A} for every $\gamma\in \CC_\bullet$, collapse $\Gamma_\gamma$ to a point; and
\item\label{cond:collapse_B} for every $\gamma\in \CC_-$, collapse $\Gamma_\gamma$ into a closed segment by collapsing every $\gamma_{t}$ (in the foliation of $\Gamma_\gamma$) into a point.
\end{enumerate}
After \ref{cond:collapse_A} and \ref{cond:collapse_B} we obtain a finite cactoid $X'$ whose small spheres and segments are naturally parametrized by $I=I_\bullet \sqcup I_\circ$ and $\CC_-$, respectively. Now we
\begin{enumerate}[label=\text{(\Alph*)},font=\normalfont,leftmargin=*]
\addtocounter{enumi}{2}
\item\label{cond:collapse_C} 
collapse every small sphere of $X'$ in $I_\bullet$ to a point.
\end{enumerate}
We denote by $X$ the resulting finite cactoid and by
\begin{equation}
\label{eq:quot map} \Pi_{\CC_\bullet, \CC_- , I_\bullet,I_\circ}\colon  (S^2,A,\CC) \to (X,\bar A)
\end{equation} the associated quotient map. Here, the pair $(X, \bar A)$ represents the cactoid $X$ marked by 
\[\bar A \coloneq \Pi_{\CC_\bullet, \CC_- , I_\bullet,I_\circ} \left(A\cup \CC_\bullet \cup \bigcup_{\gamma\in \CC}  \{\gamma_{-1} \cup \gamma_{1}\} \right).\]
We say that $\Pi_{\CC_\bullet, \CC_- , I_\bullet,I_\circ}$ and $(X,\bar A)$ are induced by the collapsing data \eqref{eq:coll data}.

The following observations are immediate from the construction: 
small spheres and segments of $(X, \bar A)$ are naturally parameterized by $I_\circ$ and $\CC_-$, respectively;
small spheres and segments intersect at marked points; every small segment connects two marked points in $\bar A$. 

The quotient cactoid $(X,\bar A)$ is unique in the following sense.

\begin{lemma}[Uniqueness of $(S^2,A,\CC)\to (X,\bar A)$]
\label{lem:uniq of X}
 Suppose 
\[\Pi_a\colon (S^2,A, \CC)\to (X_a,\bar A_a)\sp\sp\text{ and }\sp\sp \Pi_b\colon (S^2,A, \CC)\to (X_b,\bar A_b)\]
are two monotone maps realizing the collapsing data~\eqref{eq:coll data}.Then there is a homeomorphism $h\colon  (X_a,\bar A_a)\to  (X_b,\bar A_b)$ such that $h\circ \Pi_a$ and $\Pi_b$ are isotopic via a continuous path of monotone maps $p_t\colon (S^2,A,\CC)\to (X_b,\bar A_b)$ satisfying the above \ref{cond:collapse_A}, \ref{cond:collapse_B}, \ref{cond:collapse_C}.
\end{lemma}
\begin{proof}
   The cactoids $(X_a,\bar A_a)$ and $(X_b,\bar A_b)$ are parameterized by the collapsing data~\eqref{eq:coll data}. Therefore, there is a homeomorphism $h\colon  (X_a,\bar A_a)\to  (X_b,\bar A_b)$ respecting~\eqref{eq:coll data}. The fibers \[ [h\circ \Pi_a]^{-1}(x)\sp\sp\text{ and }\sp\sp \Pi^{-1}_b(x), \sp\sp x\in X_b\] are described in \ref{cond:collapse_A}, \ref{cond:collapse_B}, \ref{cond:collapse_C}; thus there is a homotopy of $p'_t$ of $(S^2,A)$ moving $[h\circ \Pi_a]^{-1}(x)$ into $\Pi^{-1}_b(x)$ for all $x$. The $p'_t$ induces a required homotopy $p_t$ of monotone maps.
\end{proof}

\subsection{Cactoid correspondences}
\label{ss:bar f}

Let $f\colon (S^2,A) \selfmap$ be a Thurston map and $\CC$ be a multicurve on $(S^2,A)$. Consider the covering 
\begin{equation}
\label{eq:preCactoid:f}
f\colon (S^2,f^{-1}(A),f^{-1}(\CC)) \to (S^2,A,\CC).
\end{equation}
We assume that the sets $I$ and $f^{-1}(I)$ parametrize the small spheres of $(S^2,A,\CC)$ and $(S^2,f^{-1}(A),f^{-1}(\CC))$, respectively. Given the collapsing data \eqref{eq:coll data}, let us set
\begin{equation}
\label{eq:coll data:lift}
f^{-1}(\CC_\bullet, \CC_- , I_\bullet , I_\circ)\coloneq (f^{-1}(\CC_\bullet), f^{-1}(\CC_-) , f^{-1}(I_\bullet),f^{-1}(I_\circ)).
\end{equation}
Clearly, \eqref{eq:coll data:lift} 
specifies a collapsing data for $(S^2,f^{-1}(A),f^{-1}(\CC))$.

The quotient map~\eqref{eq:quot map} induces a monotone equivalence relation $\sim$ on $(S^2,A,\CC)$.  Let $(X_1, \bar A_1)$ be the corresponding marked quotient cactoid.  We define $f^*(\sim)$ to be the (monotone) equivalence relation on $(S^2,f^{-1}(A),f^{-1}(\CC))$  whose equivalence classes are the connected components of the preimages of the equivalence classes of~$\sim$. Then $f^*(\sim)$ defines a quotient map
\begin{equation}
\label{eq:X_1 X_2}
\Pi_{f^{-1}(\CC_\bullet, \CC_- , I_\bullet , I_\circ)}  \colon (S^2,f^{-1}(A),f^{-1}(\CC))\to (X_2,\bar A_2),
\end{equation} 
induced by the collapsing data \eqref{eq:coll data:lift}. 
Furthermore, the map $f$ induces a branched covering $\bar f \colon (X_2,\bar A_2)\to (X_1,\bar A_1)$ satisfying the following commutative diagram
\begin{equation}
\label{eq:bar f}
{\begin{tikzpicture}
  \matrix (m) [matrix of math nodes,row sep=4em,column sep=7em,minimum width=2em]
  {
     (S^2,f^{-1}(A),f^{-1}(\CC)) &(S^2,A,\CC) \\
     (X_2,\bar A_2) & (X_1,\bar A_1) \\};
  \path[-stealth]
    (m-1-1) edge node [left] {$\Pi_{f^{-1}(\CC_\bullet, \CC_- , I_\bullet , I_\circ)} $} (m-2-1)
            edge  node [above] {$f$} (m-1-2)
    (m-2-1.east|-m-2-2) edge node [below] {$\bar f$}
          (m-2-2)
    (m-1-2) edge node [right] {$\Pi_{ \CC_\bullet, \CC_- , I_\bullet , I_\circ }$} (m-2-2);
\end{tikzpicture}}.
\end{equation}

The cactoid map $\bar f$ is unique in the following sense: the domain and target cactoids are unique in the sense of Lemma~\ref{lem:uniq of X} and an isotopy for $f$ induces an isotopy for $\bar f$.

We note that the map  $\bar f\colon  X_2\setminus \bar A_2 \to X_1\setminus \bar A_1$ is a covering, but it may have different degrees on different components of $X_2\setminus A_2$. Moreover, for every marked small sphere $(\bar S_z,\bar A_{z,2})$ in $(X_2,\bar A_2)$, we have a branched covering \[\bar f\mid (\bar S_z,\bar A_{z,2})\colon (\bar S_z,\bar A_{z,2})\to (\bar S_{f(z)},\bar A_{f(z),1}),\]
where $ (\bar S_{f(z)},\bar A_{f(z),1})$ is a marked small sphere of $(X_1,\bar A_1)$.

\subsubsection{Monotone maps between cactoids} 
\label{sss:MonotMapsCact}
 Consider now a Thurston map $f\colon (S^2,A,\CC)\selfmap$ and view it as a correspondence
\[f, i \colon (S^2,f^{-1}(A),f^{-1}(\CC)) \rightrightarrows (S^2,A,\CC)\] as in \eqref{eq:corr:f_i}, i.e., $f$ is a covering map (the same as the original map) and $i$ is a forgetful monotone map  (see Sections~\ref{ss:Not Conv} and \ref{ss: dec and amal} for the conventions). Below we will fix $f$ and isotope $i$ to a forgetful monotone map $\iota$ so that the new correspondence $f,\iota$ projects to a cactoid correspondence (Lemma~\ref{lem:bar iota}).

Let $\CC_\text{np}$ be the invariant submulticurve of $\CC$ generated by all its bicycles and all its non-principal unicycles. Instead of~\eqref{eq:gamma in Gamma}, we will from now on require
\begin{equation}
   \label{eq:gamma in Gamma:2}  \gamma\subsetneq \Gamma_\gamma\sp\sp\sp\text{ if and only if }\sp \sp \gamma\in \CC_\text{np}. 
\end{equation}
In this case, we say that we have \emph{dynamical} collapsing data \eqref{eq:coll data}.



 Recall that the index set $I$ parametrizes the small spheres of $(S^2,A,\CC)$. Following the notation in Section \ref{ss: dec and amal}, $f\circ i^*\colon I\selfmap$ describes the dynamics of small spheres of $(S^2,A,\CC)$. Let  $\Pi_2\coloneq\Pi_{f^{-1}(\CC_\bullet, \CC_- , I_\bullet , I_\circ)} $ and $\Pi_1\coloneq\Pi_{ \CC_\bullet, \CC_- , I_\bullet , I_\circ }$ be the quotient maps as in ~\eqref{eq:bar f}. Recall also that $I_\circ, \CC_-$ and $f^{-1}(I_\circ), f^{-1}(\CC_-)$ parametrize the small spheres and segments of the finite cactoids $X_1$ and $X_2$, respectively.

Let $\Gamma_\CC$ be the set of all curves $\gamma_t$ in $\Gamma_\gamma$,  $\gamma\in \CC$. For each $\gamma\in \CC$, let $\gamma^+,\gamma^-$ be the two curves in $f^{-1}(\Gamma_\CC)= \{f^{-1}(\gamma)\colon \gamma \in \Gamma_\CC\}$ such that the closed (possibly degenerate) annulus between $\gamma^+,\gamma^-$ contains all curves in $f^{-1}(\Gamma_\CC)$ that are isotopic to $\gamma$. Set $\bar A_2^{(1)}\coloneq \Pi_2(A) \cup \bigcup_{\gamma\in \CC} \Pi_2(\gamma^+ \cup \gamma^-).$






  
\begin{lemma}
\label{lem:bar iota}
Suppose we are given a dynamical collapsing data \eqref{eq:coll data} that satisfies 
\begin{equation}
\label{lem:bar iota:assump}
\CC_-\subset  f^{-1}(\CC_-) \sp\sp\text{ and }\sp\sp f\circ i^* (I_\circ)\subset I_\circ.
\end{equation}
Then the forgetful map $i\colon (S^2,f^{-1}(A),f^{-1}(\CC)) \to (S^2,A,\CC)$ is homotopic rel.\ $A$ to a forgetful monotone map $\iota\colon (S^2,f^{-1}(A))\to (S^2, A)$ so that the following holds:
\begin{enumerate}[label=\text{(\roman*)},font=\normalfont,leftmargin=*]

\item\label{lem:bar iota:assump:cond1}  $\iota$ projects to a forgetful monotone map $\bar \iota\colon (X_2, \bar A_2)\to (X_1, \bar A_1)$, that is, the following diagram commutes:

\begin{equation}
\label{eq:bar iota}
{\begin{tikzpicture}
  \matrix (m) [matrix of math nodes,row sep=4em,column sep=7em,minimum width=2em]
  {
     (S^2,f^{-1}(A),f^{-1}(\CC)) &(S^2,A,\CC) \\
     (X_2,\bar A_2) & (X_1,\bar A_1) \\};
  \path[-stealth]
    (m-1-1) edge node [left] {$\Pi_2 $} (m-2-1)
            edge  node [above] {$\iota $} (m-1-2)
    (m-2-1.east|-m-2-2) edge node [below] {$\bar \iota$}
          (m-2-2)
    (m-1-2) edge node [right] {$\Pi_1$} (m-2-2);
\end{tikzpicture}}
\end{equation}

\item\label{lem:bar iota:assump:cond2} for every small sphere $\bar S_{z,2}$ of $X_2$, $\bar \iota|\bar S_{z,2}$ is a homeomorphism if $\bar S_{z,2}$ is in $i^*(I_\circ)$; otherwise, it is constant.

\item\label{lem:bar iota:assump:cond3} for every small segment $T_{\gamma}$ of $X_2$ associated with $\gamma\in f^{-1}(\CC_-)$, $\bar \iota|T_{\gamma}$ is a homeomorphism if $\gamma$ is isotopic to a curve in $\CC_-$; otherwise, $\bar \iota|T_{\gamma}$ is constant.
\end{enumerate}

Moreover, the map $\bar \iota$ is unique up to isotopy rel.\ $\bar A_2^{(1)}$ (among maps satisfying the desired conditions).
\end{lemma}
%

\begin{proof}
The combinatorics uniquely determines the small spheres and segments of $X_2$ where $\bar \iota$ is a homeomorphism. This allows to define $\bar \iota$ and $\iota$ is then a lift of $\bar\iota$. 

More precisely, every small sphere $ \bar S_{z,1}$ of $X_1$ is identified with a unique component $S_{z,1}$ of $S^2\setminus \CC$.
Let $S_{i^*(z),2}$ be the unique component of $S^2\setminus f^{-1}(\CC)$ such that $S_{i^*(z),2}$ is homotopic to $S_{z,1}$ rel.\ $A$, see Section \ref{ss: dec and amal}.
Since $f\circ i^* (I_\circ)\subset I_\circ$, the component $S_{i^*(z),2}$ is identified with a unique small sphere $\bar S_{z,2}$ of $X_2$. Up to isotopy, this uniquely specifies a homeomorphism 
\begin{equation}
\label{eq:prf:lem:bar iota}
\bar \iota \colon (\bar S_{z,2},\bar A^{(1)}_2\cap \bar S_{z,2}) \to (\bar S_{z,1},\bar A_1\cap \bar S_{z,1}).
\end{equation}
Let $X'_2\subset X_2$ be the union of all small spheres $\bar S_{z,2}$ arising in~\eqref{eq:prf:lem:bar iota}. It follows from $\CC_-\subset  f^{-1}(\CC_-)$ that for every segment $T$ of $X_1$ there is at least one segment in $X_2$ with image in $T$. Therefore, $\bar \iota \colon X'_2\to X_1$ extends to a required map $\bar \iota \colon X_2\to X_1$.

Since non-trivial fibers of $\Pi_2=\Pi_{f^{-1}(\CC_\bullet, \CC_- , I_\bullet , I_\circ)} $ and $\Pi_1=\Pi_{ \CC_\bullet, \CC_- , I_\bullet , I_\circ }$
are closed surfaces with Jordan boundaries, $\bar \iota$ lifts to a required monotone map $\iota$. 
\end{proof}

Since $\bar \iota \circ  \Pi_2=\Pi_1\circ \iota $, and $\iota$ is homotopic to the identity rel.\ $A$, we have:

\begin{corollary}
\label{cor:Isot:Pi_1 Pi_0}
The monotone maps $\bar \iota\circ \Pi_2, \Pi_1\colon (S^2,A)\to (X_1,\bar A_1)$ are homotopic rel.\ $A$.
\end{corollary}


\subsubsection{Iterating correspondences}\label{sss: iter_corres}
For a Thurston map $f\colon (S^2,A,\CC)\selfmap$ consider the backward iteration:
\begin{equation}
\label{eq:f:back iter}
 \dots\overset{f}{\longrightarrow}(S^2,f^{-2}(A),f^{-2}(\CC))\overset{f}{\longrightarrow}(S^2,f^{-1}(A),f^{-1}(\CC))\overset{f}{\longrightarrow} (S^2,A,\CC).
\end{equation}
Pulling back the forgetful map $\iota=\iota_1$ from Lemma~\ref{lem:bar iota}, we obtain the maps
\begin{equation}
\label{eq:iota:back iter}
 \dots\overset{\iota_3}{\longrightarrow}(S^2,f^{-2}(A),f^{-2}(\CC))\overset{\iota_2}{\longrightarrow}(S^2,f^{-1}(A),f^{-1}(\CC))\overset{\iota_1}{\longrightarrow} (S^2,A,\CC)
\end{equation}
with $f\circ \iota_{n+1}=\iota_n\circ f$ such that the pair $f,\iota_n$ induces via the monotone maps
\[ \Pi_n \colon (S^2,f^{-n+1}(A), f^{-n+1}(\CC)) \to (X_n,\bar A_n)\]
(iterated lifts of $\Pi_1$) the cactoid correspondence $\bar f,\bar \iota \colon ( X_{n+1},\bar A_{n+1})\rightrightarrows ( X_{n},\bar A_{n})$. We obtain the inverse system:
\begin{equation}
\label{eq:f iota:back iter}
\bar f,\bar \iota\colon  \dots \rightrightarrows (X_3,\bar A_3)\rightrightarrows (X_2,\bar A_2)\rightrightarrows (X_1,\bar A_1)
\end{equation}
and we denote by $\bar f\colon X_\infty\selfmap$ the inverse limit of~\eqref{eq:f iota:back iter} with respect to $\bar \iota$.

\subsection{Expanding cactoid correspondences}  We begin by extending the notion of topological expansion (Section \ref{subsec:top_exp}) to correspondences $\bar f,\bar \iota\colon (X_2,\bar A_2)\rightrightarrows (X_1,\bar A_1)$. 

Let $\UU=(U_j)_{j\in J}$ be a finite cover of $(X_1,\bar A_1)$ by connected sets. We denote by $\bar f^*(\UU)$ the finite cover of $(X_2,\bar A_2)$ consisting of connected components of $\bar f^{-1}(U_j)$ over all $j\in J$. Then \[\bar \iota\circ \bar f^*(\UU)\coloneq \{\bar \iota(U)\mid U\in \bar f^*(\UU)\}\]
is a cover of $X_1$.

We say that the correspondence $\bar f,\bar \iota\colon (X_2,\bar A_2)\rightrightarrows (X_1,\bar A_1)$ is \emph{topologically expanding} if there is an open cover $\UU=(U_s)_{s\in S}$ of $(X_1,\bar A_1)$ by connected open sets such that the maximal diameter of components of $\UU_n\coloneq (\bar \iota \circ \bar f^*)^n(\UU)$ tends to $0$ as $n\to \infty$. Note that sets in $\UU_n$ need not be open.

We will now specify the discussion to the case when $f\colon(S^2,A,\CC)\selfmap$ is a formal amalgam of a B\"ottcher expanding map $f_{\text{BE}}\colon(S^2,A,\CC)\selfmap$, see Lemma~\ref{lem:FormAmalg}. Let us recall that $f$ is the gluing of the map~\eqref{eq:sph map} on a union $\widetilde S_\blowup$ of finitely many spheres and the annular map~\eqref{eq:ann map}. Moreover, connected components of $\widetilde S_\blowup$ and $\AA$ form a backward-invariant partition for $f$. Below we will argue that, up to isotopy of $\iota$, all relevant maps respect the partition by $\widetilde S_\blowup,\AA$. Then we will construct semi-conjugacies from $f$ and $f_\text{BE}$ towards the limiting cactoid maps and describe the fibers of the semi-conjugacies.  


Condition~\eqref{eq:gamma in Gamma:2} allows us to assume that the annuli $\Gamma_\gamma$ in Steps~\ref{cond:collapse_A},~\ref{cond:collapse_B} coincide with the annuli in $\AA$. Therefore,
\[\Pi_1=\Pi_{\CC_\bullet, \CC_- , I_\bullet,I_\circ}\colon  (S^2,A,\CC) \to (X_1,\bar A_1) \sp\sp\sp\sp  \big(\text{see~\eqref{eq:quot map}}\big)\] 
respects the partition by $\widetilde S_\blowup,\AA$: for every sphere $S_i$ in $\widetilde S_\blowup$ (i.e., for every connected component of $\widetilde S_\blowup$) its image $\Pi_1(S_i)$ is either a singleton or a small sphere of $X_1$ and $\Pi_1(\AA)\subset \bar A_1\cup \{\text{segments of $X_1$}\}$.

Since the partition by $\widetilde S_\blowup,\AA$ is backward-invariant, $\Pi_2$ is the lift of $\Pi_1$, and $\bar \iota$ respects the partition of $X_2$ (Conditions~\ref{lem:bar iota:assump:cond2} and~\ref{lem:bar iota:assump:cond3}), we obtain the following piece-wise conditions:
 \begin{enumerate}[label=\text{(PW\arabic*)},font=\normalfont,leftmargin=*]
 \item\label{cond:PW:1}  for every small sphere $S_i$ in $\widetilde S_\blowup$ its image $\bar \iota\circ \Pi_2(S_i)$ is either a singleton or a small sphere of $X_1$;
 \item\label{cond:PW:2}  $\bar \iota\circ \Pi_2(\AA)\subset \bar A_1\cup \{\text{segments of $X_1$}\}$; and hence
 \item\label{cond:PW:3}  $\iota(\widetilde S_\blowup)\subset \widetilde S_\blowup$ and $\iota(\AA)\subset \AA$.
 \end{enumerate}

\begin{lemma}
\label{lem:bar iota:exp} For a formal amalgam $f\colon(S^2,A,\CC)\selfmap$ as above and under Assumption~\eqref{lem:bar iota:assump} of Lemma~\ref{lem:bar iota}, assume in addition that 
\begin{enumerate}[label=\text{(\Alph*)},font=\normalfont,leftmargin=*]
\item all small maps of $(S^2,A,\CC)$ parametrized by $ f\circ \iota^*\colon I_\circ\selfmap$ are Sierpi\'{n}ski;
\item $\CC_-$ is generated by its bicycles: if $\CC^\bi_-$ is the union of all bicycles in $\CC_-$, then $\CC_-\subset f^{-n}(\CC^\bi_-)$ for $n\gg 1$.
\end{enumerate}
Let $\bar f,\bar \iota\colon (X_2, \bar A_2) \rightrightarrows (X_1, \bar A_1)$ be the correspondence between finite cactoids constructed in Section \ref{ss:bar f}.
Then, by isotoping $\iota$, we can assume that this correspondence is topologically expanding and $\iota,\Pi_2,\bar\iota$ still satisfy Conditions~\ref{lem:bar iota:assump:cond1},\ref{lem:bar iota:assump:cond2},\ref{lem:bar iota:assump:cond3} from Lemma \ref{lem:bar iota} and piece-wise Conditions~\ref{cond:PW:1}~\ref{cond:PW:2},~\ref{cond:PW:3}.  
\end{lemma}

Observe that Assumptions~\eqref{lem:bar iota:assump} in Lemma~\ref{lem:bar iota} imply that every periodic cycle of $ f\circ \iota^*\colon I\selfmap$ is either in $I_\circ$ or in $I_\bullet$, so Condition (A) makes sense.

\begin{proof}
We have already shown how to satisfy Conditions~\ref{lem:bar iota:assump:cond1},\ref{lem:bar iota:assump:cond2},\ref{lem:bar iota:assump:cond3},~\ref{cond:PW:1}~\ref{cond:PW:2}, \ref{cond:PW:3}. It remains to isotope $\iota$ rel.\ $\partial \widetilde S_\blowup$ so that the induced correspondence is topologically expanding. 

By \cite{Meyer_Sierp, BD_Exp},  collapsing all Fatou components in a Sierpi\'{n}ski map  $g\colon(S^2, B)\selfmap$ results in a totally  expanding map $\bar g\colon(S^2, \bar B)$.  Assumption (A) implies that we can isotope $\bar \iota$ within small spheres so that $\bar f \circ \bar \iota^*$ is expanding on small spheres in $(X_1,\bar A_1)$. This defines isotopy of $\iota$ within $\widetilde S_\blowup$.

Let $\gamma$ be a curve in a primitive bicycle in $\CC_-$ and denote by $T_\gamma$ the corresponding segment in the cactoid $(X_1, \bar A_1)$.  
Assumptions of Lemma~\ref{lem:bar iota} imply that we can modify $\bar \iota$ by an isotopy so that $\bar f \circ \bar \iota^*$ is piecewise linear on each segment in the cactoid.  Assumption (B) implies that $\bar f \circ \bar \iota^*$ is expanding on each segment $T_\gamma$.  We then lift this isotopy of $\bar \iota$ to an isotopy of $\iota$ on $\AA$ rel $\partial \AA$.

Since $\bar f \circ \bar \iota^*$ is expanding on small spheres and segments, the pair $\bar f,\bar \iota$ defines a topologically expanding correspondence. 

\end{proof}



\subsubsection{A semi-conjugacy from $f\colon S^2\selfmap$ to $\bar f\colon X_\infty\selfmap$}
Below, we follow the notation from Section \ref{sss: iter_corres}.

\begin{lemma}[Pullback argument]
\label{lem:PullBackArg} Under the assumptions of Lemma~\ref{lem:bar iota:exp}, the limits
\[ \rho_m\coloneq \lim_{n\to \infty} \bar\iota^n\circ \Pi_{n+m}\colon S^2\to X_m\]
exist as monotone maps and they semi-conjugate $f$ to $(\bar f, \bar \iota)$:
\[ \rho_m \circ f= \bar f\circ \rho_{m+1}\sp\sp\text{ and }\sp\sp \sp \rho_m=\bar \iota\circ \rho_{m+1}.\] 

Moreover, $\rho_1$ is piece-wise  rel.\ $\widetilde S_\blowup, \AA$: it maps spheres of $\widetilde S_\blowup$ onto spheres or singletons of $X_1$ and we have $\rho_1(\AA)\subset \bar A_1\cup \{\text{segments of $X_1$}\}$.
\end{lemma}

\begin{proof}
The existence of $\rho_m$ follows from the Pullback Argument, see Section \ref{sss:PullbackArg:S^2}. Namely, by Corollary~\ref{cor:Isot:Pi_1 Pi_0}, we have a homotopy $h_1$ rel.\ $A$ between $\bar\iota\circ \Pi_2$ and $\Pi_1$. Moreover, we can assume that $h_1$ respects the partition by $\widetilde S_\blowup, \AA$ because of Conditions~\ref{lem:bar iota:assump:cond1},\ref{lem:bar iota:assump:cond2},\ref{lem:bar iota:assump:cond3},~\ref{cond:PW:1}~\ref{cond:PW:2}, \ref{cond:PW:3}. Since $f\colon  (S^2,f^{-n-1}(\AA),f^{-n-1}(\CC))\to (S^2,f^{-n}(\AA),f^{-n}(\CC))$ is a covering map, we can lift $h_1$ into a homotopy $h_n$ between $\bar\iota\circ \Pi_{n+1}$ and $\Pi_n$. By the expansion of $\bar f, \bar \iota$, the tracks of $h_{n}$ decrease exponentially fast; i.e., $\rho_m$ is a well-defined monotone map as a limit of monotone maps. Clearly, $\rho_1$ is piece-wise rel.\ $\widetilde S_\blowup, \AA$ as a limit of piece-wise maps.
\end{proof}

As a corollary, we obtain that $\rho_n$ induce a semi-conjugacy $\rho_\infty\coloneq \lim_n \rho_n \colon S^2\to X_\infty$ from $f\colon S^2\selfmap$ to $\bar f\colon X_\infty\selfmap$:

\begin{equation}
\label{eq:dfn:cact map}
{\begin{tikzpicture}
  \matrix (m) [matrix of math nodes,row sep=4em,column sep=7em,minimum width=2em]
  {
     S^2 &S^2 \\
     X_\infty & X_\infty \\};
  \path[-stealth]
    (m-1-1) edge node [left] {$\rho_\infty $} (m-2-1)
            edge  node [above] {$f $} (m-1-2)
    (m-2-1.east|-m-2-2) edge node [above] {$\bar f$}
          (m-2-2)
    (m-1-2) edge node [right] {$\rho_\infty$} (m-2-2);
\end{tikzpicture}}
\end{equation}






%


Recall from Section~\ref{ss:FA} that $\filled_{\widetilde S}$ and $\Jul_{\AA}$ denote the non-escaping sets of $f|\widetilde S_\blowup$ and $f|\AA$, respectively. Moreover, $\filled_{\widetilde S,i}$ denotes the component of $\filled_{\widetilde S}$ on the sphere indexed by $i$. And $\Jul_{\AA_\Sigma}$ denotes the set of points in $\Jul_\AA$ with orbits in $\AA_\Sigma$, where $\Sigma$ is a strongly connected component of $\CC$.
\begin{lemma}\label{lem: fiber_descr}
Under the assumptions of Lemma~\ref{lem:bar iota:exp}, we have
\begin{enumerate}[label=\text{(\Roman*)},font=\normalfont,leftmargin=*]
\item\label{lem: fiber_descr:1} for every $i\in I_\circ$, the map $\rho_\infty \mid \filled_{\widetilde S, i}$ collapses exactly all the Fatou components of $\filled_{\widetilde S, i}$ into a sphere of $X_1$ (i.e.,~there are no extra identifications);
\item\label{lem: fiber_descr:2} for every bicycle $\Sigma\subset \CC_-$, the map $\rho_\infty\mid \Jul_{\AA_\Sigma}$ collapses every circle into a point so that buried circles in $\Jul_{\AA_\Sigma}$ are collapsed into different points.
\end{enumerate}
\end{lemma}

\begin{proof} Recall from Lemma~\ref{lem:PullBackArg} that $\rho_1$ is piece-wise rel.\ $\widetilde S_\blowup, \AA$. Since $\rho_{1}\mid \filled_{\widetilde S,i}$ is obtained by running the pullback argument on $\Pi_1\mid S_i$, where $S_i$ is the component of $\widetilde S_\blowup$ indexed by $i$, the map $\rho_1$ collapses exactly the Fatou components of $\filled_{\widetilde S,i}$ for $i\in I_\circ$ -- this is \ref{lem: fiber_descr:1}. Similarly, since $\rho_1\mid \Jul_{\AA_\Sigma}$ is obtained by running the pullback argument on $\Pi_1\mid \AA_\Sigma$, it follows from the properties of the cactoid correspondence on segments (see~\ref{lem:bar iota:assump:cond3}) that the limiting map~$\rho_1$ satisfies \ref{lem: fiber_descr:2}. 
\end{proof}

\subsubsection{A semi-conjugacy from $f_{\text{BE}}\colon S^2\selfmap$ to $\bar f\colon X_\infty\selfmap$} \label{ss:fibers_in_quotient} Let $\tau\colon S^2\to S^2$ be a continuous monotone map providing a semi-conjugacy from $f$ to $f_{\text{BE}}$, see Section~\ref{sss:PullbackArg:S^2}. By Proposition~\ref{prop:hom equiv}, equivalence classes $\tau^{-1}(z)$ consist of homotopy equivalent points. Since $\bar f\colon X_\infty \to X_\infty$ is expanding, $\rho_\infty$ is constant on homotopy equivalent points, and thus we have the induced semi-conjugacy from $f_{\text{BE}}$ to $\bar f$:
\begin{equation}
\label{ed:dfn:pi:BE} \pi_{\bar f} \coloneq \rho_\infty\circ  \tau^{-1} \colon S^2\to X_\infty ,\sp\sp\sp \pi_{\bar f} \circ f_{\text{BE}}= \bar f \circ \pi_{\bar f}.
\end{equation}
We also have the induced semi-conjugacies on cactoid correspondences:
\[\pi_{\bar f, m}\coloneq  [X_\infty \to X_m]\circ \pi_{\bar f},\sp\sp  \pi_{\bar f,m} \circ f_{\text{BE}}= \bar f \circ \pi_{\bar f, m+1}, \sp\sp  \pi_{\bar f,m}= \iota \circ \pi_{\bar f,m+1}.\]

\section{Proofs of the main results}

\label{sec: cannonicity}




We are now ready to prove the main results of this paper stated in the Introduction. We start with Theorem \ref{thm: max exp quotient}.

Let $f\colon (S^2,A)\selfmap$ be a B\"ottcher expanding map with $\Fat(f)\not=\emptyset$, and $\CC_\cro$ be the crochet multicurve constructed by the Crochet Algorithm (see Section \ref{ss:iter_construction}). We will assume that $I$ parametrizes the small spheres of $(S^2,A,\CC_\cro)$ and $f\colon I \selfmap$ provides the dynamics on the small spheres. Set
\begin{itemize}
\item $\CC_-$ to be the invariant  sub-multicurve of $\CC_\cro$ generated by all its bicycles;
\item $\CC_\bullet\coloneq \CC_\cro\setminus \CC_-$;
\item $I_\circ$ to be the full orbit of small spheres corresponding to Sierpi\'{n}ski small maps; 
\item  $I_\bullet$ to be the full orbit of small spheres corresponding to small crochet maps.
\end{itemize} 
By Section \ref{ss:fibers_in_quotient}, the collapsing data $\CC_\cro = \CC_\bullet\sqcup \CC_-$ and $I=I_\bullet\sqcup I_\circ$ induces 
a totally topologically expanding map $\bar f\colon X_f\selfmap $ on a cactoid $X_f (= X_\infty)$ together with the semi-conjugacy $\pi_{\bar f}$ from $f$ to $\bar f$. 
Recall also that $\sim_{\Fat(f)}$ denotes the smallest closed equivalence relation on $S^2$ generated by identifying all points in every Fatou component of $f$. 

\smallskip

\begin{theorem}
\label{thm:MEQ}
Let $f\colon (S^2,A)\selfmap$ be a B\"ottcher expanding  map with $\Fat(f)\not=\emptyset$, and $\CC_\cro$ be the crochet multicurve constructed by the Crochet Algorithm (Corollary~\ref{cor: cro multicurve exists}). Suppose that $\bar f\colon X_f\selfmap$ is the induced cactoid map 
together with the semi-conjugacy $\pi_{\bar f} \colon S^2\to X_f$ as above. Then 
\begin{equation}
\label{eq:thm:MEQ}
\pi_{\bar f}  (x)=\pi_{\bar f} (y)\sp\sp\text{ iff }\sp x\sim_{\Fat(f)} y; \sp\sp \text{ i.e., }\sp\sp X_f\cong S^2/\sim_{\Fat(f)}.
\end{equation}
Moreover, $\pi_{\bar f} $ is the maximal expanding quotient: every  semi-conjugacy $\pi_g\colon S^2\to Y$ from $f$ to a totally topologically expanding map $g\colon Y\selfmap$ factorizes through $\pi_{\bar f} :$

\begin{center}
\begin{tikzcd}[row sep=tiny]
                         & S^2\arrow{dl}{\pi_{\overline f}}  \arrow{dd}{\pi_{g}}\arrow[loop right]{l}{f} \\
  X_f \arrow[loop left]{l}{\overline{f}}   \arrow{dr}&              \\
  &Y \arrow[loop right]{l}{g}
\end{tikzcd}
\end{center} 
\end{theorem}


Clearly, Theorem \ref{thm:MEQ} establishes Theorem \ref{thm: max exp quotient} from the Introduction.

\begin{proof} Let $\sim_{\bar f}$ be the equivalence relation on $S^2$ induced by $\pi_{\bar f}$.  We need to show that $\sim_{\bar f} \, = \, \sim_{\Fat(f)}$.

\begin{claim}
If $\pi_g\colon S^2\to Y$ is a semi-conjugacy from $f$ to a totally topologically expanding map $g\colon Y\selfmap$ then $\pi_g$ must collapse each Fatou component of $f$.  In particular, $\sim_{\bar f}\,  \supset \,\sim_{\Fat(f)}$.
\end{claim}
\begin{proof}
Indeed, it is sufficient to show this for each periodic Fatou component $F$.  Let $\alpha$ be a periodic internal ray in $F$, and suppose that $\pi_g(\alpha)$ has at least two points.  Let $U_0$ be a (small) open set that covers $\pi_g(c_F)$ where $c_F$ is the center of $F$.  Let $\widetilde U_0$ the component of $\pi_g^{-1}(U_0)$ that covers $c_F$, and $\widetilde U_n$ be the component of $f^{-nk}(\widetilde U_0)$ that covers $c_F$ (here $k=\per(c_F)$ is the period of $c_F$).  Note that $\bigcup_n \widetilde U_n$ covers $\operatorname{int}(\alpha)$,  but the diameter of $\pi_g( \widetilde U_n)$ tends to $0$, since $g$ is expanding.  Consequently, $\pi_g(\alpha)$ is a singleton, and $\pi_g$ must collapse (the closure of) $F$ to a point.  
\end{proof}

Let us show the other inclusion: $\sim_{\bar f}\,  \subset \,\sim_{\Fat(f)}$. Let $f_\text{FA}\colon (S^2,A, \CC_\cro)\selfmap$ be the formal amalgam for $f\colon (S^2,A, \CC_\cro)\selfmap$ used in the construction of $\pi_{\bar f}$, see Section~\ref{ss:fibers_in_quotient} and note that the maps $f_\text{FA},f$ were denoted there by $f, f_\text{BE}$, respectively. We denote by $f_\text{FA}\colon (S^2,f^{-n}(A), f^{-n}(\CC_\cro))\selfmap$ the iteration of $f_\text{FA}$ and by $\filled_{\widetilde S^n}, \Jul_{\AA^n}$ the associated sphere and annuli non-escaping sets, see Sections~\ref{sss:FA:n} and~\ref{sss:glunig of S^n A^n}. We also denote by $\rho=\rho_\infty$ the semi-conjugacy from $f_\text{FA}$ to $\bar f$ and by $\tau$ the semi-conjugacy from $f_\text{FA}$ to $f$ so that $\pi_{\bar f}=\rho\circ \tau^{-1}$.

First, we claim that $\sim_{\Fat(f)}$ collapses crochet components in $\KC_{\widetilde S^n}$ to points, as well as connected components of $ \Jul_{\AA^{n}}$. Since $\CC_\cro$ is generated by the boundaries of crochet components, it is sufficient to show the claim about the crochet components.

Let us apply the Crochet Algorithm to $f$.  Suppose $\CC^k$ is the invariant multicurve in $(S^2, A)$ obtained after the $k$-th iteration of the recursive Step 3 of the algorithm. Let $S^1_i$ be a small crochet sphere wrt.\  $\CC^1$ containing a maximal cluster $K^1_i$. Note that $K^1_i$ may be recognized as the set of points in the crochet sphere $S^1_i$ that do not escape to the Fatou components corresponding to the curves in $\CC^1$.  By definition of $\sim_{\Fat(f)}$,  $K^1_i/\sim_{\Fat(f)}$
is a singleton, and thus $\Jul(S^1_i)/\sim_{\Fat(f)}$ is a singleton as well, where $\Jul(S^1_i)$ is the small Julia set associated with $S^1_i$.
The same is true for every component of $f^{-n}(K^1_i)$, $n\geq 1$.  


Consider now the multicurve $\CC^2\supset \CC^1$ obtained after the second iteration of the algorithm.  Suppose $S^2_i$ is a small crochet sphere wrt.\ $\CC^2$ that does not appear after the first iterate.   Then $S^2_i \subset S^1_j$ for a small sphere wrt.\ $\CC^1$.  The sphere $S^2_i$ corresponds to a maximal cluster $K^2_i$ in $S^1_j$. Note that this cluster is constructed using Fatou components of $f$ and the Fatou components corresponding to the curves in $\CC^1$.  Note that $\sim_{\Fat(f)}$ collapses the boundary of each such Fatou component.  Consequently,  $\sim_{\Fat(f)}$ collapses the small Julia set $\Jul(S^2_i)$.  By induction,  $\sim_{\Fat(f)}$ collapses each crochet component wrt.\ $\CC^k$.

By the following Claim (combined with Lemma~\ref{lem:TotDisc}), the fibers of  $\sim_{\Fat(f)}$  have disjoint images in $X_\infty$. This completes the proof of Theorem \ref{thm:MEQ}.

\begin{claim}
For all $n\ge 0$, we have the following properties:
\begin{enumerate}[label=\text{(\arabic*)},font=\normalfont,leftmargin=*]
\item\label{main proof:cond:1} If $\beta$ is a buried curve in $\Jul_{\AA^n}$, then $\rho(\beta)$ is disjoint from $\rho (\filled_{\widetilde S^n}\cup (\Jul_{\AA^n}\setminus \beta))$.
\item\label{main proof:cond:2} Two non-neighboring disjoint components $\filled_{\widetilde S^n, i},\filled_{\widetilde S^n, j}$ of $\filled_{\widetilde S^n}$ have disjoint images $\rho(\filled_{\widetilde S^n, i}),\rho(\filled_{\widetilde S^n, j})$ unless $\filled_{\widetilde S^n, i},\filled_{\widetilde S^n, j}$ have a common crochet neighbor $\filled_{\widetilde S^n,k}$; in the exceptional case, $\rho(\filled_{\widetilde S^n, i}),\rho(\filled_{\widetilde S^n, j})$ are small spheres with a single common point  $\rho(\filled_{\widetilde S^n,k})$.

\item\label{main proof:cond:3} For a crochet component $\filled_{\widetilde S^n, i}$ of $\filled_{\widetilde S^n}$, its image $\rho(\filled_{\widetilde S^n, i})$ is disjoint from
\[ \rho\big(\Jul_{\AA^n}\setminus \{\text{neighbors of $\filled_{\widetilde S^n,i}$}\}\big).\]
\item\label{main proof:cond:4} For a Sierpi\'{n}ski  component $\filled_{\widetilde S^n, i}$ of $\filled_{\widetilde S^n}$, its image $\rho(\filled_{\widetilde S^n, i})$ is disjoint from
\[ \rho\big(\Jul_{\AA^n}\setminus \{\text{neighbors of neighbors of $\filled_{\widetilde S^n,i}$}\}\big).\]

\item\label{main proof:cond:5} The image $\displaystyle \rho \left(S^2\setminus \bigcup_{n\ge 0}  \big(\filled_{\widetilde S^n} \cup \Jul_{\AA^n}\big) \right)$ is totally disconnected and is disjoint from $\displaystyle \rho \left( \bigcup_{n\ge 0}  \tau\big(\filled_{\widetilde S^n} \cup \Jul_{\AA^n}\big) \right)$.
 \end{enumerate}
\end{claim}

\begin{proof}
Since every primitive unicycle is a neighbor to at least one Sierpi\'{n}ski sphere, Sierpi\'{n}ski spheres and bicycles are dense in the following sense: for every $n\ge 0$ iterated preimages of bicycles and Sierpi\'{n}ski spheres separate 
\begin{itemize}
\item every buried curve $\beta$ in $\Jul_{\AA^n}$ from every other curve in $\Jul_{\AA^n}$ and every component of $\filled_{\widetilde S^n}$;
\item non-neighboring disjoint components $\filled_{\widetilde S^n, i},\filled_{\widetilde S^n, j}$ of $\filled_{\widetilde S^n}$ from~\ref{main proof:cond:2} unless it is an exceptional case;
\item a crochet component $\filled_{\widetilde S^n, i}$ from $\Jul_{\AA^n}\setminus \{\text{neighbors of $\filled_{\widetilde S^n,i}$}\}$;
\item a Sierpi\'{n}ski  component $\filled_{\widetilde S^n, i}$ from $\Jul_{\AA^n}\setminus \{\text{neighbors of neighbors of $\filled_{\widetilde S^n,i}$}\}$.
\end{itemize}
Since the $\rho$-images of Sierpi\'{n}ski  components are small spheres in $X_\infty$ and the \linebreak {$\rho$-images} of bicycles contain Cantor sets in $X_\infty\setminus \{\text{small spheres of }X_\infty\}$ (Lemma~\ref{lem: fiber_descr}), we obtain the separation properties~\ref{main proof:cond:1}, \ref{main proof:cond:2}, \ref{main proof:cond:3}, \ref{main proof:cond:4}; these properties also imply~\ref{main proof:cond:5}.
\end{proof} 
 \end{proof}

\begin{rem}\label{rem: crochet_multicurve_unique} Let $\CC$ be a pre-crochet multicurve for $f$, and suppose that $\DD$ is the associated crochet multicurve (obtained by the iterative elimination of all primitive crochet unicycles from $\CC$; see Section \ref{ss: gluing_crochet}). Similarly to the proof of Theorem~\ref{thm:MEQ}, we may show that the collapsing data determined by $\DD$ generates the cactoid $S^2/\sim_{\Fat(f)}$, which implies that $\DD$ coincides with $\CC_\cro$ (up to isotopy), i.e., the crochet multicurve for $f$ is uniquely defined.  
\end{rem}

We now present several results that provide topological characterizations of Sierpi\'{n}ski maps,  crochet maps,  and \emph{Sierpi\'{n}ski-free maps} (i.e., B\"{o}ttcher expanding maps with non-empty Fatou sets and without  Sierpi\'{n}ski small maps in every decomposition).  First, we recall from the Introduction the following characterizations of Sierpi\'{n}ski maps (Proposition \ref{thm:sierpinski}).


\begin{proposition}
\label{prop:sierpinski}
Let $f\colon(S^2,A)\selfmap$ be a B\"ottcher expanding map with $\Fat(f)\not=\emptyset$. Then the following are equivalent:
\begin{enumerate}[label=\text{(\roman*)},font=\normalfont,leftmargin=*]
\item $f$ is a Sierpi\'{n}ski map, i.e., $\Jul(f)$ is homeomorphic to the standard Sierpi\'{n}ski carpet;
\item $S^2/\sim_{\Fat(f)}$ is a sphere and $\sim_{\Fat(f)}$ is trivial on $A$;
\item every connected periodic zero-entropy graph is homotopically trivial rel.  $A$. 
\end{enumerate}
\end{proposition}
\begin{proof}
 Follows easily follows from Whyburn's characterization \cite{Why}, Moore's theorem \cite{Moore,WhyBook}, and \cite[Section~4.5]{BD_Exp}.
\end{proof}

The next result characterizes crochet maps (Theorem \ref{thm:crochet_inro}).



\begin{theorem}
\label{thm:crochet}
Let $f\colon(S^2,A)\selfmap$  be a B\"ottcher expanding map with $\Fat(f)\not=\emptyset$. Then the following are equivalent:
\begin{enumerate}[label=\text{(\roman*)},font=\normalfont,leftmargin=*]
\item $f$ is a crochet map, that is,  there is a connected forward-invariant zero-entropy graph $G$ containing $A$; 
\item $S^2/\sim_{\Fat(f)}$ is a singleton;
\item every two points in $A$ may be connected by a path $\alpha$ with $\alpha\cap \Jul(f)$ being countable. 
\end{enumerate}
\end{theorem}

\begin{proof}

(i)$\Rightarrow$(ii),\ (iii) Let $G$ be a connected forward-invariant $0$-entropy graph spanning the set of marked points $A$.  Recall that $G$ is composed of periodic and preperiodic internal rays, thus (iii) follows.  For each point $a\in A^\infty$, let $F_a$ be the Fatou component of $f$ centered at $a$.  
Consider the connected set $K\coloneq G\cup \bigcup_{a\in A^\infty} F_a$.  
Note that the preimage $f^{-n}(K)$ is also connected for each $n\geq 0$. By definition of $\sim_{\Fat(f)}$,  all points in $K$, and thus in $f^{-n}(K)$,  are equivalent to each other.  Now let $V$ be any complementary component of $K$.  By expansion,  the diameter of each component of $f^{-n}(V)$ tends to $0$ as $n\to \infty$.  Since the equivalence relation $\sim_{\Fat(f)}$ is closed,  the quotient $S^2/\sim_{\Fat(f)}$ is a singleton, and (ii) follows. 


$\neg$(i)$\Rightarrow$ $\neg$(ii),$\neg$(iii) Suppose that $f$ is not crochet, that is, it does not admit a connected forward-invariant $0$-entropy graph spanning $A$.  Let us apply the Crochet Algorithm to $f$ producing a crochet multicurve $\CC_\cro$. By Lemma \ref{lem:cro-mult-decomp}, there is either at least one Sierpi\'{n}ski small map in the decomposition wrt.\ $\CC_\cro$ or at least one primitive bicycle in $\CC_\cro$. 
Consequently,  the quotient cactoid $S^2/\sim_{\Fat(f)}$ contains a small sphere or a segment,  and $\neg$(ii) follows.  
Let $\tau\colon S^2\to S^2$ be a continuous monotone map proving a semiconjugacy from a formal amalgam $f_{FA}$ induced by small maps wrt.\  $\CC_\cro$ to $f$ (Section \ref{sss:PullbackArg:S^2}). 



Suppose that the decomposition of $f$ along $\CC_\cro$ contains a Sierpi\'{n}ski small map $\widehat f\colon \widehat S \selfmap$.  
Then the Julia set $\Jul(f)$ contains the image $\tau(\Jul(\widehat f))$. In fact $\tau$ embeds $\intr(\Jul(\widehat f))$ into $\Jul(f)$ by Theorem \ref{thm:Form amalg}. Suppose $a_1,a_2\in A$ are two marked points in distinct components of $S^2\setminus  \tau(\Jul(\widehat f))$. Then any path $\alpha$ connecting $a_1$ and $a_2$ intersects $\Jul(f)$ in uncountably many points,  and $\neg$(iii) follows.

Suppose that the multicurve $\CC_\cro$ contains a primitive bicycle $\Sigma$.  Following the discussion in Section \ref{sss:glunig of S^n A^n}, let $\AA_\Sigma\coloneq  \bigcup_{\gamma\in \Sigma} \AA_\gamma$ be the union of annuli that are homotopic to curves in $\Sigma$ in the construction of the formal amalgam $f_{FA}$. The non-escaping set $\Jul_\Sigma$ of $f_{FA}\colon \AA_\Sigma\cap f^{-1}(\AA_\Sigma)\to \AA_\Sigma$ is a direct product between a Cantor set and $\mathbb{S}^1$. By Theorem \ref{thm:Form amalg}, the set $\Jul_\Sigma$ contains uncountably many buried curves that are sent by $\tau$ injectively into $\Jul(f)$. Choose now any two marked points $a_1,a_2\in A$ on different sides of a curve $\gamma\in\Sigma$. Then any path $\alpha$ connecting $a_1$ and $a_2$ must meet $\tau(\Jul_\Sigma) \subset \Jul(f)$ in uncountably many points,  thus $\neg$(iii) follows.

\end{proof}

Next we observe the following separation property for crochet maps.

\begin{lemma}
\label{lem:crochet-separating-set}
Let $f\colon(S^2,A)\selfmap$ be a crochet map. Then there exits a countable set $Z\in \Jul(f)$ that separates every two distinct points $x,y\in \Jul(f)$, that is, $x$ and $y$ lie in different components of $\Jul(f)\setminus (Z\setminus \{x,y\}) $.
\end{lemma}
\begin{proof}
 Let $G$ be a connected forward-invariant $0$-entropy graph spanning the set of marked points $A$.  As in the proof of Theorem \ref{thm:crochet}, we set $K\coloneq G\cup \bigcup_{a\in A^\infty} F_a$, where $F_a$ is the Fatou component of $f$ centered at $a\in A^\infty$.

Suppose now that $x,y$ are two distinct points in $\Jul(f)$.  Choose $n$ so that 
$$\max\{\diam(V^n):\text{$V^n$ is a component of $f^{-n}(S^2\setminus K)$}\}  \leq \operatorname{dist}(x,y)/3.$$
Set $G_n\coloneq f^{-n}(G)$, and let $N_x$ be the union of all components $W^n$ of $S^2\setminus G_n$ with $x\in \overline{W^n}$. Then $y\notin N_x$ and $(\overline{N_x}\cap G_n)\cap \Jul(f)$ is a countable set separating $x$ and $y$.

It follows that the countable set $Z\coloneq \Jul(f)\cap \bigcup_n G_n$ separates any two points in the Julia set of $f$.
\end{proof}

We also note the following property of \emph{iterated monodromy groups} of crochet maps. (We assume that the reader is familiar with the relevant terminology; see, e.g., \cite{Nekra}.)

\begin{lemma}
\label{lem:crochet-maps-imgs}
Let $f\colon(S^2,A)\selfmap$ be a crochet map. Then the iterated monodromy group of $f$ is generated by a polynomial growth automaton with respect to some groupoid basis.
\end{lemma}

\begin{proof}
Let $G$ be a connected forward-invariant $0$-entropy graph spanning the set of marked points $A$, and suppose that $E$ and $W$ denote the sets of edges and faces of $G$, respectively. First, for each face $U$ of $G$ we choose a basepoint $t_U\in U$. Then, we choose the connecting paths in the following way: each basepoint $t_U$ is connected to each preimage in $f^{-1}(\{t_U\colon U\in W\})\cap U$ by a path that does not intersect the graph $G$. Finally, we pick a groupoid basis. Let $e$ be an edge of $G$ on the boundary of two faces $U,U'$ ($U=U'$ if $e$ is on the boundary of a unique face). Choose a path $g_e$ that connects the basepoints $t_U$ and $t_{U'}$ and intersects the graph $G$ exactly once in $\intr(e)$. Let $\GC$ be the groupoid generated by $\{g_e\colon e\in E\}$ (which acts on the iterated preimages of the basepoints). Since $f|G$ has $0$-entropy, the action of $\GC$ is described by an automaton of polynomial activity growth. The statement follows. 
\end{proof}

The following result provides equivalent characterizations of Sierpi\'{n}ski-free maps (Theorem \ref{thm: sierp-free-intro}).


\begin{theorem}\label{thm:sierpinski-free}
Let $f\colon(S^2,A)\selfmap$ be a B\"ottcher expanding map with $\Fat(f)\not=\emptyset$. 
Then the following are equivalent:
\begin{enumerate}[label=\text{(\roman*)},font=\normalfont,leftmargin=*]
\item $f$ is Sierpi\'{n}ski-free, i.e., the decomposition of $f$ with respect to every invariant multicurve $\CC$ does not produce a Sierpi\'{n}ski small map;
\item none of the small maps in the decomposition of $f$ along the crochet multicurve $\CC_\cro$ is a Sierpi\'{n}ski map;
\item $S^2/{\sim_{\Fat(f)}}$ is a dendrite.
\end{enumerate}
\end{theorem}
\begin{proof}
(i)$\Rightarrow$(ii) This is immediate.

(ii)$\Rightarrow$(iii) By Theorem \ref{thm:MEQ}, the cactoid $S^2/{\sim_{\Fat(f)}}$ coincides with the quotient $X_f$ of $S^2$ under the semi-conjugacy $\pi_{\bar f}$ induced by the respective collapsing data $\CC_\cro = \CC_\bullet\sqcup \CC_-$ and $I=I_\bullet\sqcup I_\circ$. Since $I_\circ=\emptyset$ by (i), it follows that $X_f$ is the inverse limit of the dendroid cactoids $X_n$ (see Section \ref{sec: cactoid maps}), which implies (ii).

(iii)$\Rightarrow$(i) Suppose there is an invariant multicurve $\DD$ with a small Sierpi\'{n}ski map $\widehat f \colon \widehat S\selfmap$ in the associated decomposition. Let $J$ be the index set parametrizing the small spheres of $(S^2,A,\DD)$, and $J_\circ\subset J$ be the orbit of small spheres induced by the small map $\widehat f \colon  \widehat S\selfmap$. 
Set $J_\bullet \coloneq J\setminus J_\circ$, $\DD_\bullet \coloneq \DD$, and $\DD_- \coloneq \emptyset$. By the discussion in Section \ref{sec: cactoid maps}, the collapsing data $\DD= \DD_\bullet\sqcup \DD_-$ and $J=J_\bullet\sqcup J_\circ$  induces 
a totally topologically expanding map $g \colon Y\selfmap $ on a cactoid $Y$ together with the semi-conjugacy $\pi_{g}\colon S^2\to Y$ from $f$ to $g$. By Theorem \ref{thm:MEQ}, $\pi_g$ should factor through $\pi_{\bar f}\colon S^2\to X_f=S^2/_{\sim \Fat(f)}$. This is clearly not possible since $X_f$ is a dendrite while $Y$ contains small spheres. This finishes the proof.
\end{proof}

Let $\CC_\cro$ be the crochet multicurve constructed by the Crochet Algorithm. Our goal now is to provide an alternative characterization for $\CC_\cro$ given by~Theorem~\ref{thm:Sierp bicycle}. We may naturally subdivide $\CC_\cro$ into two (possibly, empty) submulticurves $\CC_\Sie$ and $\CC_\bi$ in the following way: $\CC_\Sie$ is the sub-multicurve generated by the boundaries of all small Sieripi\'{n}ski spheres of $(S^2,A,\CC_\cro)$; and $\CC_\bi\coloneq\CC_\cro\setminus \CC_\Sie$. Note that all primitive components of $\CC_\bi$ are bicycles.


\begin{lemma}\label{lem: max_sierp_bicycles}
Let $f\colon(S^2,A)\selfmap$ be a B\"ottcher expanding map with $\Fat(f)\not=\emptyset$ and $\DD$ be an invariant multicurve. Then the following are true:
\begin{enumerate}[label=\text{(\roman*)},font=\normalfont,leftmargin=*]
\item If $\widetilde S$ is a Sierpi\'{n}ski sphere wrt.\ $\DD$, then $\widetilde S \subset S$ (up to homotopy) for a Sierpi\'{n}ski small sphere $ S$ wrt.\  $\CC_\cro$. 
\item If $\Sigma\subset \DD$ is a bicycle, then $\Sigma\subset \CC_\cro$ or $\Sigma$ is inside a Sierpi\'{n}ski small sphere wrt. $\CC_\cro$ (up to homotopy).
\end{enumerate}
\end{lemma}

\begin{proof}
Let us construct $\CC_\cro$ using the Crochet Algorithm, see Section \ref{ss:iter_construction}. Then, on each step of the recursive construction of the pre-crochet multicurve, the small Sierpi\'{n}ski Julia set (of $\widetilde S$) or bicycles (of $\DD$) cannot cross the  $0$-entropy graphs constructed during this step.  Consequently, the sphere $\widetilde S$ is inside  Sierpi\'{n}ski spheres wrt.\ $\CC_\cro$. Similarly, every curve $\gamma\in \Sigma$ is either in $\CC_\cro$ or inside a Sierpi\'{n}ski small sphere wrt. $\CC_\cro$.
\end{proof}






The lemma above immediately implies Theorem \ref{thm:Sierp bicycle} from the Introduction. We also record the following easy corollary.


\begin{corollary}Let $f$ be a B\"ottcher expanding map with non-empty Fatou set.  Then each small sphere $\widehat f\colon \widehat S\selfmap$ in the decomposition of $f$ along $\CC_\Sie$ is Sierpi\'{n}ski-free.  Furthermore,  the multicurve $\widehat \CC_\bi\coloneq \CC_\bi \cap \widehat S$ is the crochet multicurve for $\widehat f$.  
\end{corollary}
\begin{proof}
Indeed,   $\widehat f\colon \widehat S\selfmap$ cannot contain a Sierpi\'{n}ski map by Lemma~\ref{lem: max_sierp_bicycles}.  Thus, $\widehat f$ is  Sierpi\'{n}ski-free and, by  Theorem \ref{thm:Sierp bicycle},  the decomposing curve for $\widehat f$ must be generated by bicycles,  that is,  $\widehat \CC_\bi\coloneq \CC_\bi \cap \widehat S$. The statement follows. 
\end{proof}

Finally, we prove the following result characterizing the crochet multicurve $\CC_\cro$ (Theorem \ref{thm:cro decomp}).

Let $f\colon(S^2,A)\selfmap$ be a B\"ottcher expanding map with an invariant multicurve $\DD$. Suppose that all maps in the decomposition of $f$ along $\DD$ are Sierpi\'{n}ski or crochet. Let $f_{FA}$ be the formal amalgam for $f\colon(S^2,A, \DD)\selfmap$ and $\tau$ be the semiconjugacy from $f_{FA}$ to $f$. Let $\KC_{\widetilde S,i}$ be the component of the non-escaping set $\KC_{\widetilde S}$ associated with a periodic small sphere $S_i$ of $(S^2,A,\DD)$. The set $\tau(\KC_{\widetilde S,i})\cap \Jul(f)$ is called a \emph{small Julia set} of the first return map to $S_i$.

\begin{theorem}
\label{thm:cro decomp can}
Let $f\colon(S^2,A)\selfmap$ be a B\"ottcher expanding map with $\Fat(f)\not=\emptyset$.  There is a unique canonical invariant multicurve $\CC_\cro$ whose small maps are Sierpi\'{n}ski and crochet maps such that for the quotient map $\pi_{\overline{f}}\colon S^2\to S^2/{\sim_{\Fat(f)}}$ the following are true:
\begin{enumerate}[label=\text{(\roman*)},font=\normalfont,leftmargin=*]
\item  small Julia sets of Sierpi\'{n}ski maps project onto spheres;
\item small Julia sets of crochet maps project to points;
\item  different crochet Julia sets project to different points in $S^2/{\sim_{\Fat(f)}}$.
\end{enumerate}
\end{theorem}  

\begin{proof}
The multicurve $\CC_\cro$ constructed by the Crochet Algorithm satisfies the required properties (by the construction of the cactoid $X_{\bar f}\cong S^2/_{\sim\Fat(f)}$). 

Conversely, let $\DD$ be any invariant multicurve such that each small in the decomposition of $f$ along $\DD$ is either Sierpi\'{n}ski or crochet and it satisfies Conditions (i), (ii), (iii) of the theorem. Condition (i) implies that $\DD_\Sie = \CC_\Sie$ (i.e., Sierpi\'{n}ski spheres are maximal). So, we reduced the statement to the Sierpi\'{n}ski-free case. Condition (ii) and (iii) imply that the crochet Julia sets of $\DD$ are within the crochet Julia set of $\CC_\cro$ and thus $\DD$ is isotopic to $\CC_\cro$.
\end{proof}

   \newcommand{\etalchar}[1]{$^{#1}$}



\begin{thebibliography}{CYMT99}

\bibitem[Bar22]{GAP_IMG}
L.~Bartholdi.
\newblock {I}{M}{G} --- {C}omputations with iterated monodromy groups.
\newblock {\em Version 0.3.2}, 2022.

\bibitem[BD17]{BD_Algo}
L.~Bartholdi and D.~Dudko.
\newblock Algorithmic aspects of branched coverings.
\newblock {\em Ann. Fac. Sci. Toulouse Math. (6)}, 26(5):1219--1296, 2017.

\bibitem[BD18]{BD_Exp}
L.~Bartholdi and D.~Dudko.
\newblock Algorithmic aspects of branched coverings {IV}/{V}. {E}xpanding maps.
\newblock {\em Trans. Amer. Math. Soc.}, 370(11):7679--7714, 2018.

\bibitem[BD21a]{BD_III}
L.~Bartholdi and D.~Dudko.
\newblock Algorithmic aspects of branched coverings {III/V}. {E}rasing maps,
  orbispaces, and the {B}irman exact sequence.
\newblock {\em Groups, Geom. Dyn.}, 15(4):1197--1265, 2021.

\bibitem[BD21b]{BD_Dec}
L.~Bartholdi and D.~Dudko.
\newblock Algorithmic aspects of branched coverings {II}/{V}: {S}phere bisets
  and decidability of {T}hurston equivalence.
\newblock {\em Invent. Math.}, 223(3):895--994, 2021.

\bibitem[BEK{\etalchar{+}}12]{MatingQuestions}
X.~Buff, A.L. Epstein, S.~Koch, D.~Daniel, K.M. Pilgrim, , M.~Rees, and Tan
  Lei.
\newblock Questions about polynomial matings.
\newblock In {\em Annales de la Facult{\'e} des sciences de Toulouse:
  Math{\'e}matiques}, volume~21, pages 1149--1176, 2012.

\bibitem[BFH92]{BFH_Class}
B.~Bielefeld, Y.~Fisher, and J.H. Hubbard.
\newblock The classification of critically preperiodic polynomials as dynamical
  systems.
\newblock {\em J. Amer. Math. Soc.}, 5(4):721--762, 1992.

\bibitem[BKN10]{AmenabilityInBoundedGroups}
L.~Bartholdi, V.A. Kaimanovich, and V.~Nekrashevych.
\newblock On amenability of automata groups.
\newblock {\em Duke Math. J.}, 154(3):575--598, 2010.

\bibitem[BM17]{THEBook}
M.~Bonk and D.~Meyer.
\newblock {\em Expanding {T}hurston maps}.
\newblock Math. Surveys and Monographs 225. Amer. Math. Soc., Providence, RI,
  2017.

\bibitem[BN06]{BarNekr_Twist}
L.~Bartholdi and V.~Nekrashevych.
\newblock Thurston equivalence of topological polynomials.
\newblock {\em Acta Math.}, 197(1):1--51, 2006.

\bibitem[CM22]{CarrascoConfDim}
M.~Carrasco and J.M. Mackay.
\newblock Conformal dimension of hyperbolic groups that split over elementary
  subgroups.
\newblock {\em Invent. Math.}, 227(2):795--854, 2022.

\bibitem[CYMT99]{CanaryMinsky}
R.~Canary, Yair Y.N.~Minsky, and E.C. Taylor.
\newblock Spectral theory, {H}ausdorff dimension and the topology of hyperbolic
  $3$-manifolds.
\newblock {\em J. Geom. Anal.}, 9(1):17--40, 1999.

\bibitem[DH84]{DH_Orsay}
A.~Douady and J.H. Hubbard.
\newblock {\em \'{E}tude dynamique des polyn\^omes complexes. {P}artie {I}},
  volume~84 of {\em Publications Math\'ematiques d'Orsay [Mathematical
  Publications of Orsay]}.
\newblock Universit\'e de Paris-Sud, D\'epartement de Math\'ematiques, Orsay,
  1984.

\bibitem[DH93]{DH_Th_char}
A.~Douady and J.H. Hubbard.
\newblock A proof of {T}hurston's topological characterization of rational
  functions.
\newblock {\em Acta Math.}, 171(2):263--297, 1993.

\bibitem[DMRS19]{ClassNewtonFixed}
K.~Drach, Y.~Mikulich, J.~R{\"u}ckert, and D.~Schleicher.
\newblock A combinatorial classification of postcritically fixed {N}ewton maps.
\newblock {\em Ergod. Theory Dyn. Syst.}, 39(11):2983--3014, 2019.

\bibitem[Dou83]{Douady_Mating}
A.~Douady.
\newblock Syst\`emes dynamiques holomorphes.
\newblock In {\em Bourbaki seminar, {V}ol. 1982/83}, volume 105 of {\em
  Ast\'erisque}, pages 39--63. Soc. Math. France, Paris, 1983.

\bibitem[DS22]{NewtonRigidity}
K.~Drach and D.~Schleicher.
\newblock Rigidity of {N}ewton dynamics.
\newblock {\em Adv. Math.}, 408:108591, 2022.

\bibitem[GHMZ18]{Meyer_Sierp}
Y.~Gao, P.~Ha\"{i}ssinsky, D.~Meyer, and J.~Zeng.
\newblock Invariant {J}ordan curves of {S}ierpi\'{n}ski carpet rational maps.
\newblock {\em Ergod. Theory Dyn. Syst.}, 38(2):583--600, 2018.

\bibitem[Hlu17]{H_Thesis}
M.~Hlushchanka.
\newblock {\em Invariant graphs, tilings, and iterated monodromy groups}.
\newblock PhD thesis, Jacobs University Bremen, 2017.

\bibitem[HP08]{HP_Dimension}
P.~Ha{\"\i}ssinsky and K.M. Pilgrim.
\newblock Thurston obstructions and {A}hlfors regular conformal dimension.
\newblock {\em J. Math. Pures Appl.}, 90(3):229--241, 2008.

\bibitem[HP09]{HP_Expanding}
P.~Ha{\"{\i}}ssinsky and K.M. Pilgrim.
\newblock Coarse expanding conformal dynamics.
\newblock {\em Ast\'erisque}, (325):viii+139 pp. (2010), 2009.

\bibitem[HS94]{HS_Spider}
J.H. Hubbard and D.~Schleicher.
\newblock The spider algorithm.
\newblock In {\em Complex dynamical systems ({C}incinnati, {OH}, 1994)},
  volume~49 of {\em Proc. Sympos. Appl. Math.}, pages 155--180. Amer. Math.
  Soc., Providence, RI, 1994.

\bibitem[IS10]{IshiiSmillie}
Y.~Ishii and J.~Smillie.
\newblock Homotopy shadowing.
\newblock {\em Amer. J. Math}, 132(4):987--1029, 2010.

\bibitem[JNdlS16]{Jusch_Nekr_Salle}
K.~Juschenko, V.~Nekrashevych, and M.~de~la Salle.
\newblock Extensions of amenable groups by recurrent groupoids.
\newblock {\em Invent. Math.}, 206(3):837--867, 2016.

\bibitem[Kam01]{KameyamaThEq}
A.~Kameyama.
\newblock The {T}hurston equivalence for postcritically finite branched
  coverings.
\newblock {\em Osaka Journal of Mathematics}, 38(3):565--610, 2001.

\bibitem[Lei92]{TanLeiMatings}
Tan Lei.
\newblock Matings of quadratic polynomials.
\newblock {\em Ergod. Theory Dyn. Syst.}, 12(3):589--620, 1992.

\bibitem[LMS22]{RussellDierk_Class}
R.~Lodge, Y.~Mikulich, and D.~Schleicher.
\newblock A classification of postcritically finite {N}ewton maps.
\newblock {\em In the {T}radition of {T}hurston {II}}, pages 421--448, 2022.

\bibitem[Mey11]{M_unmating}
D.~Meyer.
\newblock Unmating of rational maps, sufficient criteria and examples.
\newblock {\em Frontiers in Complex Dynamics: In Celebration of John Milnor’s
  80th Birthday}, pages 197--234, 2011.

\bibitem[Moo25]{Moore}
R.L. Moore.
\newblock Concerning upper semicontinuous collections of compacta.
\newblock {\em Trans. Amer. Math. Soc.}, 27:416--426, 1925.

\bibitem[Nek05]{Nekra}
V.~Nekrashevych.
\newblock {\em Self-similar groups}, volume 117 of {\em Mathematical Surveys
  and Monographs}.
\newblock American Mathematical Society, Providence, RI, 2005.

\bibitem[NPT]{NPT_Amenability}
V.~Nekrashevych, K.M. Pilgrim, and D.~Thurston.
\newblock On amenable iterated monodromy groups.
\newblock Preprint, Available at
  https://www.math.tamu.edu/{$\sim$}nekrash/Preprints/ amenableimg.pdf.

\bibitem[Par20]{ParkLevy}
I.~Park.
\newblock Levy and {T}hurston obstructions of finite subdivision rules.
\newblock {\em arXiv preprint arXiv:2012.00243}, 2020.

\bibitem[Par21]{ParkThesis}
I.~Park.
\newblock {\em Julia sets with {A}hlfors regular conformal dimension one}.
\newblock PhD thesis, Indiana University, 2021.

\bibitem[Pil03]{Pilgrim_Comb}
K.M. Pilgrim.
\newblock {\em Combinations of complex dynamical systems}, volume 1827 of {\em
  Lecture Notes in Mathematics}.
\newblock Springer-Verlag, Berlin, 2003.

\bibitem[PL98]{PT}
K.M. Pilgrim and Tan Lei.
\newblock Combining rational maps and controlling obstructions.
\newblock {\em Ergod. Theory Dyn. Syst.}, 18(1):221--245, 1998.

\bibitem[PM12]{Meyer_Mating}
C.L. Petersen and D.~Meyer.
\newblock On the notions of mating.
\newblock {\em Ann. Fac. Sci. Toulouse Math. (6)}, 21(5):839--876, 2012.

\bibitem[Poi10]{Poirier}
A.~Poirier.
\newblock Hubbard trees.
\newblock {\em Fund. Math.}, 208(3):193--248, 2010.

\bibitem[Ree92]{Rees_1}
M.~Rees.
\newblock A partial description of parameter space of rational maps of degree
  two, part {I}.
\newblock {\em Acta Math.}, 168(1-2):11--87, 1992.

\bibitem[Sel12]{SelingerPullback}
N.~Selinger.
\newblock Thurston’s pullback map on the augmented {T}eichm{\"u}ller space
  and applications.
\newblock {\em Invent. Math.}, 189(1):111--142, 2012.

\bibitem[Shi00]{Shishikura_Rees}
M.~Shishikura.
\newblock On a theorem of {M}ary {R}ees.
\newblock {\em The {M}andelbrot {S}et, {T}heme and {V}ariations, {LMS}
  {L}ecture {N}otes}, 274, 2000.

\bibitem[SY15]{SY_Decid}
N.~Selinger and M.~Yampolsky.
\newblock Constructive geometrization of {T}hurston maps and decidability of
  {T}hurston equivalence.
\newblock {\em Arnold Math. J.}, 1(4):361--402, 2015.

\bibitem[Thu20]{Dylan_Positive}
D.P. Thurston.
\newblock A positive characterization of rational maps.
\newblock {\em Ann. of Math. (2)}, 192(1):1--46, 2020.

\bibitem[Why42]{WhyBook}
G.T. Whyburn.
\newblock {\em Analytic topology}.
\newblock AMS Colloquium Publications 28. Amer. Math. Soc., Providence, RI,
  1942.

\bibitem[Why58]{Why}
G.T. Whyburn.
\newblock Topological characterization of the {S}ierpi\'nski curve.
\newblock {\em Fund. Math.}, 45:320--324, 1958.

\bibitem[You48]{MonotoneMaps}
J.W.T. Youngs.
\newblock Homeomorphic approximations to monotone mappings.
\newblock {\em Duke Math. J.}, 15(1):87--94, 1948.

\end{thebibliography}
\end{document}